\documentclass[12pt, reqno]{amsart}
\usepackage{amsmath, amsthm, amscd, amsfonts, amssymb, graphicx, color}
\usepackage{enumerate}
\usepackage{chemarrow}
\usepackage[bookmarksnumbered, colorlinks, plainpages]{hyperref}
\hypersetup{colorlinks=true,linkcolor=red, anchorcolor=green, citecolor=cyan, urlcolor=red, filecolor=magenta, pdftoolbar=true}

\newtheorem{thm}{Theorem}[section]
\newtheorem{cor}[thm]{Corollary}
\newtheorem{lem}[thm]{Lemma}
\newtheorem{prop}[thm]{Proposition}

\newtheorem{rem}[thm]{Remark}
\theoremstyle{definition}
\newtheorem{defn}[thm]{Definition}

\numberwithin{equation}{section}

\begin{document}

\setcounter{page}{1}

\title[Dilation theory and analytic model theory]{Dilation theory and analytic model theory for doubly commuting sequences of $C_{.0}$-contractions}

\author[Hui Dan, \quad  Kunyu Guo]{Hui Dan, \quad Kunyu Guo}

\address{Hui Dan, School of Mathematical Sciences, Fudan University, Shanghai 200433, P.R. China.}
\email{\textcolor[rgb]{0.00,0.00,0.84}{hdan@fudan.edu.cn}}

\address{Kunyu Guo, School of Mathematical Sciences, Fudan University, Shanghai 200433, P.R. China.}
\email{\textcolor[rgb]{0.00,0.00,0.84}{kyguo@fudan.edu.cn}}

\let\thefootnote\relax\footnote{}



%\date{.
%\newline \indent $^{*}$Corresponding author}

\begin{abstract}
Sz.-Nagy and  Foias proved that each $C_{\cdot0}$-contraction
has a dilation to a Hardy shift and thus established an
elegant  analytic functional model for contractions of class $C_{\cdot0}$ \cite{SNFBK}. This has motivated lots of further works on model theory and generalizations to commuting tuples of
$C_{\cdot0}$-contractions. In this paper, we focus on doubly commuting sequences of $C_{\cdot0}$-contractions, and establish  the dilation theory and the analytic model theory for these sequences of operators. These results are  applied to generalize the Beurling-Lax  theorem and  Jordan blocks in the multivariable  operator theory
%for doubly commuting Hardy submodules and results related with  Jordan blocks in the multi-variable case
to the operator theory in countably infinitely many variables.
\end{abstract} \maketitle

%---------------------------------------------------------------------------------------%
\section{Introduction}
\subsection{Background}
Suppose that $T$ is a contraction on a  Hilbert space $\mathcal{H}$; that is, $T$ is a bounded linear operator with $\|T\|\leq1$. The contraction  $T$ is said to be in the class $C_{\cdot0}$ if $T^{*k}\rightarrow0$ (as $k\rightarrow\infty$) in the strong operator topology. An operator $V$ on a larger Hilbert space $\mathcal{K}\supseteq\mathcal{H}$ is said to be
a dilation of $T$ if for each $n\in\mathbb{N}$,
$$T^n=P_{\mathcal{H}}V^n|_{\mathcal{H}},$$ where $\mathbb{N}=\{1,2,\cdots\}$, the set of positive integers.
The dilation theory developed by B. Sz.-Nagy and C. Foias is of great significance in operator theory.
They built functional models for contractions, which not only reveals
the structure of  these operators, but also gives a way for calculations \cite{SNFBK}.
In particular, it was shown that a
contraction of class $C_{\cdot0}$ has a dilation to a Hardy shift. To be more specific, let $\mathbb{D}$ denote the open unit disk in the complex plane, and $H^2(\mathbb{D})$ denote the Hardy space over $\mathbb{D}$. For a contraction $T$, put  $$D_T=(I-T^*T)^{\frac{1}{2}},$$
the defect operator of $T$, and
 $$\mathfrak{D}_T=\overline{D_T\mathcal{H}},$$
the defect space of $T$.
 If $T$ is a $C_{\cdot0}$-contraction, then $T$ is unitarily equivalent to the compression of the multiplication operator $M_z$, defined on the vector-valued Hardy space $H_{\mathfrak{D}_{T^*}}^2(\mathbb{D})=H^2(\mathbb{D})\otimes\mathfrak{D}_{T^*}$, on some invariant subspace  $\mathcal{J}$ of $H_{\mathfrak{D}_{T^*}}^2(\mathbb{D})$ for the backward shift $M_z^*$.

The existence of the isometric dilation of any commuting pair of contractions is proved by T. And\^{o} \cite{An}.
However, commuting $n$-tuples of  contractions  have no isometric dilations in general when $n\geq3$ \cite{Par}. Under some additional conditions, the above graceful analytic model for single contraction of class $C_{\cdot0}$ can be generalized to the situation of a commuting  finite-tuple $(T_1,\cdots,T_n)$ of $C_{\cdot0}$-contractions. That is to say, this tuple
 $(T_1,\cdots,T_n)$ has a dilation to the tuple $(M_{\zeta_1},\cdots,M_{\zeta_n})$ of coordinate multiplication operators on a $\mathcal{E}$-valued analytic function space $\mathcal{H}_{\mathcal{E}}=\mathcal{H}\otimes\mathcal{E}$ with $\mathcal{H}$ consisting of holomorphic functions over some domain in $\mathbb{C}^n$ \cite{CV,MV,Po,AEM,BNS}. The particular case that the tuple being doubly commuting is rather interesting since in this case the function space $\mathcal{H}$ is exactly the Hardy space $H^2(\mathbb{D}^n)$ over the $n$-polydisc, and the underlying space $\mathcal{E}$ is the defect space of the tuple $(T_1^*,\cdots,T_n^*)$. Moreover, we can  require this dilation to be minimal and regular \cite{BNS}. Recall that two operators $T$, $S$ are said to be doubly commuting if $TS=ST$ and $T^*S=ST^*$, and a tuple or a sequence of operators is said to be doubly commuting if any pair of operators in it are doubly commuting. For more about developments on model theory, we refer the readers to \cite{Slo,Ag,Mu,Va}.

A natural question arises: Does a commuting sequence of $C_{\cdot0}$-contractions, under the conditions analogous to those given in %\cite{CV,MV,Po,AEM,BNS}
the above-mentioned papers, have a dilation to the  tuple of coordinate multiplication operators on some vector-valued Hardy space over infinite-dimensional polydisc?
Except for the assumption ``doubly commuting", most of these conditions do not carry over to infinitely many operators well.
For this reason, we focus on doubly commuting sequences of $C_{\cdot0}$-contractions in this paper, and establish dilation theory and analytic model theory for these sequences of operators.

Another motivation for our study of such operator sequences  comes from investigations on the structure of some special submodules and quotient modules of the  Hardy module
 over the polydisc in \cite{SSW,Sar2}, which are said to be doubly commuting.
%It is also noticeable and interesting that how will our investigations on doubly commuting sequences of $C_{\cdot0}$-contractions be applied to studying the structure of some special submodules and quotient modules of the  Hardy module in infinitely many variables.

The analytic Hilbert module theory developed by R. Douglas and V. Paulsen opens a new door for the study of joint invariant subspaces of the Hardy space $H^2(\mathbb{D}^n)\ (n\in\mathbb{N})$ for the tuple $\mathbf{M}_\zeta$ of coordinate multiplication operators \cite{Dou,DP}.
Let $\mathcal{P}_n$ denote the polynomial ring in $n$-complex variables.
It is known that the Hardy space $H^2(\mathbb{D}^n)$
carries a  $\mathcal{P}_n$-Hilbert module structure,
where the module action is defined by multiplications
by polynomials, and a submodule of $H^2(\mathbb{D}^n)$ is just a $\mathbf{M}_\zeta$-joint invariant subspace.
A quotient module of $H^2(\mathbb{D}^n)$ is defined to be the  orthocomplement of some submodule with a $\mathcal{P}_n$-module structure determined by the compression of $\mathbf{M}_\zeta$ on it.
These notions are defined
analogously on  vector-valued Hardy spaces $H_{\mathcal{E}}^2(\mathbb{D}^n)$. For more details, we refer the reader to \cite{DPSY,CG}.

The famous theorem of A. Beurling states that
every nonzero submodule of $H^2(\mathbb{D})$ is of form
$\eta H^2(\mathbb{D})$ for some inner function $\eta\in H^\infty(\mathbb{D})$, where $H^\infty(\mathbb{D})$ denotes the space of bounded holomorphic functions on $\mathbb{D}$ \cite{Beu1}. P. Lax  generalized Beurling's theorem to  vector-valued Hardy spaces $H_{\mathcal{E}}^2(\mathbb{D})$ \cite{La}:
a nonzero submodule of $H_{\mathcal{E}}^2(\mathbb{D})$
takes the form $\theta H_{\mathcal{F}}^2(\mathbb{D})$
with $\theta\in H_{\mathcal{B}(\mathcal{F},\mathcal{E})}^\infty(\mathbb{D})$
being inner, where $H_{\mathcal{B}(\mathcal{F},\mathcal{E})}^\infty(\mathbb{D})$ is the space of uniformly bounded holomorphic $\mathcal{B}(\mathcal{F},\mathcal{E})$-valued function on $\mathbb{D}$. For the multi-variable situation, such Beurling-Lax type theorem fails  in general \cite{Ru}. Some efforts was made to determine when a submodule enjoys a Beurling-Lax type representation. In \cite{Man},
V. Mandrekar considered the case of the scalar-valued Hardy space over the bidisc. He obtained
a necessary and sufficient condition that the restriction of the tuple $\mathbf{M}_\zeta$ on the submodule  is doubly commuting.
This  was further generalized to the vector-valued Hardy space $H_{\mathcal{E}}^2(\mathbb{D}^n)$ for arbitrary positive integer $n$ in \cite{SSW}, and to the Hardy space over the infinite-dimensional polydisc within the language of the Hilbert space of Dirichlet series with square-summable coefficients
in \cite{O}.

%However,
%it was shown in \cite{SSW} that if the restriction of the tuple $\mathbf{M}_\zeta$ on a nonzero submodule $\mathcal{S}$ of
%$H_{\mathcal{E}}^2(\mathbb{D}^n)$ is doubly commuting, then there exists a $\mathcal{B}(\mathcal{F},\mathcal{E})$-valued inner function $\Psi$ on $\mathbb{D}^n$, such that $\mathcal{S}=\Psi H_{\mathcal{F}}^2(\mathbb{D}^n)$.
%This can be viewed as a multi-variable version of the Beurling-Lax  theorem. A. Olofsson obtained
%a Beurling type theorem for the Hardy space over the infinite-dimensional polydisc, which is stated in the language of the Hilbert space of Dirichlet series with square-summable coefficients \cite{O}.

A  proper quotient module of $H^2(\mathbb{D})$ is called a model space, and
the compression of the Hardy shift $M_z$ on a  model space is called a Jordan block. The notion of the Jordan block plays a central role in Sz.-Nagy and Foias' model theory for operators of class $C_0$: every $C_0$-operator is quasi-similar to the direct sum of Jordan blocks \cite{SNFBK}. Following former works  on the Hardy quotient module over the bidisc \cite{DY1,DY2,INS}, J. Sarkar proved that if the compression of the tuple $\mathbf{M}_\zeta$ on a nonzero quotient module $\mathcal{Q}$ of
$H^2(\mathbb{D}^n)$ is doubly commuting, then
$$\mathcal{Q}=\mathcal{J}_1\otimes\cdots\otimes\mathcal{J}_n,$$
where either $\mathcal{J}_i$ is a model space or $\mathcal{J}_i=H^2(\mathbb{D})$ ($1\leq i\leq n$) \cite{Sar2}. Thus, for any $i$ so that $\mathcal{J}_i\neq H^2(\mathbb{D})$, the compression $P_\mathcal{Q}M_{\zeta_i}|_\mathcal{Q}$ is the tensor product of
a Jordan block and an identity operator. Then the compression of the tuple $\mathbf{M}_\zeta$ on $\mathcal{Q}$ can be considered as the Jordan block in finitely many variables. We refer the readers to \cite{Yang1,Yang2,II,QY,Sar3} for related works and further discussions. Also see \cite[Remark 4.7]{DGH}.

In this paper, we
establish dilation theory and analytic model theory for
doubly commuting sequences of $C_{\cdot0}$-contractions,
and then apply them to
 generalize the Beurling-Lax  theorem for doubly commuting  submodules and the Jordan block type characterization for doubly commuting quotient module in the multi-variable case
to the Hardy module in infinitely many variables.
%\textbf{It is also interesting  how  our study on doubly commuting sequences of $C_{\cdot0}$-contractions is applied to obtain such generalizations.}
%In particular, the proof of the Beurling-Lax  theorem in the infinite-variable case we presented is completely different from that in \cite{O}.

\subsection{Statements of the main results}

To state our  main results, we need to introduce some notations and definitions.
Let $\mathbb{D}^\infty$ denote the Cartesian product $\mathbb{D}\times\mathbb{D}\times\cdots$
of  countably
infinitely many unit disks.
The Hilbert's multidisk $\mathbb{D}_2^\infty$ is defined to be
$$\mathbb{D}_2^\infty=\{\zeta=(\zeta_1,\zeta_2,\cdots)\in l^2:|\zeta_n|<1\  \mathrm{for}\ \mathrm{all}\ n\geq1\}.$$
The Hardy space $H^2(\mathbb{D}_2^\infty)$ in infinitely many variables is defined as follows:
$$H^2(\mathbb{D}_2^\infty)
=\{F=\sum_{\alpha\in\mathbb{Z}_+^{(\infty)}}c_\alpha\zeta^\alpha:
\|F\|^2=\sum_{\alpha\in\mathbb{Z}_+^{(\infty)}}|c_\alpha|^2<\infty\},$$
where $\mathbb{Z}_+^{(\infty)}$ denotes the set of finitely supported sequences of nonnegative integers, and $\zeta^\alpha$ denotes the monimial $$\zeta^\alpha=\zeta_1^{\alpha_1}\cdots\zeta_n^{\alpha_n}$$
for $\alpha=(\alpha_1,\cdots,\alpha_n,0,0,\cdots)\in\mathbb{Z}_+^{(\infty)}$.
 The space
 $H^2(\mathbb{D}_2^\infty)$ is a reproducing kernel Hilbert space over the Hilbert's multidisk $\mathbb{D}_2^\infty$ with the kernels \cite{Ni2} $$\mathbf{K}_\lambda(\zeta)=\prod_{n=1}^\infty
 \frac{1}{1-\overline{\lambda_n}\zeta_n},
 \quad\lambda=(\lambda_1,\lambda_2,\cdots)\in\mathbb{D}_2^\infty.
 $$
 This space has close connections with the study of Beurling's completeness problem and Dirichlet series \cite{Beu2,HLS,Ni2}, and is the expected function space upon which we build analytic models. The vector-valued Hardy space $H_{\mathcal{E}}^2(\mathbb{D}_2^\infty)$ in infinitely many variables is defined analogously as in the finite-variable situation. See Section 2 for the details.

If $\mathbf{T}=(T_1,T_2,\cdots)\in\mathcal{DC}$ is a doubly commuting sequence of contractions on a Hilbert space $\mathcal{H}$,
then the infinite product  $$D_{\mathbf{T}}=\mathrm{(SOT)}
\lim_{n\rightarrow\infty}D_{T_1}\cdots D_{T_n}
=\mathrm{(SOT)}
\lim_{n\rightarrow\infty}(I-T_1^*T_1)^{\frac{1}{2}}\cdots
(I-T_n^*T_n)^{\frac{1}{2}}$$ of defect operators of $\{T_n\}_{n\in\mathbb{N}}$ converges \cite[Proposition 43.1]{Con}, and is called the \textit{defect operator} of $\mathbf{T}$. Similarly, define the \textit{defect space} $\mathfrak{D}_{\mathbf{T}}$ of  $\mathbf{T}$ to be
$$\mathfrak{D}_{\mathbf{T}}=\overline{D_{\mathbf{T}}\mathcal{H}}.$$

%Throughout this paper, the tuple $(T_1,\cdots,T_n)$ of finitely many contractions will be viewed as a sequence $(T_1,\cdots,T_n,0,0,\cdots)$.
Let  $\mathcal{DC}(\mathcal{H})$ (or simply $\mathcal{DC}$ if no confusion is caused) denote the class of doubly commuting sequences of $C_{\cdot0}$-contractions on $\mathcal{H}$.
For convenience, we identify the class of doubly commuting finite-tuples of $C_{\cdot0}$-contractions with the class $\mathcal{DCF}$ of   sequences in class $\mathcal{DC}$ with  only  finitely many nonzero components. Let $\mathbf{T}^*$ denote the sequence $(T_1^*,T_2^*,\cdots)$ and
$\mathbf{T}^\alpha$  the operator $T_1^{\alpha_1}\cdots T_n^{\alpha_n}$ for $\mathbf{T}=(T_1,T_2,\cdots)$ and $\alpha=(\alpha_1,\cdots,\alpha_n,0,0,\cdots)\in\mathbb{Z}_+^{(\infty)}$ ($n\in\mathbb{N}$).
The key ingredient in the dilation theory for doubly commuting finite-tuples of $C_{\cdot0}$-contractions is the following norm identity \cite{BNS}:
\begin{equation}\label{key identity}\|x\|^2=\sum_{\alpha\in\mathbb{Z}_+^n}\|D_{\mathbf{T}^*}\mathbf{T}^{*\alpha}x\|^2,\quad x\in \mathcal{H},\end{equation} where $n\in\mathbb{N}$ and $\mathbf{T}=(T_1,\cdots,T_n)\in\mathcal{DCF}$.
It is natural to expect that such an identity  holds for sequences in class $\mathcal{DC}$  in the following sense: for each $x\in \mathcal{H}$ and $\mathbf{T}=(T_1,T_2,\cdots)\in\mathcal{DC}$,
$$\|x\|^2=\sum_{\alpha\in\mathbb{Z}_+^{(\infty)}}\|D_{\mathbf{T}^*}\mathbf{T}^{*\alpha}x\|^2.$$
Unfortunately, the answer is negative in general. We will see that there exists a sequence $\mathbf{T}\in\mathcal{DC}$ such that the defect operator $D_{\mathbf{T}^*}$ of $\mathbf{T}^*$ is $0$, which cannot  occur in the finite-tuple case.

The class of doubly commuting sequences of pure isometries on $\mathcal{H}$ is denoted by $\mathcal{DP}(\mathcal{H})$ (or $\mathcal{DP}$). Since an isometry $V$ is of class $C_{\cdot0}$ if and only if $V$ is pure (that is to say, the unitary part in the Wold decomposition of $V$ is $0$), $\mathcal{DP}$ is a subclass of $\mathcal{DC}$. Any doubly commuting sequence of contractions has a minimal, doubly commuting, regular isometric dilation, which is  unique up to unitarily equivalence \cite{Sha}. Furthermore, we will show that the minimal regular isometric dilation of a sequence in class $\mathcal{DC}$ consists of doubly commuting pure isometries. %that is, there exists an isometry $\Pi:\mathcal{H}\rightarrow \mathcal{K}$ such that $\Pi T_n^*=V_n^*\Pi$ for each $n\in\mathbb{N}$. Moreover, we will show that the minimal isometric dilation of a sequence in class $\mathcal{DC}$ is also doubly commuting.
% in the sense that if there is another isometric dilation $\mathbf{V}'$ of $\mathbf{T}$, then $\mathbf{V}'$ is also a dilation of $\mathbf{V}$.
This  therefore provides us with an approach to the question arisen in previous subsection via the study of the sequences in class $\mathcal{DP}$. These notions concerning  the  dilation will be explained in Section 2.

Unlike the finite-tuple case,  sequences in class $\mathcal{DP}$ require further classification. Let us start some notations and definitions.
%Below we give a definition for those sequences in class $\mathcal{DC}$ would give an answer to the question via their minimal regular isometric dilations. We first introduce some notations.
For a family $\mathcal{T}$  of bounded linear operators on $\mathcal{H}$ and a subset $E$ of $\mathcal{H}$, let $[E]_{\mathcal{T}}$ denote the joint invariant subspace for $\mathcal{T}$ generated by $E$. In particular, if $\mathbf{T}$ is a sequence of operators, then one has
$$[E]_{\mathbf{T}}=\bigvee_{\alpha\in\mathbb{Z}_+^{(\infty)}}\mathbf{T}^\alpha E,$$ where the notation $\bigvee$ denotes the closed linear span of subsets of a Hilbert space.
Following \cite{Hal}, if a closed subspace $M$ of $\mathcal{H}$ is orthogonal to $\mathbf{T}^\alpha M$ for any $\alpha\in\mathbb{Z}_+^{(\infty)}\setminus\{0\}$, then $M$ is called a \textit{wandering subspace} for the sequence $\mathbf{T}$ (see \cite{BEKS} for an analogous definition in the finite-tuple case).

Suppose $\mathbf{V}=(V_1,V_2,\cdots)\in\mathcal{DP}$. It is easy to verify that the defect spce $$\mathfrak{D}_{\mathbf{V}^*}=\bigcap_{n=1}^\infty \mathrm{Ker}\ V_n^*,$$ and $\mathfrak{D}_{\mathbf{V}^*}$ is a wandering subspace for $\mathbf{V}$.
 It follows from Beurling's theorem  that if $\mathcal{I}$ is an invariant subspace of $H^2(\mathbb{D})$ for the Hardy shift $M_z$, then the wandering subspace for the isometry $M_\mathcal{I}=M_z|_\mathcal{I}$ always generates the entire $\mathcal{I}$; that is,
$$[\mathfrak{D}_{M_{\mathcal{I}}^*}]_{M_{\mathcal{I}}}
=[\mathcal{I}\ominus z\mathcal{I}]_{M_{\mathcal{I}}}=\mathcal{I}.$$
The conclusion is indeed valid for the vector-valued Hardy space $H_{\mathcal{E}}^2(\mathbb{D})$ by P. Halmos's observation in \cite{Hal} or the Beurling-Lax theorem. This suggests the following definition.

\begin{defn} \label{defn1} Suppose $\mathbf{V}\in\mathcal{DP}(\mathcal{H})$. The sequence $\mathbf{V}$ is said to be of Beurling type if $[\mathfrak{D}_{\mathbf{V}^*}]_{\mathbf{V}}=\mathcal{H}$.

Suppose $\mathbf{T}\in\mathcal{DC}(\mathcal{H})$. The sequence $\mathbf{T}$ is said to be of Beurling type if its minimal regular isometric dilation is of Beurling type.
\end{defn}

It was shown in \cite{Sar1} (also see \cite{BL}) that every doubly commuting finite-tuple of pure isometries
enjoys a ``Beurling type'' property.
The tuple $$\mathbf{M} _\zeta=(M_{\zeta_1},M_{\zeta_2},\cdots)$$ of coordinate multiplication operators on a vector-valued Hardy space $$H_{\mathcal{E}}^2(\mathbb{D}_2^\infty)=H^2(\mathbb{D}_2^\infty)\otimes\mathcal{E}$$ is clearly of Beurling type. More interestingly, we will see that the converse also holds: if a sequence  $\mathbf{V}\in\mathcal{DP}$ is of Beurling type, then $\mathbf{V}$ must be of the above form.
Then the question reduces to the characterization for sequences in class $\mathcal{DC}$ which are of Beurling type.

 Let $\varphi_a\ (a\in\mathbb{D})$ denote the holomorphic
automorphism
$$\varphi_a(z)=\frac{a-z}{1-\bar{a}z},\quad z\in\mathbb{D}$$
of $\mathbb{D}$.
It is known that the Riesz functional calculus $$\varphi_a(T)=(aI-T)(I-\bar{a}T)^{-1}\quad (a\in\mathbb{D})$$ of a contraction $T$ on the Hilbert space $\mathcal{H}$ is also a contraction. Furthermore, if $T\in C_{\cdot0}$, then $\varphi_a(T)\in C_{\cdot0}$ by the dilation theory for single $C_{\cdot0}$-contraction.
%Let $\mathbb{D}^\infty$ denote the Cartesian product of a sequence of $\mathbb{D}$, and
For $\mathbf{T}=(T_1,T_2,\cdots)\in\mathcal{DC}$  and $\lambda=(\lambda_1,\lambda_2,\cdots)\in\mathbb{D}^\infty$, set
$$\Phi_\lambda(\mathbf{T})=(\varphi_{\lambda_1}(T_1), \varphi_{\lambda_2}(T_2), \cdots).$$
Obviously, for any $\lambda\in\mathbb{D}^\infty$, $\Phi_\lambda$ defines a bijection from $\mathcal{DC}$ onto itself, and $\Phi_0=\Phi_\lambda\circ\Phi_\lambda=id_{\mathcal{DC}}$. Therefore, for a sequence $\mathbf{T}$ of operators, $\mathbf{T}\in\mathcal{DC}$ if and only if $\Phi_\lambda(\mathbf{T})\in\mathcal{DC}$.
 It is also easy to see that $\Phi_\lambda$ maps $\mathcal{DP}$ onto itself. We also write
 $\Phi_\lambda^\alpha(\mathbf{T})$ for the operator
 $\varphi_{\lambda_1}^{\alpha_1}(T_1)\cdots
 \varphi_{\lambda_n}^{\alpha_n}(T_n)$, where $\alpha=(\alpha_1,\cdots,\alpha_n,0,0,\cdots)\in\mathbb{Z}_+^{(\infty)}$ ($n\in\mathbb{N}$).

%\textbf{ However, a result in a previous paper \cite{DGH} of us implies that if $\lambda\notin\mathbb{D}_2^\infty$, then the wandering subspace $\mathfrak{D}_{\Phi_\lambda(\mathbf{M}_\zeta)^*}$ of the sequence $$\Phi_\lambda(\mathbf{M}_\zeta)=(M_{\widetilde{\varphi_{\lambda_1}}},M_{\widetilde{\varphi_{\lambda_2}}},\cdots)$$ on $H^2(\mathbb{D}_2^\infty)$ is $\{0\}$, where $\widetilde{\varphi_{\lambda_n}}(\zeta)=\varphi_{\lambda_n}(\zeta_n)\ (n\in\mathbb{N},\zeta\in\mathbb{D}^\infty_2)$ (see Lemma \ref{wand subsp is 0}). Therefore, $\Phi_\lambda(\mathbf{M}_\zeta)$ is  not of Beurling type. Also see \cite[Section 4]{O} for another example of a sequence that is not of Beurling type.
% We need to mention here that a doubly commuting finite-tuple $(V_1,\cdots,V_n)$ of pure isometries is always jointly unitarily equivalent to the tuple $(M_{\zeta_1},\cdots,M_{\zeta_n})$ of coordinate multiplication operators on a vector-valued Hardy space $H^2_{\mathcal{E}}(\mathbb{D}^n)$ over the $n$-polydisc $\mathbb{D}^n$ \cite{Sar3}. This indicates that
% the structure of sequences in class $\mathcal{DP}$ is much more complicated and nontrivial than the finite-tuple case.}

There are various examples of sequences $\mathbf{V}$ in class $\mathcal{DP}$ satisfying that the wandering space $\mathfrak{D}_{\mathbf{V}^*}=\{0\}$, while the defect space
$\mathfrak{D}_{\Phi_\lambda(\mathbf{V})^*}$
of $\Phi_\lambda(\mathbf{V})^*$ is nonzero.
For instance, let $\{\eta_n\}_{n\in\mathbb{N}}$ be a sequence of nonconstant inner functions in $H^\infty(\mathbb{D})$ and consider the sequence
$\mathbf{V}=(M_{\widetilde{\eta_1}},M_{\widetilde{\eta_2}},\cdots)$ of multiplication operators on
the Hardy space $H^2(\mathbb{D}_2^\infty)$, $\widetilde{\eta_n}(\zeta)=\eta_n(\zeta_n)\ (n\in\mathbb{N},\zeta\in\mathbb{D}^\infty_2)$.
We will prove that for ``almost all" choices of the sequence $\{\eta_n\}_{n\in\mathbb{N}}$ of inner functions, the wandering space $\mathfrak{D}_{\mathbf{V}^*}$ for the $\mathcal{DP}$-sequence $\mathbf{V}$ is $\{0\}$, and $\Phi_\lambda(\mathbf{V})$ is of Beurling type for some $\lambda\in\mathbb{D}^\infty$ (See Proposition \ref{sequence of multiplication operators} and Remark \ref{remark for Section 4}). Therefore, for those $\mathcal{DP}$-sequences which are not of Beurling type, the family of maps $\Phi_\lambda\ (\lambda\in\mathbb{D}^\infty)$ could be a powerful tool
in building analytic models. We thus could obtain important information of such $\mathcal{DP}$-sequences $\mathbf{V}$ from the Beurling type sequences $\Phi_\lambda(\mathbf{V})$ for some $\lambda\in\mathbb{D}^\infty$.

Inspired by this, we consider the following sequences in class $\mathcal{DP}$.

\begin{defn} \label{defn2} Suppose $\mathbf{V}\in\mathcal{DP}$. The sequence $\mathbf{V}$ is said to be of quasi-Beurling type if $\Phi_\lambda(\mathbf{V})$ is of Beurling type for some $\lambda\in\mathbb{D}^\infty$.

Suppose $\mathbf{T}\in\mathcal{DC}$. The sequence $\mathbf{T}$ is said to be of quasi-Beurling type if its minimal regular isometric dilation is of quasi-Beurling type.
\end{defn}

Sequences of quasi-Beurling type is relatively tractable in class $\mathcal{DP}$. In  Section 3, we will show  that a sequence $\mathbf{V}\in\mathcal{DP}$ is of quasi-Beurling type if and only if $\mathbf{V}$ is jointly unitarily equivalent to the sequence $\Phi_\lambda(\mathbf{M}_\zeta)$ on a vector-valued Hardy space $H^2_{\mathcal{E}}(\mathbb{D}_2^\infty)$
for some $\lambda\in\mathbb{D}^\infty$.
Recall that two sequence $\mathbf{T}=(T_1,T_2,\cdots)$ and $\mathbf{S}=(S_1,S_2,\cdots)$ of operators, defined on  $\mathcal{H}$ and $\mathcal{K}$ respectively, are said to be jointly unitarily equivalent if there exists a unitary operator $U:\mathcal{H}\rightarrow\mathcal{K}$ such that $$S_n=UT_nU^*,\quad n\in\mathbb{N}.$$

The first  main result in this paper is to give a complete characterization of sequences in class $\mathcal{DC}$ that can be decomposed into  direct sums of sequences of quasi-Beurling type. Note that for a commuting sequence of operators
on $\mathcal{H}$, by using a standard argument involving Zorn's Lemma,
one can decompose  $\mathcal{H}$ into orthogonal direct sums of separable $\mathbf{T}$-joint reducing subspaces.
It suffices to restrict our study to the case of separable Hilbert spaces. From now on, we only consider
separable Hilbert spaces.

A sequence $\mathbf{T}\in\mathcal{DC}(\mathcal{H})$ is said to have \textit{a decomposition
of quasi-Beurling type} if there exists an orthogonal decomposition $\mathcal{H}=\bigoplus_\gamma \mathcal{H}_\gamma$ of the Hilbert space $\mathcal{H}$, such that each $\mathcal{H}_\gamma$ is $\mathbf{T}$-joint reducing and each $\mathbf{T}|_{\mathcal{H}_\gamma}$ is of quasi-Beurling type.

\begin{thm} \label{direct sums of quasi-BT} Suppose $\mathbf{T}\in\mathcal{DC}(\mathcal{H})$. The followings are equivalent:
\begin{itemize}
  \item [(1)] $\mathbf{T}$ has a decomposition
of quasi-Beurling type;
  \item [(2)] for each $x\in \mathcal{H}$, there exists a sequence $\lambda^{(1)},\lambda^{(2)},\cdots$ of points in $\mathbb{D}^\infty$, such that
      $$\|x\|^2=\sum_{k=1}^\infty\sum_{\alpha\in\mathbb{Z}_+^{(\infty)}}
      \|D_{\Phi_{\lambda^{(k)}}(\mathbf{T})^*}\Phi_{\lambda^{(k)}}^\alpha(\mathbf{T})^{*}x\|^2;$$
  \item [(3)] $\bigvee_{\lambda\in\mathbb{D}^\infty}
\mathfrak{D}_{\Phi_\lambda(\mathbf{T})^*}=\mathcal{H}$.
\end{itemize}
In fact,
%if in addition the Hilbert space $\mathcal{H}$ is separable, then
the sequence $\lambda^{(1)},\lambda^{(2)},\cdots$ in condition (2) can be chosen to be independent of $x\in \mathcal{H}$.
\end{thm}

The following particular case of Theorem \ref{direct sums of quasi-BT} completely answers the question that when a doubly commuting sequence of $C_{\cdot0}$-contractions has a regular dilation  to the  tuple of coordinate multiplication operators on some vector-valued Hardy space over the infinite-dimensional polydisc.

\begin{cor} \label{BT} Let  Suppose $\mathbf{T}\in\mathcal{DC}(\mathcal{H})$. The followings are equivalent:
\begin{itemize}
  \item [(1)] $\mathbf{T}$ is of Beurling type;
  \item [(2)] the minimal regular isometric dilation of $\mathbf{T}$ is jointly unitarily equivalent to the tuple $\mathbf{M}_\zeta$ of coordinate multiplication operators on a vector-valued Hardy space $H_{\mathcal{E}}^2(\mathbb{D}_2^\infty)$;
  \item [(3)] for each $x\in \mathcal{H}$, $\|x\|^2=\sum_{\alpha\in\mathbb{Z}_+^{(\infty)}}\|D_{\mathbf{T}^*}\mathbf{T}^{*\alpha}x\|^2$;
  \item [(4)] $\bigvee_{\lambda\in\mathbb{D}_2^\infty}
\mathfrak{D}_{\Phi_\lambda(\mathbf{T})^*}=\mathcal{H}$.
\end{itemize}
\end{cor}

Note that for $\mathbf{T}\in\mathcal{DCF}$, the condition (3) in Corollary \ref{BT} coincides with the identity (\ref{key identity}), Corollary \ref{BT}  actually generalizes the  finite-tuple case.

Here are some remarks for Theorem \ref{direct sums of quasi-BT}. If $\mathbf{T}\in\mathcal{DC}$ with a decomposition $\mathbf{T}=\bigoplus_\gamma\mathbf{T}_\gamma$ of quasi-Beurling type, then
there correspond a point $\lambda_\gamma\in\mathbb{D}^\infty$ and a Hilbert space $\mathcal{E}_\gamma$ to each index $\gamma$, such that
$\mathbf{T}_\gamma$ is jointly unitarily equivalent to $P_{\mathcal{Q}_\gamma}\Phi_{\lambda_\gamma}(\mathbf{M}_\zeta)|_{\mathcal{Q}_\gamma}$,
where the sequence $\Phi_{\lambda_\gamma}(\mathbf{M}_\zeta)$ is defined on the vector-valued Hardy space $H^2_{\mathcal{E}_\gamma}(\mathbb{D}_2^\infty)$, and $\mathcal{Q}_\gamma\subseteq H^2_{\mathcal{E}_\gamma}(\mathbb{D}_2^\infty)$ is a $\mathbf{M}_\zeta^*$-joint invariant subspace. This gives
\begin{equation}\label{representation of direct sum of quasi-BT}\mathbf{T}\cong \bigoplus_\gamma P_{\mathcal{Q}_\gamma}\Phi_{\lambda_\gamma}(\mathbf{M}_\zeta)|_{\mathcal{Q}_\gamma},
\end{equation} %P_{\mathcal{Q}_1}\Phi_{\lambda^{(1)}}(\mathbf{M}_\zeta)|_{\mathcal{Q}_1}
%\oplus P_{\mathcal{Q}_2}\Phi_{\lambda^{(2)}}(\mathbf{M}_\zeta)|_{\mathcal{Q}_2}\oplus\cdots,
and we therefore build an analytic model for a sequence $\mathbf{T}\in\mathcal{DC}(\mathcal{H})$ under the assumption  that the subset $$\{D_{\Phi_\lambda(\mathbf{T})^*}x:\lambda\in\mathbb{D}^\infty,x\in\mathcal{H}\}$$ is complete in $\mathcal{H}$.
%Also, an example will be given to
%illustrate this assumption is nontrivial, although it
Note that this assumption always holds for the finite-tuple case (see Lemma \ref{intersection of Ker}).
Condition (2) in the theorem further generalizes the identity (\ref{key identity}).  See Section 3 for  the details in these paragraphs.
Also, the following result  illustrates that unlike the finite-tuple case (Lemma \ref{intersection of Ker}), some extreme phenomenon would occur in the infinite-tuple case.
\begin{thm} \label{example}  There exists a sequence $\mathbf{V}\in\mathcal{DP}$ such that $\mathfrak{D}_{\Phi_\lambda(\mathbf{V})^*}=\{0\}$ for each
$\lambda\in\mathbb{D}^\infty$.
\end{thm}

We further refine the representation (\ref{representation of direct sum of quasi-BT})
by giving a characterization of the subspaces $\mathcal{Q}_\gamma$  involving characterization  functions of operators in $\mathbf{T}_\gamma$. We will give the details in Section 4.
This generalizes results in \cite{BNS} to the infinite-variable case.

In section 4,
we will prove that every sequence in class $\mathcal{DP}$ is jointly unitarily equivalent to
a sequence of multiplication operators induced by
operator-valued inner functions each of which involves one different variable (Theorem \ref{analytic model}). Thus we establish  operator-valued analytic functional models for general sequences in class $\mathcal{DC}$, which generalize (\ref{representation of direct sum of quasi-BT}).
We also have the following application of our results.
\begin{cor} \label{application 2}
  Suppose $\mathbf{T}\in\mathcal{DC}$. Then there exists a sequence
  $\{B_n\}_{n\in\mathbb{N}}$ of finite Blaschke products, such that
  $\{\prod_{i=1}^{n}B_i(T_i)\}_{n\in\mathbb{N}}$ converges in the strong operator topology.
\end{cor}

In the rest of this paper, we use the language of Hilbert module by Douglas and Paulsen \cite{DP} to state  the generalization of the Beurling-Lax  theorem and results related with  Jordan blocks in  the finite-variable  case
to the case  in infinitely many variables.

Let $\mathcal{P}_\infty$ denote the polynomial ring  in countably infinitely many complex variables, in which each polynomial only involves finitely many variables.
  Similar to the finite-variable case, for   each commuting sequence $\mathbf{T}$ of operators on a Hilbert space $\mathcal{H}$, one  defines a $\mathcal{P}_\infty$-module structure on $\mathcal{H}$ by
   $$ph=p(\mathbf{T})h,\quad p\in\mathcal{P}_\infty, h\in\mathcal{H}.$$
  Say that this $\mathcal{P}_\infty$-module $\mathcal{H}$ is \textit{doubly commuting} if the sequence  $\mathbf{T}$ is doubly commuting.
  Conversely, any $\mathcal{P}_\infty$-module structure is determined by a commuting sequence of operators in the above way.

  We can also define the  Hardy module structure on  a vector-valued Hardy space $H_{\mathcal{E}}^2(\mathbb{D}_2^\infty)$ in infinitely many variables via the tuple $\mathbf{M}_\zeta$ of coordinate multiplication operators on $H_{\mathcal{E}}^2(\mathbb{D}_2^\infty)$. The module action is the multiplication by polynomials and  submodules of $H_{\mathcal{E}}^2(\mathbb{D}_2^\infty)$ are exactly joint invariant subspaces for $\mathbf{M}_\zeta$. By the definition,
a submodule $\mathcal{S}$ of $H_{\mathcal{E}}^2(\mathbb{D}_2^\infty)$ is doubly commuting if the restriction $$(M_{\zeta_1}|_\mathcal{S},M_{\zeta_2}|_\mathcal{S},\cdots)$$ of $\mathbf{M}_\zeta$ on $\mathcal{S}$ is doubly commuting; a quotient module $\mathcal{Q}$ of $H_{\mathcal{E}}^2(\mathbb{D}_2^\infty)$ is doubly commuting if the compression $$(P_\mathcal{Q}M_{\zeta_1}|_\mathcal{Q},P_\mathcal{Q}M_{\zeta_2}|_\mathcal{Q},\cdots)$$ of $\mathbf{M}_\zeta$ on $\mathcal{Q}$ is doubly commuting. We will prove that such doubly commuting restriction  and  compression are of Beurling type.
 As a consequence, we have
\begin{thm} \label{doubly commuting submodule} Let   $\mathcal{S}$ be a  submodule of the vector-valued Hardy module
 $H_{\mathcal{E}}^2(\mathbb{D}_2^\infty)$. Then $\mathcal{S}$ is doubly commuting if and only if there exist a Hilbert space $\mathcal{F}$ and an inner function $\Psi\in H_{\mathcal{B}(\mathcal{F},\mathcal{E})}^\infty(\mathbb{D}_2^\infty)$, such that
 $$\mathcal{S}=\Psi H_{\mathcal{F}}^2(\mathbb{D}_2^\infty).$$
\end{thm}

 \begin{thm} \label{doubly commuting quotient of H^2(D^infty)} Every  doubly commuting quotient module of
 $H^2(\mathbb{D}_2^\infty)$ is the tensor product of some sequence of quotient modules of $H^2(\mathbb{D})$.
\end{thm}

Theorem \ref{doubly commuting submodule} is a infinite-variable version of the Beurling-Lax Theorem
for doubly commuting Hardy submodules.
Also from Theorem \ref{doubly commuting quotient of H^2(D^infty)}, the compression of the tuple
 of coordinate multiplication operators on a nontrivial  doubly commuting quotient  module of
 $H^2(\mathbb{D}_2^\infty)$ is a Jordan block in infinitely many variables.
 %An example is also presented to illustrate that
% Theorem \ref{doubly commuting quotient of H^2(D^infty)} fails for the vector-valued Hardy module $H_{\mathcal{E}}^2(\mathbb{D}_2^\infty)$.
  %even when $\mathcal{E}$ is finite-dimensional.
 See Section 5 for  the details.

It is worth mentioning that our methods presented in this paper are quite different from that in \cite{O,Sar2,BNS} and also valid for nonseparable Hilbert spaces.

\section{Some preparations}
In this section, we first establish a lemma to guarantee the validity of Definition \ref{defn1} and \ref{defn2}. Then we list some basic properties of vector-valued Hardy spaces and operator-valued functions.
Lastly, the rest of this section is dedicated to  preparations for the proofs of main results in subsequent sections.

\subsection{The minimal regular isometric dilation}
Suppose that the sequence $\mathbf{T}=(T_1,T_2,\cdots)$  is a doubly commuting sequence of contractions on $\mathcal{H}$. Say, a sequence $\mathbf{V}=(V_1,V_2,\cdots)$ of isometries, defined on a larger Hilbert space $\mathcal{K}_{\mathbf{V}}\supseteq\mathcal{H}$, is
an isometric dilation of $\mathbf{T}$ if for each $\alpha\in\mathbb{Z}_+^{(\infty)}$,
$$\mathbf{T}^\alpha=P_{\mathcal{H}}\mathbf{V}^\alpha|_{\mathcal{H}}.$$
Furthermore, if for $\alpha,\beta\in\mathbb{Z}_+^{(\infty)}$ satisfying
$\alpha\wedge\beta=(0,0,\cdots)$, $$\mathbf{T}^{*\alpha}\mathbf{T}^\beta
=P_{\mathcal{H}}\mathbf{V}^{*\alpha}\mathbf{V}^\beta|_{\mathcal{H}},$$
then the isometric dilation $\mathbf{V}$ of $\mathbf{T}$ is said to be regular,
where $$\alpha\wedge\beta=(\min\{\alpha_1,\beta_1\},\min\{\alpha_2,\beta_2\},\cdots).$$
An isometric dilation $\mathbf{V}$ of $\mathbf{T}$ is said to be minimal if the $\mathbf{V}$-joint invariant subspace $[\mathcal{H}]_{\mathbf{V}}$ generated by
$\mathcal{H}$ is  $\mathcal{K}_{\mathbf{V}}$.
It is clear to see that for any regular isometric dilation $\mathbf{V}$ of $\mathbf{T}$, the restriction $\mathbf{V}|_{[\mathcal{H}]_{\mathbf{V}}}$
of the sequence $\mathbf{V}$ on $[\mathcal{H}]_{\mathbf{V}}$ is a minimal regular isometric dilation of $\mathbf{T}$. The minimal regular isometric dilation is unique up to the joint unitary equivalence in the sense that
if both $\mathbf{V}=(V_1,V_2,\cdots)$ and $\mathbf{W}=(W_1,W_2,\cdots)$ are minimal regular isometric dilations of $\mathbf{T}$, then there exists a unitary operator $U:\mathcal{K}_{\mathbf{V}}\rightarrow\mathcal{K}_{\mathbf{W}}$
such that $U$ fixes vectors in $\mathcal{H}$ and $$UV_n=W_nU,\quad n\in\mathbb{N}.$$
Applying \cite[Theorem 4.2]{Sha} to the semigroup $\mathbb{Z}_+^{(\infty)}$,  we see that the minimal regular isometric dilation of $\mathbf{T}$ always exists and is also doubly commuting.
Note that the existence
%of the minimal regular isometric dilation
can be  also deduced from \cite[pp. 36-37]{SNFBK} by restrict the minimal regular unitary dilation $\mathbf{U}$ of $\mathbf{T}$ on $[\mathcal{H}]_{\mathbf{U}}$.
For the convenience of the readers, below we   prove directly for the case $\mathbf{T}\in\mathcal{DC}$  that  the minimal regular isometric dilation $\mathbf{V}$ of $\mathbf{T}$ is in class $\mathcal{DP}$. Moreover, $\mathbf{V}$ is a coextension of $\mathbf{T}$, which means that $\mathcal{H}$ is invariant for $\mathbf{V}^*$ and $\mathbf{T}^*=\mathbf{V}^*|_{\mathcal{H}}$, the restriction of $\mathbf{V}^*$ on $\mathcal{H}$.

\begin{lem}\label{min iso dilation}
The minimal regular isometric dilation of a sequence $\mathbf{T}\in\mathcal{DC}$ is in class $\mathcal{DP}$ and  coextends $\mathbf{T}$.
\end{lem}
\begin{proof}
  It is equivalent to show that if $\mathbf{V}=(V_1,V_2,\cdots)$ is the minimal regular isometric dilation of $\mathbf{T}\in\mathcal{DC}(\mathcal{H})$, then for any given $n\in\mathbb{N}$, $V_n$ is pure and douly commutes with $V_m$ for any $m\neq n$, and $\mathcal{H}$ is invariant for $V_n^*$.   Assume $n=1$ without loss of generality, and put $\mathcal{E}=\mathfrak{D}_{T_1^*}$ and $\mathbf{T}'=(T_2,T_3,\cdots)$.
  Since $\mathbf{T}$ is doubly commuting, we see that  $\mathcal{E}$ is $\mathbf{T}'$-joint reducing and the restriction $\mathbf{T}'|_{\mathcal{E}}=(T_2|_{\mathcal{E}},T_3|_{\mathcal{E}},\cdots)$
 of $\mathbf{T}'$ on $\mathcal{E}$ is also doubly commuting. Let the sequence $\mathbf{S}=(S_1,S_2,\cdots)$, defined on a Hilbert space $\mathcal{F}\supseteq\mathcal{E}$, be the minimal regular isometric dilation of $\mathbf{T}'|_{\mathcal{E}}$.

  Sz.-Nagy and Foias' functional model theory for single $C_{.0}$-contraction gives the following isometric embedding  \cite{SNFBK}
  \begin{eqnarray*}V:\quad \mathcal{H}& \rightarrow &
  H_{\mathcal{E}}^2(\mathbb{D})=H^2(\mathbb{D})\otimes\mathcal{E},\\
  x & \mapsto & \sum_{k=0}^\infty z^k\cdot D_{T_1^*}T_1^{*k}x.\end{eqnarray*}
 By identifying $\mathcal{H}$ with the $M_z^*\otimes I_{\mathcal{E}}$-invariant subspace $V\mathcal{H}$ via the isometry $V$,  one obtains a minimal isometric dilation $M_z\otimes I_{\mathcal{E}}$ of the contraction $T_1$, which is also a coextension.
  Set $$W_1=M_z\otimes I_{\mathcal{F}}$$ and $$W_m=I_{H^2(\mathbb{D})}\otimes S_{m-1},\quad m\geq2.$$ It is routine to check that the sequence $\mathbf{W}=(W_1,W_2,\cdots)$ of isometries is a regular isometric dilation of $\mathbf{T}$. We claim that this dilation $\mathbf{W}$ is also minimal. Since
  $\mathbf{S}=(S_1,S_2,\cdots)$ is the minimal regular isometric dilation of $\mathbf{T}'|_{\mathcal{E}}$,
  we have
  $$[H_{\mathcal{E}}^2(\mathbb{D})]_{(W_2,W_3,\cdots)}
  =H^2(\mathbb{D})\otimes[\mathcal{E}]_\mathbf{S}
  =H_{\mathcal{F}}^2(\mathbb{D}).$$
  This together with
  $[\mathcal{H}]_{W_1}=H_{\mathcal{E}}^2(\mathbb{D})$ proves the claim.
By the uniqueness of minimal regular isometric dilation, there exists a unitary operator $U$ that interwinds $\mathbf{V}$ and $\mathbf{W}$ and fixes vectors in $\mathcal{H}$. Therefore $V_1$ is pure and doubly commutes with $V_m\ (m\geq2)$. Moreover, $$V_1^*\mathcal{H}
=U^*W_1^*U\mathcal{H}=U^*W_1^*\mathcal{H}
=U^*(M_z\otimes I_{\mathcal{E}})^*\mathcal{H}
\subseteq U^*\mathcal{H}=\mathcal{H}.$$
The proof is complete.
\end{proof}

We also record the following useful lemma concerning  the minimal regular isometric dilation.

\begin{lem} \label{intersection of H_n} Let $\mathbf{T}$ be a sequence in class $\mathcal{DC}(\mathcal{H})$, and $\mathbf{V}\in\mathcal{DP}(\mathcal{K})$ be an isometric coextension of $\mathbf{T}$.
 Put $$\mathcal{H}_n=\bigvee_{\substack{\alpha\in\mathbb{Z}_+^{(\infty)} \\ \alpha_n=0 }}\mathbf{V}^\alpha\mathcal{H},\quad n\in\mathbb{N}.$$ Then the following conclusions hold:
 \begin{itemize}
   \item [(1)] $\mathcal{H}_n\ (n\in\mathbb{N})$ is $\mathbf{V}^*$-joint invariant and $[\mathcal{H}]_{\mathbf{V}}$ is $\mathbf{V}$-joint reducing.
   \item [(2)] $P_{\mathcal{H}}V_n^*P_{\mathcal{H}_n}
       =V_n^*P_{\mathcal{H}}$ for each $n\in\mathbb{N}$.
   \item [(3)]  Furthermore, if $\mathbf{V}$ is the minimal isometric dilation of $\mathbf{T}$, then $$\mathcal{H}=\bigcap_{n=1}^\infty \mathcal{H}_n.$$
 \end{itemize}
 Here, for a closed subspace $M$ of $\mathcal{K}$, $P_M$ denotes the orthogonal projection from $\mathcal{K}$ onto $M$.
\end{lem}

\begin{proof}
(1) Since the sequence $\mathbf{V}$ is doubly commuting, and $\mathcal{H}$ is joint invariant for $\mathbf{V}^*$, one obtains that for $n,m\in\mathbb{N}$ and $\alpha\in\mathbb{Z}_+^{(\infty)}$ with $\alpha_n=0$,
$$V_m^*\mathbf{V}^\alpha\mathcal{H}
=\mathbf{V}^\alpha V_m^*\mathcal{H}\subseteq\mathbf{V}^\alpha\mathcal{H}
\subseteq\mathcal{H}_n,\quad\mathrm{if}\ \alpha_m=0;$$
$$V_m^*\mathbf{V}^\alpha\mathcal{H}
=\mathbf{V}^{\alpha-1_m} \mathcal{H}\subseteq\mathcal{H}_n,\quad\mathrm{if}\ \alpha_m\geq1,$$ where
$$1_m=(0,\cdots,0,\overset{m\mathrm{-th}}{1},0,\cdots).$$
This gives that each $\mathcal{H}_n$ is $\mathbf{V}^*$-joint invariant, and then
$$[\mathcal{H}]_{\mathbf{V}}=\bigvee_{n=1}^\infty\mathcal{H}_n$$ is $\mathbf{V}$-joint reducing.
\vskip2mm
(2) Since the sequences $\mathbf{T}$ and $\mathbf{V}$ are doubly commuting, and $\mathcal{H}$ is joint invariant for $\mathbf{V}^*$, we see that for $n\in\mathbb{N}$ and $\alpha\in\mathbb{Z}_+^{(\infty)}$ with $\alpha_n=0$,
$$P_{\mathcal{H}}V_n^*\mathbf{V}^\alpha|_{\mathcal{H}}
=P_{\mathcal{H}}\mathbf{V}^\alpha P_{\mathcal{H}} V_n^*|_{\mathcal{H}}
=\mathbf{T}^\alpha T_n^*=T_n^*\mathbf{T}^\alpha
=V_n^* P_{\mathcal{H}} \mathbf{V}^\alpha|_{\mathcal{H}},
$$
forcing $$\mathbf{V}^\alpha\mathcal{H}\subseteq\mathrm{Ker}\ (P_{\mathcal{H}}V_n^*-V_n^*P_{\mathcal{H}}).$$
It follows that for each $n\in\mathbb{N}$, $$\mathcal{H}_n\subseteq
\mathrm{Ker}\ (P_{\mathcal{H}}V_n^*-V_n^*P_{\mathcal{H}}),$$
and then
$$P_{\mathcal{H}}V_n^*P_{\mathcal{H}_n}
=V_n^*P_{\mathcal{H}}P_{\mathcal{H}_n}
       =V_n^*P_{\mathcal{H}}.$$
\vskip2mm
(3) Write $\widetilde{\mathcal{H}}=\bigcap_{n=1}^\infty \mathcal{H}_n$. It is trivial that $\mathcal{H}\subseteq\widetilde{\mathcal{H}}$.
It follows from (1) and (2) that $\widetilde{\mathcal{H}}$ is joint invariant for $\mathbf{V}^*$, and for each $n\in\mathbb{N}$,
 $$P_{\mathcal{H}}V_n^*P_{\widetilde{\mathcal{H}}}
 =P_{\mathcal{H}}V_n^*P_{\mathcal{H}_n}P_{\widetilde{\mathcal{H}}}
 =V_n^*P_{\mathcal{H}}P_{\widetilde{\mathcal{H}}}
 =V_n^*P_{\mathcal{H}}.$$
Assume that $x$ is an element in $\widetilde{\mathcal{H}}\ominus\mathcal{H}$. Then for each $n\in\mathbb{N}$,
$$P_{\mathcal{H}}V_n^*x=P_{\mathcal{H}}V_n^*P_{\widetilde{\mathcal{H}}}x=V_n^*P_{\mathcal{H}}x=0,$$
forcing $V_n^*x\in \widetilde{\mathcal{H}}\ominus\mathcal{H}$.
Therefore, $\mathbf{V}^{*\alpha}x\in \widetilde{\mathcal{H}}\ominus\mathcal{H}$ for any $\alpha\in\mathbb{Z}_+^{(\infty)}$; that is,
$x$ is orthogonal to the $\mathbf{V}$-invariant subspace $[\mathcal{H}]_{\mathbf{V}}$ generated by $\mathcal{H}$. Since $\mathbf{V}$ is the minimal isometric dilation of the sequence $\mathbf{T}$  defined on $\mathcal{H}$, we actually have $[\mathcal{H}]_{\mathbf{V}}=\mathcal{K}$.
Thus $x=0$, and this proves $\mathcal{H}=\widetilde{\mathcal{H}}$.
\end{proof}
%\begin{lem} \label{Phi_lambda min iso dilation}
%If $\mathbf{T}$ is a sequence in class $\mathcal{DC}$ and $\mathbf{V}$ is
%the minimal isometric dilation of $\mathbf{T}$, then $\Phi_\lambda(\mathbf{V})$ is
%the minimal isometric dilation of $\Phi_\lambda(\mathbf{T})$ for any $\lambda\in\mathbb{D}^\infty$.
%\end{lem}

%For an isometric dilation $\mathbf{V}$ of a doubly commuting sequence $\mathbf{T}$ of contractions on $\mathcal{H}$, we could identify $\mathcal{H}$ with the subspace $\Pi_{\mathbf{V}} \mathcal{H}$ of $\mathcal{K}_{\mathbf{V}}$ via the isometry $\Pi_{\mathbf{V}}$, and $T_n^*$ would be the restriction of
%$V_n^*$ on $\mathcal{H}$.

\subsection{Some basic properties of vector-valued Hardy space and operator-valued function}

Here we list some basic properties of vector-valued Hardy spaces and operator-valued functions, and the notations $\mathcal{E},\mathcal{F}$ and $\mathcal{G}$ will always denote some Hilbert spaces.

The vector-valued Hardy space $H_{\mathcal{E}}^2(\mathbb{D}_2^\infty)$, over the domain $\mathbb{D}_2^\infty$ (connected open subset) in the Hilbert space $l^2$, consists of all
$\mathcal{E}$-valued functions of form
$$F(\zeta)=\sum_{\alpha\in\mathbb{Z}_+^{(\infty)}}\zeta^\alpha\cdot x_\alpha,\quad \zeta\in\mathbb{D}_2^\infty$$
with each $x_\alpha\in\mathcal{E}$ and $\|F\|^2=\sum_{\alpha\in\mathbb{Z}_+^{(\infty)}}\|x_\alpha\|_\mathcal{E}^2<\infty$.
By the Cauchy-Schwarz inequality, the above series converges pointwisely on $\mathbb{D}_2^\infty$ in  $\mathcal{E}$-norm.
We will follow the definition of holomorphic mapping given in \cite{Di} for any vector-valued function $F:\mathbb{D}_2^\infty\rightarrow\mathcal{X}$, where
$\mathcal{X}$ is an arbitrary Banach space.
Every function $F:\mathbb{D}_2^\infty\rightarrow\mathcal{E}$ in $H_{\mathcal{E}}^2(\mathbb{D}_2^\infty)$ is then holomorphic in this sense.

The space $H_{\mathcal{E}}^2(\mathbb{D}_2^\infty)$ can be  considered as the tensor product of the Hardy space $H^2(\mathbb{D}_2^\infty)$ and the Hilbert space $\mathcal{E}$
by identifying the vector-valued function $F\cdot x$ with the tensor product $F\otimes x$, where $F\in H^2(\mathbb{D}_2^\infty)$ and $x\in\mathcal{E}$.
Then the tuple of coordinate multiplication operators on $H_{\mathcal{E}}^2(\mathbb{D}_2^\infty)$ has the form
$$\mathbf{M}_\zeta\otimes I_{\mathcal{E}}=(M_{\zeta_1}\otimes I_{\mathcal{E}},M_{\zeta_2}\otimes I_{\mathcal{E}},\cdots).$$ For simplicity, we often write only $\mathbf{M}_\zeta$ for this tuple. Moreover, one can expanse functions
in $H_{\mathcal{E}}^2(\mathbb{D}_2^\infty)$ with respect to any orthonormal basis $\{e_ k\}_{k\in\mathbb{N}}$ of $\mathcal{E}$ as
$\sum_{k=1}^{\infty} F_ k\cdot e_ k$,
where each $F_ k\in H^2(\mathbb{D}_2^\infty)$ and $\sum_{ k=1}^\infty\|F_ k\|_{H^2(\mathbb{D}_2^\infty)}^2<\infty$.

Let $\mathbf{K}_{\lambda}$ denote the reproducing kernel of $H^2(\mathbb{D}^\infty_2)$ at the point $\lambda\in\mathbb{D}^\infty_2$.
Recall that a subset $E$ of a Hilbert space $\mathcal{H}$ is said to be complete
in $\mathcal{H}$ if $E$ spans a dense subspace of $\mathcal{H}$; that is,
the orthocomplement $E^\perp$ of $E$ in $\mathcal{H}$ is $\{0\}$.

\begin{lem} \label{K_lambda} Suppose $F\in H_{\mathcal{E}}^2(\mathbb{D}_2^\infty)$. Then the following conclusions hold:
\begin{itemize}
  \item [(1)]
$\langle F,\mathbf{K}_\lambda\cdot x\rangle=\langle F(\lambda),x\rangle$
for any $\lambda\in\mathbb{D}_2^\infty$ and any $x\in\mathcal{E}$. Consequently,
the set $$\{\mathbf{K}_\lambda\cdot x:\lambda\in\mathbb{D}_2^\infty,x\in\mathcal{E}\}$$ is complete in $H_{\mathcal{E}}^2(\mathbb{D}_2^\infty)$.
  \item [(2)] If $F\neq0$ and $M_{\zeta_n}^*F=\overline{\lambda_n}F$ for each $n\in\mathbb{N}$ and a sequence $\lambda=(\lambda_1,\lambda_2,\cdots)$ of complex numbers, then $\lambda\in\mathbb{D}_2^\infty$
and $F=\mathbf{K}_\lambda\cdot x$ for some $x\in\mathcal{E}$.
\end{itemize}
\end{lem}
\begin{proof}
(1) Assume $\|x\|=1$ without loss of generality and take any  orthonormal basis $\{e_ k\}_{k\in\mathbb{N}}$ of the subspace $\{x\}^\perp$ of $\mathcal{E}$. Expanse
the function $F$ with respect to the orthonormal basis $\{x\}\cup\{e_ k\}_{k\in\mathbb{N}}$ of $\mathcal{E}$ as
$$F=F_x\cdot x+\sum_{k=1}^{\infty} F_ k\cdot e_ k.$$
Then we have $$\langle F,\mathbf{K}_\lambda\cdot x\rangle=\langle F_x\cdot x,\mathbf{K}_\lambda\cdot x\rangle+\sum_{k=1}^{\infty}\langle F_ k\cdot e_ k,\mathbf{K}_\lambda\cdot x\rangle=F_x(\lambda)$$
and
$$\langle F(\lambda),x\rangle=\langle F_x(\lambda) x,x\rangle+\sum_{k=1}^{\infty}\langle F_ k(\lambda)e_ k, x\rangle=F_x(\lambda).$$
This proves (1).
\vskip2mm

(2) Expanse the function $F$ with respect to an orthonormal basis $\{e_ k\}_{k\in\mathbb{N}}$ of $\mathcal{E}$ as
$F=\sum_{k=1}^{\infty} F_ k\cdot e_ k$,
then $$M_{\zeta_n}^*F=M_{\zeta_n}^*(\sum_{k=1}^{\infty} F_ k\cdot e_ k)=\sum_{k=1}^{\infty} (M_{\zeta_n}^*F_ k)\cdot e_ k,\quad n\in\mathbb{N}.$$
It follows that $M_{\zeta_n}^*F_ k=\overline{\lambda_n}F_ k$ for any  $ k,n\in\mathbb{N}$. This implies that  all $F_k$ are in the orthocomplement of the invariant subspace generated by $\zeta_1-\lambda_1$, $\zeta_2-\lambda_2$, $\cdots$. Since $F\not =0,$ it follows that this invariant subspace is proper, and its  orthocomplement is $\mathbb{C}\mathbf{K}_\lambda.$ Therefore,
we have $\lambda\in\mathbb{D}_2^\infty$
and $F_ k=c_ k\mathbf{K}_\lambda$ with
$\sum_{k=1}^{\infty}|c_ k|^2<\infty$,
and thus $$F=\sum_{k=1}^{\infty} F_ k\cdot e_ k=\mathbf{K}_\lambda\cdot(\sum_{k=1}^{\infty} c_ k e_ k).$$
The proof is complete.
\end{proof}

The space $H_{\mathcal{B}(\mathcal{F},\mathcal{E})}^\infty(\mathbb{D}_2^\infty)$
consists of all uniformly bounded holomorphic operator-valued functions $\Psi:\mathbb{D}_2^\infty\rightarrow\mathcal{B}(\mathcal{F},\mathcal{E})$. Here by $\Psi$ is uniformly bounded we mean
$\|\Psi\|_\infty=\sup_{\zeta\in\mathbb{D}_2^\infty}\|\Psi(\zeta)\|<\infty$,
and $\mathcal{B}(\mathcal{F},\mathcal{E})$ denotes the Banach space of bounded linear operators from  $\mathcal{F}$ to $\mathcal{E}$.
Every function $\Psi$ in $H_{\mathcal{B}(\mathcal{F},\mathcal{E})}^\infty(\mathbb{D}_2^\infty)$ naturally induce a multiplication operator $M_\Psi$ as follows:
 \begin{eqnarray*}M_\Psi:\quad H_{\mathcal{F}}^2(\mathbb{D}_2^\infty)& \rightarrow &
  H_{\mathcal{E}}^2(\mathbb{D}_2^\infty),\\
  F & \mapsto & \Psi F,\end{eqnarray*}
where $\Psi F(\zeta)=\Psi(\zeta)F(\zeta)\ (\zeta\in\mathbb{D}_2^\infty)$.
It is clear that $M_\Psi$ is bounded of norm $\leq$ $\|\Psi\|_\infty$.
We now claim that
\begin{equation} \label{M_Psi*}
M_\Psi^*(\mathbf{K}_\lambda\cdot x)=\mathbf{K}_\lambda\cdot\Psi(\lambda)^*x
\end{equation} for any
$\lambda\in\mathbb{D}_2^\infty$ and any $x\in\mathcal{E}$.
By Lemma \ref{K_lambda} (1), it suffices to show that for any fixed $\mu\in\mathbb{D}_2^\infty$ and $y\in\mathcal{F}$,
$$\langle \mathbf{K}_\mu\cdot y,M_\Psi^*(\mathbf{K}_\lambda\cdot x)\rangle
=\langle \mathbf{K}_\mu\cdot y,\mathbf{K}_\lambda\cdot\Psi(\lambda)^*x\rangle.$$
Again by Lemma \ref{K_lambda} (1), we have
$$\langle \mathbf{K}_\mu\cdot y,M_\Psi^*(\mathbf{K}_\lambda\cdot x)\rangle
=\langle \Psi(\mathbf{K}_\mu\cdot y),\mathbf{K}_\lambda\cdot x\rangle
=\langle \mathbf{K}_\mu(\lambda)\Psi(\lambda)y, x\rangle.$$
On the other hand,
$$\langle \mathbf{K}_\mu\cdot y,\mathbf{K}_\lambda\cdot\Psi(\lambda)^*x\rangle
=\langle\mathbf{K}_\mu,\mathbf{K}_\lambda\rangle\cdot\langle \Psi(\lambda)y, x\rangle=\mathbf{K}_\mu(\lambda)\langle \Psi(\lambda)y, x\rangle.$$
This proves the claim.

We are ready to prove the following.
\begin{prop} \label{operatoe interwine CMO}
Let $\mathbf{M}_\zeta,\mathbf{M}_\xi$ be the tuple of coordinate multiplication operators on $H_{\mathcal{E}}^2(\mathbb{D}_2^\infty)$ and $H_{\mathcal{F}}^2(\mathbb{D}_2^\infty)$, respectively. If $T:H_{\mathcal{F}}^2(\mathbb{D}_2^\infty)\rightarrow H_{\mathcal{E}}^2(\mathbb{D}_2^\infty)$ is a bounded linear operator satisfying
 $$TM_{\xi_n}=M_{\zeta_n}T,\quad n\in\mathbb{N},$$
 then there exists an operator-valued function $\Psi\in H_{\mathcal{B}(\mathcal{F},\mathcal{E})}^\infty(\mathbb{D}_2^\infty)$
 such that $T=M_\Psi$.
\end{prop}
\begin{proof}
By (\ref{M_Psi*}), for each $n\in\mathbb{N}$, each $\lambda\in\mathbb{D}_2^\infty$ and each $x\in\mathcal{E}$,
$$M_{\xi_n}^*T^*(\mathbf{K}_\lambda\cdot x)=T^*M_{\zeta_n}^*(\mathbf{K}_\lambda\cdot x)
=\overline{\lambda_n}T^*(\mathbf{K}_\lambda\cdot x).$$
Then by Lemma \ref{K_lambda} (2), to any pair $\lambda,x$  there corresponds an element $y_{\lambda,x}\in\mathcal{F}$ such that $$T^*(\mathbf{K}_\lambda\cdot x)=\mathbf{K}_\lambda\cdot y_{\lambda,x},$$ and therefore $x\mapsto y_{\lambda,x}$ defines a linear operator $S_\lambda\in\mathcal{B}(\mathcal{E},\mathcal{F})$ for each $\lambda\in\mathbb{D}_2^\infty$.
Now put $\Psi(\lambda)=S_\lambda^*$ ($\lambda\in\mathbb{D}_2^\infty$).
It is routine to check that $\Psi:\mathbb{D}_2^\infty\rightarrow\mathcal{B}(\mathcal{F},\mathcal{E})$
is uniformly bounded and holomorphic with
$\|\Psi\|_\infty\leq\|T\|$. Since $T^*$ and $M_\Psi^*$ coincide on the set $\{\mathbf{K}_\lambda\cdot x:\lambda\in\mathbb{D}_2^\infty,x\in\mathcal{E}\}$ by (\ref{M_Psi*}),
  it follow from Lemma \ref{K_lambda} (1)
 that $T=M_\Psi$.
\end{proof}

In \cite{Be}, the infinite tensor product of a sequence of Hilbert space $\{\mathcal{H}_n\}_{n\in\mathbb{N}}$ with stabilizing sequence
$\{e^{(n)}\}_{n\in\mathbb{N}}$ is introduced, where
$e^{(n)}$ is a unit vector in $\mathcal{H}_n$ for each $n\in\mathbb{N}$.
We remark that
the Hilbert space structure of $H^2(\mathbb{D}_2^\infty)$ coincides with the infinite tensor product
$$H^2(\mathbb{D})\otimes H^2(\mathbb{D})\otimes\cdots$$ of  countably infinitely many Hardy spaces over $\mathbb{D}$ with stabilizing sequence
$\{1\}_{n\in\mathbb{N}}$.
So for any closed subspace $M$ of the Hardy space
$$H^2(\mathbb{D}^n)=\underbrace{H^2(\mathbb{D})\otimes\cdots\otimes H^2(\mathbb{D})}_{n\ \mathrm{times}}\quad(n\in\mathbb{N}),$$
 the infinite tensor product
$$M\otimes H^2(\mathbb{D})\otimes H^2(\mathbb{D})\otimes\cdots$$ is exactly the
$(M_{\zeta_{n+1}},M_{\zeta_{n+2}},\cdots)$-joint invariant subspace
of $H^2(\mathbb{D}_2^\infty)$ generated by $M(=M\otimes\mathbb{C}\otimes\mathbb{C}\cdots)$.
For later use, we also note that the vector-valued Hardy space $H_{\mathcal{E}}^2(\mathbb{D}_2^\infty)$ can be written as
\begin{equation} \label{representation of vector-valued Hardy space}
\begin{split}
   H_{\mathcal{E}}^2(\mathbb{D}_2^\infty) & =\underbrace{H^2(\mathbb{D})\otimes\cdots\otimes H^2(\mathbb{D})}_{n-1\ \mathrm{times}}\otimes H_{\mathcal{E}}^2(\mathbb{D}) \otimes H^2(\mathbb{D})\otimes H^2(\mathbb{D})\otimes\cdots \\
     & =H_{\mathcal{E}}^2(\mathbb{D}^n) \otimes \mathcal{L}_n,
\end{split}
\end{equation}
where \begin{equation*}
\begin{split}
   \mathcal{L}_n & =\underbrace{\mathbb{C}\otimes\cdots\otimes \mathbb{C}}_{n\ \mathrm{times}}\otimes H^2(\mathbb{D})\otimes H^2(\mathbb{D})\otimes\cdots \\
     & =\overline{\mathrm{span}\{\zeta^\alpha:\alpha=(\alpha_1,\alpha_2,\cdots)\in\mathbb{Z}_+^{(\infty)}\ \mathrm{with} \ \alpha_1=\cdots=\alpha_n=0\}}.
\end{split}
\end{equation*}
%is a closed subspace of $H^2(\mathbb{D}_2^\infty)$.

One  simplest class of vector-valued or operator-valued functions on $\mathbb{D}_2^\infty$ is the functions induced by those of one variable. More precisely, put $\tilde{f}(\zeta)=f(\zeta_n)$
and  $\widetilde{\theta}(\zeta)=\theta(\zeta_n)$  ($\zeta\in\mathbb{D}_2^\infty$) for any $f\in H_{\mathcal{E}}^2(\mathbb{D})$, any $\theta\in H_{\mathcal{B}(\mathcal{F},\mathcal{E})}^\infty(\mathbb{D})$
and any $n\in\mathbb{N}$. Then we have $\tilde{f}\in H_{\mathcal{E}}^2(\mathbb{D}_2^\infty)$ and $\widetilde{\theta}\in H_{\mathcal{B}(\mathcal{F},\mathcal{E})}^\infty(\mathbb{D}_2^\infty)$.
The multiplication operator $M_{\widetilde{\theta}}$ has the form
$$M_{\widetilde{\theta}}=\underbrace{I_{H^2(\mathbb{D})}\otimes\cdots\otimes I_{H^2(\mathbb{D})}}_{n-1\ \mathrm{times}}\otimes M_\theta \otimes I_{H^2(\mathbb{D})}\otimes\cdots$$
 with respect to the representation (\ref{representation of vector-valued Hardy space}).
The property of $M_{\widetilde{\theta}}$ thus relies heavily on that of $\theta$. For instance, $M_{\widetilde{\theta}}$ is isometric if and only if $M_\theta$ is isometric, if and only if $\theta$ is inner.
Recall that $\theta\in H_{\mathcal{B}(\mathcal{F},\mathcal{E})}^\infty(\mathbb{D})$ is said to be
inner if the boundary values of $\theta$,  on some subset $E$ of the unit circle $\mathbb{T}$ with full measure, are isometries,
which is defined to be the radial limit
$$\mathrm{(SOT)}\lim_{r\rightarrow1^-}\theta(rz),\quad z\in E$$
(see \cite{Ni1,SNFBK} for instance). On the other hand,  since $r\mathbb{D}^\infty\nsubseteqq\mathbb{D}_2^\infty$ for any $0<r<1$, the radial limits for  bounded holomorphic functions on $\mathbb{D}_2^\infty$
do not make sense in general. In \cite{SS}, E. Saksman and K. Seip defined boundary values for  bounded holomorphic functions on $\mathbb{D}_2^\infty$  by taking  quasi-radial limits instead.
However, to avoid more discussion about the boundary behavior of functions on $\mathbb{D}_2^\infty$, we introduce an alternative definition that $\Psi\in H_{\mathcal{B}(\mathcal{F},\mathcal{E})}^\infty(\mathbb{D}_2^\infty)$
is said to be \textit{inner} if $M_\Psi$ is an isometry.
Finally, applying Lemma \ref{K_lambda} (1) to the function $\tilde{f}$ and (\ref{M_Psi*}) to the function $\widetilde{\theta}$, we have
\begin{itemize}
  \item [(i)] $\langle f,K_a\cdot x\rangle=\langle f(a),x\rangle$
for any $a\in\mathbb{D}$ and $x\in\mathcal{E}$. Consequently,
the set $$\{K_a\cdot x:a\in\mathbb{D},x\in\mathcal{E}\}$$ is complete in $H_{\mathcal{E}}^2(\mathbb{D})$.
  \item [(ii)] $M_\theta^*(K_a\cdot x)=K_a\cdot\theta(a)^*x$ for any $a\in\mathbb{D}$ and any $x\in\mathcal{E}$.
\end{itemize}
Here, $K_a$ denotes the reproducing kernel of $H^2(\mathbb{D})$ at the point $a\in\mathbb{D}$. Note that by (ii), for  $a,b\in\mathbb{D}$ and $x,y\in\mathcal{E}$, we have \begin{equation} \label{a calculation}
           \begin{split}
              \langle M_\theta M_\theta^*(K_a\cdot x),K_b\cdot y\rangle & =\langle  M_\theta^*(K_a\cdot x),M_\theta^*(K_b\cdot y)\rangle \\
                & =\langle K_a\cdot\theta(a)^*x,K_b\cdot\theta(b)^*y\rangle \\
                & =\frac{1}{1-\bar{a}b}\langle \theta(b)\theta(a)^*x,y\rangle.
           \end{split}
         \end{equation}

The following lemma is also needed in the sequel.

\begin{lem} \label{operator-valued in one variable} Suppose $\theta\in H_{\mathcal{B}(\mathcal{F},\mathcal{E})}^\infty(\mathbb{D})$ and
$\vartheta\in H_{\mathcal{B}(\mathcal{G},\mathcal{E})}^\infty(\mathbb{D})$.
Then the following conclusions hold:
\begin{itemize}
  \item [(1)] Let $\mathcal{E}_0$ be a closed subspace of $\mathcal{E}$. Then $H_{\mathcal{E}_0}^2(\mathbb{D})$ is reducing for $M_\theta M_\theta^*$ if and only if  $\mathcal{E}_0$ is invariant for
$\theta(b)\theta(a)^*$ for any $a,b\in\mathbb{D}$. In this case, if $\theta$ is inner we have
$$M_\theta^*H_{\mathcal{E}_0}^2(\mathbb{D})=H_{\mathcal{F}_0}^2(\mathbb{D}),$$
where $\mathcal{F}_0=\bigvee_{a\in\mathbb{D}}\theta(a)^*\mathcal{E}_0$.
  \item [(2)]
$M_\theta M_\theta^*=M_\vartheta M_\vartheta^*$ if and only if $$\theta(b)\theta(a)^*=\vartheta(b)\vartheta(a)^*$$ for any $a,b\in\mathbb{D}$.
\end{itemize}
\end{lem}
\begin{proof}
(1) By (i), $H_{\mathcal{E}_0}^2(\mathbb{D})$ is reducing for $M_\theta M_\theta^*$ if and only if
$$\langle M_\theta M_\theta^*(K_a\cdot x),K_b\cdot y\rangle=0$$ for any $a,b\in\mathbb{D}$, any $x\in\mathcal{E}_0$ and any $y\in\mathcal{E}\ominus\mathcal{E}_0$.
This together with (\ref{a calculation}) gives that $H_{\mathcal{E}_0}^2(\mathbb{D})$ is reducing for $M_\theta M_\theta^*$ if and only if $\theta(b)\theta(a)^*\mathcal{E}_0\subseteq\mathcal{E}_0$
for any $a,b\in\mathbb{D}$.
For the latter conclusion, note that
by (ii), $M_\theta^*H_{\mathcal{E}_0}^2(\mathbb{D})$ is a dense subspace of $H_{\mathcal{F}_0}^2(\mathbb{D})$. Now suppose in addition $\theta$ is inner. Then $M_\theta M_\theta^*$ is an orthogonal projection on $H_{\mathcal{E}}^2(\mathbb{D})$. Since $H_{\mathcal{E}_0}^2(\mathbb{D})$ is reducing for $M_\theta M_\theta^*$, $M_\theta M_\theta^*H_{\mathcal{E}_0}^2(\mathbb{D})$ is closed,
forcing $M_\theta^*H_{\mathcal{E}_0}^2(\mathbb{D})$ to be also closed. This proves (1).

\vskip2mm

(2) By (i), $M_\theta M_\theta^*=M_\vartheta M_\vartheta^*$ if and only if
$$\langle M_\theta M_\theta^*(K_a\cdot x),K_b\cdot y\rangle=\langle M_\vartheta M_\vartheta^*(K_a\cdot x),K_b\cdot y\rangle$$ for any $a,b\in\mathbb{D}$ and any $x,y\in\mathcal{E}$.
Then (2) follows from (\ref{a calculation}).
\end{proof}

\subsection{Some preparations for proofs}

\begin{lem} \label{dilation of defect op} If $\mathbf{T}$ is a doubly commuting sequence of contractions on $\mathcal{H}$ and $\mathbf{V}$ is
a doubly commuting isometric coextension of $\mathbf{T}$, then for each $\lambda\in\mathbb{D}^\infty$ and each $x\in \mathcal{H}$,
$$\|D_{\Phi_\lambda(\mathbf{T})^*}x\|=\|D_{\Phi_\lambda(\mathbf{V})^*}x\|.$$
% $D_{\Phi_{\lambda}(V)^*}^2$ is a dilation of $D_{\Phi_{\lambda}(T)^*}^2$ for any $\lambda\in\mathbb{D}^\infty$, that is,  $$D_{\Phi_{\lambda}(T)^*}^2=P_\mathcal{H} D_{\Phi_{\lambda}(V)^*}^2|_\mathcal{H}.$$
\end{lem}
\begin{proof} %Suppose that the sequence $\mathbf{V}$ is defined on a Hilbert space $\mathcal{K}$ containing $\mathcal{H}$.
For any fixed $\lambda\in\mathbb{D}^\infty$, set $$\mathbf{S}=(S_1,S_2,\cdots)=\Phi_\lambda(\mathbf{T})$$ and $$\mathbf{W}=(W_1,W_2,\cdots)=\Phi_\lambda(\mathbf{V}).$$
It is clear that
$\mathbf{W}$ is an isometric coextension of $\mathbf{S}$.
  Since for $x\in\mathcal{H}$,
$$\|D_{\Phi_\lambda(\mathbf{T})^*}x\|=\|D_{\mathbf{S}^*}x\|=\lim_{n\rightarrow\infty}
\|D_{S_1^*}\cdots D_{S_n^*}x\|$$ and $$\|D_{\Phi_\lambda(\mathbf{V})^*}x\|=\|D_{\mathbf{W}^*}x\|=\lim_{n\rightarrow\infty}
\|D_{W_1^*}\cdots D_{W_n^*}x\|,$$
it suffices to prove that for  $n\in\mathbb{N}$,
\begin{equation}\label{dilation identity in finite-tuple case}
 \|D_{S_1^*}\cdots D_{S_n^*}x\|=\|D_{W_1^*}\cdots D_{W_n^*}x\|,\quad x\in\mathcal{H}.
\end{equation}
We show this by induction on $n$.
For $n=1$, one has
$$\|D_{S_1^*}x\|^2=\|x\|^2-\|S_1^*x\|^2
=\|x\|^2-\|W_1^*x\|^2=\|D_{W_1^*}x\|^2,\quad x\in\mathcal{H}.$$
Assume that (\ref{dilation identity in finite-tuple case}) holds for $n=k$. Since $\mathbf{S}$ and $\mathbf{W}$ are doubly commuting, we have that for $x\in\mathcal{H}$,
\begin{equation*}
\begin{split}
   \|D_{S_1^*}\cdots D_{S_{k+1}^*}x\|^2
     & = \|D_{S_1^*}\cdots D_{S_k^*}x\|^2
     - \|S_{k+1}^*D_{S_1^*}\cdots D_{S_k^*}x\|^2 \\
     & = \|D_{S_1^*}\cdots D_{S_k^*}x\|^2
     - \|D_{S_1^*}\cdots D_{S_k^*}S_{k+1}^*x\|^2 \\
     & = \|D_{W_1^*}\cdots D_{W_k^*}x\|^2
     - \|D_{W_1^*}\cdots D_{W_k^*}S_{k+1}^*x\|^2 \\
     & = \|D_{W_1^*}\cdots D_{W_k^*}x\|^2
     - \|D_{W_1^*}\cdots D_{W_k^*}W_{k+1}^*x\|^2 \\
     & = \|D_{W_1^*}\cdots D_{W_k^*}x\|^2
     - \|W_{k+1}^*D_{W_1^*}\cdots D_{W_k^*}x\|^2 \\
     & = \|D_{W_1^*}\cdots D_{W_{k+1}^*}x\|^2.
\end{split}
\end{equation*}
Thus, (\ref{dilation identity in finite-tuple case}) also holds for $n=k+1$. This completes the proof.
\end{proof}

\begin{lem}\label{intersection of Ker}
If $(T_1,\cdots,T_n)$ is a doubly commuting $n$-tuple of $C_{.0}$ contractions on a Hilbert space $\mathcal{H}$, then $$\bigcap_{(\lambda_1,\cdots,\lambda_n)\in\mathbb{D}^n} \mathrm{Ker}\ D_{\varphi_{\lambda_1}(T_1)^*}\cdots D_{\varphi_{\lambda_n}(T_n)^*}=\{0\}.$$
Equivalently,
$$\bigvee_{(\lambda_1,\cdots,\lambda_n)\in\mathbb{D}^n} D_{\varphi_{\lambda_1}(T_1)^*}\cdots D_{\varphi_{\lambda_n}(T_n)^*}\mathcal{H}=\mathcal{H}.$$
\end{lem}
\begin{proof}
The equivalence is
guaranteed by
the fact that $(T_1,\cdots,T_n)$ is  doubly commuting and then the operator $$D_{\varphi_{\lambda_1}(T_1)^*}\cdots D_{\varphi_{\lambda_n}(T_n)^*}$$ is self-adjoint.
We prove this lemma by induction on $n$. For $n=1$, let $T$ be a  contractions of class $C_{.0}$, and assume $f\in\mathcal{H}$ so that $D_{\varphi_a(T)^*}f=0$ for any $a\in\mathbb{D}$. As in the proof of Lemma \ref{min iso dilation}, $T$ has a coextension to the Hardy shift $M_z$ on the vector-valued Hardy space $H_{\mathfrak{D}_{T^*}}^2(\mathbb{D})$.
Since $M_{\varphi_a}\ (a\in\mathbb{D})$ is isometric on $H_{\mathfrak{D}_{T^*}}^2(\mathbb{D})$,
 we have
$$\|f\|=\|\varphi_a(T)^*f\|=\|\varphi_a(M_z)^*f\|=\|M_{\varphi_a}^*f\|=\|M_{\varphi_a}M_{\varphi_a}^*f\|.$$
Note that $M_{\varphi_a}M_{\varphi_a}^*$ is the orthogonal projection onto $\mathrm{Ran}\ M_{\varphi_a}$ for each $a\in\mathbb{D}$, the above identity yields
$f=M_{\varphi_a}g_a=\varphi_a\cdot g_a$ for some $g_a\in H_{\mathfrak{D}_{T^*}}^2(\mathbb{D})$. This gives  $$f(a)=\varphi_a(a)\cdot g_a(a)=0,\quad a\in\mathbb{D},$$ forcing $f=0$, which proves the case $n=1$. Now assume that the conclusion holds for $n=k$, and let $(T_1,\cdots,T_{k+1})$ be a doubly commuting tuple of $C_{.0}$ contractions. Then
\begin{align*}
    & \bigvee_{(\lambda_1,\cdots,\lambda_{k+1})\in\mathbb{D}^{k+1}} D_{\varphi_{\lambda_1}(T_1)^*}\cdots D_{\varphi_{\lambda_{k+1}}(T_{k+1})^*}\mathcal{H} \\ ={} & \bigvee_{(\lambda_1,\cdots,\lambda_{k})\in\mathbb{D}^{k}}
     \left(\bigvee_{a\in\mathbb{D}} D_{\varphi_{\lambda_1}(T_1)^*}\cdots D_{\varphi_{\lambda_k}(T_k)^*} D_{\varphi_{\lambda_a}(T_{k+1})^*}\mathcal{H}\right) \\
       ={} & \bigvee_{(\lambda_1,\cdots,\lambda_{k})\in\mathbb{D}^{k}}
     D_{\varphi_{\lambda_1}(T_1)^*}\cdots D_{\varphi_{\lambda_k}(T_k)^*}\left(\bigvee_{a\in\mathbb{D}} D_{\varphi_{\lambda_a}(T_{k+1})^*}\mathcal{H}\right) \\
       ={} & \bigvee_{(\lambda_1,\cdots,\lambda_{k})\in\mathbb{D}^{k}}
     D_{\varphi_{\lambda_1}(T_1)^*}\cdots D_{\varphi_{\lambda_k}(T_k)^*}\mathcal{H} \\
       ={} & \mathcal{H}.
\end{align*}
By induction, the proof is complete.
\end{proof}

Suppose that $\mathbf{V}=(V_1,V_2,\cdots)$ is a sequence in class $\mathcal{DP}$. Recall that for any $\lambda=(\lambda_1,\lambda_2,\cdots)\in\mathbb{D}^\infty$, the defect space $\mathfrak{D}_{\Phi_\lambda(\mathbf{V})^*} $ of the sequence $\Phi_\lambda(\mathbf{V})^*$ is the closure of the range of the defect operator $$D_{\Phi_\lambda(\mathbf{V})^*}=
\prod_{n=1}^{\infty}(I-\varphi_{\lambda_n}(V_n)\varphi_{\lambda_n}(V_n)^*).$$
Since $\mathbf{V}$ is doubly commuting, $\{I-\varphi_{\lambda_n}(V_n)\varphi_{\lambda_n}(V_n)^*\}_{n\in\mathbb{N}}$ is a commuting sequence of  orthogonal projections,
and then $D_{\Phi_\lambda(\mathbf{V})^*}$ is also an orthogonal projections onto the subspace $$\bigcap_{n=1}^{\infty}\mathrm{Ran}\
(I-\varphi_{\lambda_n}(V_n)\varphi_{\lambda_n}(V_n)^*)
=\bigcap_{n=1}^{\infty}(\mathrm{Ran}\ \varphi_{\lambda_n}(V_n))^\perp
=\bigcap_{n=1}^{\infty}\mathrm{Ker}\ \varphi_{\lambda_n}(V_n)^*.$$
Note that $x\in\mathrm{Ker}\ \varphi_{\lambda_n}(V_n)^*$ if and only if
$V_n^*x=\overline{\lambda_n}x\ (n\in\mathbb{N})$, nonzero elements in defect space $\mathfrak{D}_{\Phi_\lambda(\mathbf{V})^*} $ exactly coincide with the set of joint eigenvectors of the sequence $\mathbf{V}^*$ corresponding to the joint eigenvalue $\overline{\lambda}=(\overline{\lambda_1},\overline{\lambda_2},\cdots)$.

\begin{lem} \label{local chara of K_gamma}
Suppose $\mathbf{V}\in\mathcal{DP}$. If $x$ is a nonzero element in the defect space $\mathfrak{D}_{\Phi_\lambda(\mathbf{V})^*} $ of the sequence $\Phi_\lambda(\mathbf{V})^*$ for some $\lambda\in\mathbb{D}^\infty$, then
$[x]_{\mathbf{V}}(=[\{x\}]_{\mathbf{V}})$ is $\mathbf{V}$-joint reducing, and $\mathbf{V}|_{[x]_{\mathbf{V}}}$ is jointly unitarily equivalent to the sequence $\Phi_\lambda(\mathbf{M_\zeta})$, where $\mathbf{M_\zeta}$ is the tuple of coordinate multiplication operators on the Hardy space $H^2(\mathbb{D}_2^\infty)$.
\end{lem}
\begin{proof} Assume $\lambda=\mathbf{0}=(0,0,\cdots)$ and $\|x\|=1$ without loss of generality.
 Since $x\in\mathfrak{D}_{\mathbf{V}^*}$, we have that for $n, m\in\mathbb{N}$ and $\alpha\in\mathbb{Z}_+^{(\infty)}$,
$$V_m^*\mathbf{V}^\alpha x
=\mathbf{V}^\alpha V_m^*x=0\in[x]_{\mathbf{V}},\quad\mathrm{if}\ \alpha_m=0;$$
$$V_m^*\mathbf{V}^\alpha x
=\mathbf{V}^{\alpha-1_m}x\in[x]_{\mathbf{V}},\quad\mathrm{if}\ \alpha_m\geq1.$$
This implies that $[x]_{\mathbf{V}}$ is $\mathbf{V}$-joint reducing.
The rest of proof is given by defining a linear map $U$ from $\mathcal{P}_\infty$ to $[x]_{\mathbf{V}}$ as follows:
$$Up=p(\mathbf{V})x,\quad p\in\mathcal{P}_\infty,$$
where $\mathcal{P}_\infty=\mathbb{C}[\zeta_1,\zeta_2,\cdots]$, the polynomial ring in countably infinitely many variables, which is dense in $H^2(\mathbb{D}_2^\infty)$ \cite{Ni2}. It is routine to check that $U$ can be extend to a unitary operator from $H^2(\mathbb{D}_2^\infty)$ onto $[x]_{\mathbf{V}}$, and the tuple $\mathbf{M_\zeta}$ of coordinate multiplication operators
is jointly unitarily equivalent to $\mathbf{V}|_{[x]_{\mathbf{V}}}$ via this unitary operator.
\end{proof}

\begin{lem} \label{wand subsp is 0}
Let $\lambda$ be a point in $\mathbb{D}^\infty$ and $\mathbf{M}_\zeta$  the tuple of coordinate multiplication operators on $H^2(\mathbb{D}^\infty_2)$. Then $\mathfrak{D}_{\Phi_\lambda(\mathbf{M}_\zeta)^*}=\{0\}$ if and only if $\lambda\notin\mathbb{D}_2^\infty$.
 %then the wandering subspace $\mathfrak{D}_{\Phi_\lambda(\mathbf{M}_\zeta)^*}$ of the sequence $$\Phi_\lambda(\mathbf{M}_\zeta)^*=(M_{\widetilde{\varphi_{\lambda_1}}}^*,M_{\widetilde{\varphi_{\lambda_2}}}^*,\cdots)$$ on $H^2(\mathbb{D}_2^\infty)$ is $0$, where $\widetilde{\varphi_{\lambda_n}}(\zeta)=\varphi_{\lambda_n}(\zeta_n)\ (n\in\mathbb{N})$.
\end{lem}
\begin{proof} %By definition, the wandering subspace $\mathfrak{D}_{\Phi_\lambda(\mathbf{M}_\zeta)^*}$ is the range of the defect operator
%  $$D_{\Phi_\lambda(\mathbf{M_\zeta})^*}=\prod_{n=1}^{\infty}(I-M_{\widetilde{\varphi_{\lambda_n}}}M_{\widetilde{\varphi_{\lambda_n}}}^*),$$
%of the sequence $\Phi_\lambda(\mathbf{M}_\zeta)^*$, where $\widetilde{\varphi_{\lambda_n}}(\zeta)=\varphi_{\lambda_n}(\zeta_n)\ (n\in\mathbb{N})$. Since $\{I-M_{\widetilde{\varphi_{\lambda_n}}}M_{\widetilde{\varphi_{\lambda_n}}}^*\}_{n\in\mathbb{N}}$
%is a commuting sequence of orthogonal projections on $H^2(\mathbb{D}_2^\infty)$,
By comments before the previous lemma, we have $$\mathfrak{D}_{\Phi_\lambda(\mathbf{M}_\zeta)^*}
=\bigcap_{n=1}^{\infty}
(\widetilde{\varphi_{\lambda_n}}H^2(\mathbb{D}_2^\infty))^\perp,$$
and hence
$$\mathfrak{D}_{\Phi_\lambda(\mathbf{M}_\zeta)^*}^\perp
=\bigvee_{n=1}^\infty\widetilde{\varphi_{\lambda_n}}H^2(\mathbb{D}_2^\infty)
=[\{\widetilde{\varphi_{\lambda_n}}\}_{n\in\mathbb{N}}]_{\mathbf{M}_\zeta},$$
where $\widetilde{\varphi_{\lambda_n}}(\zeta)=\varphi_{\lambda_n}(\zeta_n)\ (\zeta\in\mathbb{D}^\infty_2)$.
 By \cite[Proposition 4.5]{DGH},
 $$[\{\widetilde{\varphi_{\lambda_n}}\}_{n\in\mathbb{N}}]_{\mathbf{M}_\zeta}=H^2(\mathbb{D}_2^\infty)$$
if and only if $\lambda\notin\mathbb{D}_2^\infty$.
This completes the proof.
\end{proof}

By an irreducible family of operators we mean that these operators has no nontrivial joint reducing subspaces.

\begin{lem}\label{M_zeta is irre}
The tuple $\mathbf{M}_\zeta$ of coordinate multiplication operators on $H^2(\mathbb{D}^\infty_2)$ is irreducible.
\end{lem}
\begin{proof} It suffices to prove that for any orthogonal projection $P$ that commutes with every coordinate multiplication operator, we have $P1=1$ or $Q1=1$, where
$Q=I-P$. Put $Q1=\sum_{\alpha\in\mathbb{Z}_+^{(\infty)}} c_\alpha\zeta^\alpha$.
Since $$P\zeta^\alpha=PM_{\zeta^\alpha}1=M_{\zeta^\alpha}P1=\zeta^\alpha\cdot P1,$$ we have
 $$0=P(Q1)=P(\sum_{\alpha\in\mathbb{Z}_+^{(\infty)}} c_\alpha\zeta^\alpha)
=\sum_{\alpha\in\mathbb{Z}_+^{(\infty)}} c_\alpha P\zeta^\alpha=\sum_{\alpha\in\mathbb{Z}_+^{(\infty)}} (c_\alpha\zeta^\alpha\cdot P1)=Q1\cdot P1.$$
Note that $Q1=(I-P)1=1-P1$, the proof is complete.
\end{proof}

\begin{lem} \label{redu in tensor} Let $\mathcal{H}$ and $\mathcal{K}$ be two Hilbert spaces, and $\mathcal{T}$  an irreducible family of bounded linear operators on $\mathcal{H}$. Then
any bounded linear operator on $\mathcal{H}\otimes \mathcal{K}$ that doubly commutes with the family $\{T\otimes I_\mathcal{K}:T\in\mathcal{T}\}$ is of form
$I_\mathcal{H}\otimes S$ for some $S\in\mathcal{B}(\mathcal{K})$.

In particular,
any  joint reducing subspace  for the family $\{T\otimes I_\mathcal{K}:T\in\mathcal{T}\}$  is of form $\mathcal{H}\otimes M$ for some closed subspace $M$ of $\mathcal{K}$.
\end{lem}
\begin{proof} Since the family $\mathcal{T}$ of operators has no nontrivial joint reducing subspace, the von Neumann algebra generated by $\mathcal{T}$ is the entire $\mathcal{B}(\mathcal{H})$ by the double commutant theorem \cite{Con}. This gives
$$\mathcal{V}^*(\{T\otimes I_\mathcal{K}:T\in\mathcal{T}\})=\mathcal{B}(\mathcal{H})\otimes \mathbb{C}I_\mathcal{K}=\{T\otimes I_\mathcal{K}:T\in\mathcal{B}(\mathcal{H})\},$$
where $\mathcal{V}^*(\{T\otimes I_\mathcal{K}:T\in\mathcal{T}\})$ denotes the von Neumann algebra generated by the family $\{T\otimes I_\mathcal{K}:T\in\mathcal{T}\}$. If a bounded linear operator on $\mathcal{H}\otimes \mathcal{K}$  doubly commutes with the family $\{T\otimes I_\mathcal{K}:T\in\mathcal{T}\}$, then it also commutes with the algebra $\mathcal{B}(\mathcal{H})\otimes \mathbb{C}I_\mathcal{K}$, and therefore has  form $I_\mathcal{H}\otimes S$ for some bounded linear operator $S$ on $\mathcal{K}$ (see \cite[pp. 184]{Ta}).
\end{proof}

\begin{lem} \label{estimate lemma} Suppose $\mathbf{V}\in\mathcal{DP}(\mathcal{H})$. Then for any $\varepsilon>0$ and $x\in\mathcal{H}$, there exists a sequence $(k_1,k_2,\cdots)$ of positive integers such that $$\|D_{\mathbf{W}^*}x\|\geq(1-\varepsilon)\|x\|,$$
where $\mathbf{W}=(V_1^{k_1},V_2^{k_2},\cdots)$.
\end{lem}
\begin{proof} Take an arbitrary positive number $\varepsilon$ and assume $\|x\|=1$ without loss of generality. Since each  $V_n\ (n\in\mathbb{N})$ is pure,  there exists a sequence $(k_1,k_2,\cdots)$ of positive integers, such that \begin{equation}\label{estimate for x}
\|V_n^{*k_n}x\|<\frac{1}{2^n}\varepsilon,\quad n\in\mathbb{N}.
\end{equation}
Rewrite $(W_1,W_2,\cdots)=(V_1^{k_1},V_2^{k_2},\cdots)$.
For each $n\in\mathbb{N}$, we have
\begin{align*}
  & \|(I-W_1W_1^*)\cdots(I-W_nW_n^*)x\| \\ \geq {} & \|(I-W_2W_2^*)\cdots(I-W_nW_n^*)x\|
   -\|W_1W_1^*(I-W_2W_2^*)\cdots(I-W_nW_n^*)x\| \\
     \geq {} & \|(I-W_2W_2^*)\cdots(I-W_nW_n^*)x\|
     -\|W_1(I-W_2W_2^*)\cdots(I-W_nW_n^*)W_1^*x\| \\
     \geq {} & \|(I-W_2W_2^*)\cdots(I-W_nW_n^*)x\|
     -\|W_1^*x\|,
\end{align*}
and then by induction and (\ref{estimate for x}),
$$\|(I-W_1W_1^*)\cdots(I-W_nW_n^*)x\|\geq\|x\|-\sum_{i=1}^{n}\|W_i^*x\|
>1-(1-\frac{1}{2^n})\varepsilon.$$
Setting $n\rightarrow\infty$ in the above inequality, we obtain the desired conclusion.
\end{proof}

\section{Dilation theory}
In this section, we will give some operator-theoretical characterization for sequences in class $\mathcal{DC}$.
Recall that a sequence $\mathbf{T}\in\mathcal{DC}(\mathcal{H})$ is said to has a decomposition
of quasi-Beurling type if there exists an orthogonal decomposition $\mathcal{H}=\bigoplus_\gamma \mathcal{H}_\gamma$ of the Hilbert space $\mathcal{H}$, such that each $\mathcal{H}_\gamma$ is $\mathbf{T}$-joint reducing and $\mathbf{T}|_{\mathcal{H}_\gamma}$ is of quasi-Beurling type.

The main results in this section are restated as follows.
\vskip2mm
\noindent\textbf{Theorem \ref{direct sums of quasi-BT}} \emph{Suppose $\mathbf{T}\in\mathcal{DC}(\mathcal{H})$. The followings are equivalent:
\begin{itemize}
  \item [(1)] $\mathbf{T}$ has a decomposition
of quasi-Beurling type;
  \item [(2)] for each $x\in \mathcal{H}$, there exists a sequence $\lambda^{(1)},\lambda^{(2)},\cdots$ of points in $\mathbb{D}^\infty$, such that
      $$\|x\|^2=\sum_{k=1}^\infty\sum_{\alpha\in\mathbb{Z}_+^{(\infty)}}
      \|D_{\Phi_{\lambda^{(k)}}(\mathbf{T})^*}\Phi_{\lambda^{(k)}}^\alpha(\mathbf{T})^{*}x\|^2;$$
  \item [(3)] $\bigvee_{\lambda\in\mathbb{D}^\infty}
\mathfrak{D}_{\Phi_\lambda(\mathbf{T})^*}=\mathcal{H}$.
\end{itemize}
In fact,
%if in addition the Hilbert space $\mathcal{H}$ is separable, then
the sequence $\lambda^{(1)},\lambda^{(2)},\cdots$ in condition (2) can be chosen to be independent of $x\in \mathcal{H}$.}

\vskip2mm

\noindent\textbf{Corollary \ref{BT}} \emph{Suppose $\mathbf{T}\in\mathcal{DC}(\mathcal{H})$. The followings are equivalent:
\begin{itemize}
  \item [(1)] $\mathbf{T}$ is of Beurling type;
  \item [(2)] the minimal regular isometric dilation of $\mathbf{T}$ is jointly unitarily equivalent to the tuple $\mathbf{M}_\zeta$ of coordinate multiplication operators on a vector-valued Hardy space $H_{\mathcal{E}}^2(\mathbb{D}_2^\infty)$;
  \item [(3)] for each $x\in \mathcal{H}$, $\|x\|^2=\sum_{\alpha\in\mathbb{Z}_+^{(\infty)}}\|D_{\mathbf{T}^*}\mathbf{T}^{*\alpha}x\|^2$;
  \item [(4)] $\bigvee_{\lambda\in\mathbb{D}_2^\infty}
\mathfrak{D}_{\Phi_\lambda(\mathbf{T})^*}=\mathcal{H}$.
\end{itemize}}
\vskip2mm

\noindent\textbf{Proof of Theorem \ref{direct sums of quasi-BT}.}
(1)$\Rightarrow$(2). Suppose that there exists an (countable) orthogonal decomposition $\mathcal{H}=\bigoplus_\gamma \mathcal{H}_\gamma$ of the Hilbert space $\mathcal{H}$, such that each $\mathcal{H}_\gamma$ is $\mathbf{T}$-joint reducing and $\mathbf{T}|_{\mathcal{H}_\gamma}$ is of quasi-Beurling type. Take an arbitrary element $x$ in $\mathcal{H}$, and let $x_\gamma$ be the orthogonal projection of $x$ into the subspace $\mathcal{H}_\gamma$. Then
%%there are at most countably many indexes $\gamma$ such that $x_\gamma\neq0$ and
%$\|x\|^2=\sum_{x_\gamma\neq0}\|x_\gamma\|^2$.
$\|x\|^2=\sum_{\gamma}\|x_\gamma\|^2$.
The implication (1)$\Rightarrow$(2) is proved once we show
that for each $\gamma$, there corresponds a point $\lambda_\gamma\in\mathbb{D}^\infty$ such that
\begin{equation} \label{norm of x_gamma} \|x_\gamma\|^2=\sum_{\alpha\in\mathbb{Z}_+^{(\infty)}}
\|D_{\Phi_{\lambda_\gamma}(\mathbf{T})^*}
\Phi_{\lambda_\gamma}^\alpha(\mathbf{T})^{*}x\|^2.\end{equation}
Since each $\mathcal{H}_\gamma$ is $\mathbf{T}$-joint reducing,  we see $\mathbf{T}|_{\mathcal{H}_\gamma}\in\mathcal{DC}(\mathcal{H}_\gamma)$.
Fix an index $\gamma$ and let $\mathbf{V}\in\mathcal{DP}(\mathcal{K})$ be the minimal regular isometric dilation of $\mathbf{T}|_{\mathcal{H}_\gamma}$.
Then by Lemma \ref{min iso dilation}, $\mathbf{V}$ is a coextension of $\mathbf{T}|_{\mathcal{H}_\gamma}$. Since $\mathbf{T}|_{\mathcal{H}_\gamma}$ is of quasi-Beurling type, there is a point $\lambda\in\mathbb{D}^\infty$ such that $\mathbf{W}=(W_1,W_2,\cdots)=\Phi_\lambda(\mathbf{V})$ is of Beurling type; that is, $[\mathfrak{D}_{\mathbf{W}^*}]_{\mathbf{W}}=\mathcal{K}$.
Rewrite $\mathcal{E}=\mathfrak{D}_{\mathbf{W}^*}$. Since $\mathbf{W}$ is doubly commuting and $\mathcal{E}=\bigcap_{n=1}^\infty\mathrm{Ker}\ W_n^*$, we have that the family $\{\mathbf{W}^\alpha\mathcal{E}\}
_{\alpha\in\mathbb{Z}_+^{(\infty)}}$ of subspaces are pairwise orthogonal, and therefore $$\mathcal{K}=\bigoplus_
{\alpha\in\mathbb{Z}_+^{(\infty)}}
\mathbf{W}^\alpha\mathcal{E}.$$
It is easy to see that for each $\alpha\in\mathbb{Z}_+^{(\infty)}$, the operator $\mathbf{W}^\alpha D_{\mathbf{W}^*}\mathbf{W}^{*\alpha}$ is exactly the orthogonal projection onto $\mathbf{W}^\alpha\mathcal{E}$,
and then
$$
 \|x_\gamma\|^2  =\sum_{\alpha\in\mathbb{Z}_+^{(\infty)}}
   \|\mathbf{W}^\alpha D_{\mathbf{W}^*}\mathbf{W}^{*\alpha}x_\gamma\|^2 = \sum_{\alpha\in\mathbb{Z}_+^{(\infty)}}
   \|D_{\mathbf{W}^*}\mathbf{W}^{*\alpha}x_\gamma\|^2.
$$ Also note that
$\mathbf{W}$ is an isometrical coextension of $\Phi_\lambda(\mathbf{T})$,
  by Lemma \ref{dilation of defect op}  we have
\begin{equation*}
\begin{split}
   \|x_\gamma\|^2 %& =\sum_{\alpha\in\mathbb{Z}_+^{(\infty)}}
%   \|\mathbf{W}^\alpha D_{\mathbf{W}^*}\mathbf{W}^{*\alpha}x_\gamma\|^2 \\
    & = \sum_{\alpha\in\mathbb{Z}_+^{(\infty)}}
   \|D_{\mathbf{W}^*}\mathbf{W}^{*\alpha}x_\gamma\|^2 \\ & =
   \sum_{\alpha\in\mathbb{Z}_+^{(\infty)}}
   \|D_{\mathbf{W}^*}\Phi_\lambda^\alpha(\mathbf{T}|_{\mathcal{H}_\gamma})
   ^{*}x_\gamma\|^2 \\
   &=\sum_{\alpha\in\mathbb{Z}_+^{(\infty)}}\|D_{\Phi_\lambda(\mathbf{T}|_{\mathcal{H}_\gamma})^*}
   \Phi_\lambda^\alpha(\mathbf{T}|_{\mathcal{H}_\gamma})^{*}x_\gamma\|^2\\
     &=\sum_{\alpha\in\mathbb{Z}_+^{(\infty)}}\|D_{\Phi_\lambda(\mathbf{T})^*}|_{\mathcal{H}_\gamma}
     \Phi_\lambda^\alpha(\mathbf{T})^{*}|_{\mathcal{H}_\gamma}x_\gamma\|^2\\
     &=\sum_{\alpha\in\mathbb{Z}_+^{(\infty)}}\|D_{\Phi_\lambda(\mathbf{T})^*}\Phi_\lambda^\alpha(\mathbf{T})^{*}x_\gamma\|^2.
\end{split}
\end{equation*}

%If in addition $\mathcal{H}$ is separable, then
%Since there at most countable  direct summands in the orthogonal decomposition $\mathcal{H}=\oplus_\gamma \mathcal{H}_\gamma$.
It is clear that
 the point $\lambda_\gamma$ that appears in (\ref{norm of x_gamma}) only depend the subspace $\mathcal{H}_\gamma$, and hence not on the choice of $x$.

\vskip2mm

%Since each $\mathcal{H}_\gamma$ is $\mathbf{T}$-joint reducing and $\mathbf{T}|_{\mathcal{H}_\gamma}$ is of quasi-Beurling type for each $\gamma$, by Proposition \ref{DC of quasi-BT},
%\begin{equation}\label{x_gamma}
%\begin{split}
%    \|x_\gamma\|^2 &=\sum_{\alpha\in\mathbb{Z}_+^{(\infty)}}\|D_{\Phi_{\lambda_\gamma}(\mathbf{T}|_{\mathcal{H}_\gamma})^*}\Phi_{\lambda_\gamma}(\mathbf{T}|_{\mathcal{H}_\gamma})^{*\alpha}x\|^2\\
%     &=\sum_{\alpha\in\mathbb{Z}_+^{(\infty)}}\|D_{\Phi_{\lambda_\gamma}(\mathbf{T})^*}|_{\mathcal{H}_\gamma}\Phi_{\lambda_\gamma}^\alpha(\mathbf{T})^{*}|_{\mathcal{H}_\gamma}x\|^2\\
%     &=\sum_{\alpha\in\mathbb{Z}_+^{(\infty)}}\|D_{\Phi_{\lambda_\gamma}(\mathbf{T})^*}\Phi_{\lambda_\gamma}^\alpha(\mathbf{T})^{*}x\|^2.
%\end{split}
%\end{equation}
%      for some ${\lambda_\gamma}\in\mathbb{D}^\infty$.
%Note that there are at most countably many indexes $\gamma$ such that $x_\gamma\neq0$ and $\|x\|^2=\sum_{x_\gamma\neq0}\|x_\gamma\|^2$. This together with (\ref{x_gamma}) proves .
%

(2)$\Rightarrow$(3). Fix $x\in \mathcal{H}\ominus\left(\bigvee_{\lambda\in\mathbb{D}^\infty}
\mathfrak{D}_{\Phi_\lambda(\mathbf{T})^*}\right)$. We will show that for any given $\mu=(\mu_1,\mu_2,\cdots)\in\mathbb{D}^\infty$
and $\alpha=(\alpha_1,\cdots,\alpha_n,0,0,\cdots)\in\mathbb{Z}_+^{(\infty)}\ (n\in\mathbb{N})$, $$D_{\Phi_\mu(\mathbf{T})^*}\Phi_\mu^\alpha(\mathbf{T})^{*}x=0,$$ so that $x=0$ by the assumption in (2).

Put $\widetilde{\mathbf{T}}=(T_{n+1},T_{n+2},\cdots)$, $\tilde{\mu}=(\mu_{n+1},\mu_{n+2},\cdots)$,
and  let $(\lambda,\tilde{\mu})$ denote the sequence $$(\lambda_1,\cdots,\lambda_n,\mu_{n+1},\mu_{n+2},\cdots)$$ for $\lambda=(\lambda_1,\cdots,\lambda_n)\in\mathbb{D}^n$. Since $x$ is orthogonal to $$\mathfrak{D}_{\Phi_{(\lambda,\tilde{\mu})}(\mathbf{T})^*}=\overline{\mathrm{Ran}\ D_{\Phi_{(\lambda,\tilde{\mu})}(\mathbf{T})^*}}$$
for each $\lambda\in\mathbb{D}^n$, we have
\begin{equation} \label{x in Ker}
  x\in\bigcap_{\lambda\in\mathbb{D}^n} \mathrm{Ker}\ D_{\Phi_{(\lambda,\tilde{\mu})}(\mathbf{T})^*}^*
=\bigcap_{\lambda\in\mathbb{D}^n} \mathrm{Ker}\ D_{\Phi_{(\lambda,\tilde{\mu})}(\mathbf{T})^*}.
\end{equation}
%It follows from Lemma \ref{dilation of defect op} that $$\|D_{\Phi_{(\lambda,\tilde{\mu})}(\mathbf{V})^*}x\|=\|D_{\Phi_{(\lambda,\tilde{\mu})}(\mathbf{T})^*}x\|=0,\quad\lambda\in\mathbb{D}^n.$$
Note that for each $\lambda=(\lambda_1,\cdots,\lambda_n)\in\mathbb{D}^n$, $$D_{\Phi_{(\lambda,\tilde{\mu})}(\mathbf{T})^*}=D_{\varphi_{\lambda_1}(T_1)^*}\cdots D_{\varphi_{\lambda_n}(T_n)^*}D_{\Phi_{\tilde{\mu}}(\widetilde{\mathbf{T}})^*},$$ and therefore (\ref{x in Ker}) gives $D_{\Phi_{\tilde{\mu}}(\widetilde{\mathbf{T}})^*}x
\in\mathrm{Ker}\ D_{\varphi_{\lambda_1}(T_1)^*}\cdots D_{\varphi_{\lambda_n}(T_n)^*}$.
This together with Lemma \ref{intersection of Ker} implies  $D_{\Phi_{\tilde{\mu}}(\widetilde{\mathbf{T}})^*}x=0$.
%Again by Lemma \ref{dilation of defect op}, we have
%$$\|D_{\Phi_{\tilde{\mu}}(\widetilde{\mathbf{T}})^*}x\|=\|D_{\Phi_{\tilde{\mu}}(\widetilde{\mathbf{V}})^*}x\|=0.$$
Since the sequence $\mathbf{T}$ is doubly commuting, $D_{\Phi_{\tilde{\mu}}(\widetilde{\mathbf{T}})^*}$ commutes with $\Phi^\alpha_\mu(\mathbf{T})^{*}$ on $\mathcal{H}$, which gives  \begin{equation*}
  \begin{split}
     D_{\Phi_\mu(\mathbf{T})^*}\Phi^\alpha_\mu(\mathbf{T})^{*}x
&=D_{\varphi_{\mu_1}(T_1)^*}\cdots D_{\varphi_{\mu_n}(T_n)^*}D_{\Phi_{\tilde{\mu}}(\widetilde{\mathbf{T}})^*}\Phi_\mu^\alpha(\mathbf{T})^{*}x\\
       &=D_{\varphi_{\mu_1}(T_1)^*}\cdots D_{\varphi_{\mu_n}(T_n)^*}\Phi_\mu^\alpha(\mathbf{T})^{*}D_{\Phi_{\tilde{\mu}}(\widetilde{\mathbf{T}})^*}x \\
       &=0.
  \end{split}
\end{equation*}

\vskip2mm

(3)$\Rightarrow$(1). In order to make this part of the proof more accessible, we divide it into several steps.

\vskip2mm

\textit{Step I. We give the construction of the subspaces $\mathcal{H}_\gamma$ of the Hilbert space $\mathcal{H}$, and prove that each  $\mathcal{H}_\gamma$ is $\mathbf{T}$-joint reducing.}

\vskip2mm

Define a binary relation $\sim$ on $\mathbb{D}^\infty$ as follows: for two points $\lambda=(\lambda_1,\lambda_2,\cdots)$ and $\mu=(\mu_1,\mu_2,\cdots)$ in $\mathbb{D}^\infty$, set $$\lambda\sim\mu\quad\Longleftrightarrow \quad \sum_{n=1}^{\infty}
\left|\frac{\lambda_n-\mu_n}{1-\overline{\lambda_n}\mu_n}\right|^2<\infty.$$
It is straightforward to see that the binary relation $\sim$ is an equivalence relation on $\mathbb{D}^\infty$. The set of $\sim$-equivalence classes is denoted by $\Delta$.
For $\gamma\in\Delta$, put $$\mathcal{H}_\gamma=\bigvee_{\lambda\in\gamma}\mathfrak{D}_{\Phi_\lambda(\mathbf{T})^*}.$$

Now we will show that $\mathcal{H}_\gamma$ is $\mathbf{T}$-joint reducing for any given $\gamma\in\Delta$.
Put $\mathbf{T}'=(T_2,T_3,\cdots)$ and $\gamma'=\{(\lambda_2,\lambda_3,\cdots):\lambda=(\lambda_1,\lambda_2,\cdots)\in\gamma\}$.
Since the sequence $\mathbf{T}$ is doubly commuting, we have that the defect operators $D_{\varphi_a(T_1)^*}$ and $D_{\Phi_\lambda(\mathbf{T}')^*}$ commute with each other for any $a\in\mathbb{D}$ and $\lambda\in\mathbb{D}^\infty$.
Also, it is easy to see  $\gamma=\mathbb{D}\times\gamma'$; that is,
$$\gamma=\{(a,\mu):a\in\mathbb{D}, \mu\in\gamma'\}.$$
It follows from Lemma \ref{intersection of Ker} that
$$\mathcal{H}=\bigvee_{a\in\mathbb{D}}
D_{\varphi_a(T_1)^*}\mathcal{H},$$ and therefore
\begin{equation*}
\begin{split}
   \mathcal{H}_\gamma & =\bigvee_{\lambda\in\gamma}\mathfrak{D}_{\Phi_\lambda(\mathbf{T})^*} \\
     & =\bigvee_{\mu\in\gamma'}\left(\bigvee_{a\in\mathbb{D}}
D_{\varphi_a(T_1)^*}D_{\Phi_\mu(\mathbf{T}')^*}\mathcal{H}\right) \\
     & =\bigvee_{\mu\in\gamma'}D_{\Phi_\mu(\mathbf{T}')^*}\left(\bigvee_{a\in\mathbb{D}}
D_{\varphi_a(T_1)^*}\mathcal{H}\right)
\\
     & =\bigvee_{\mu\in\gamma'}\mathfrak{D}_{\Phi_\mu(\mathbf{T}')^*}.
\end{split}
\end{equation*}
This implies that $\mathcal{H}_\gamma$ is reducing for $T_1$. Similarly, $\mathcal{H}_\gamma$ is $T_n$-reducing for each $n\geq2$.
\vskip2mm

\textit{Step II. We prove  $\mathcal{H}=\bigoplus_{\gamma\in\Delta} \mathcal{H}_\gamma$.}

\vskip2mm

By the assumption in (3), we have $\mathcal{H}=\bigvee_{\gamma\in\Delta} \mathcal{H}_\gamma$. It remains to show that the subspaces $\mathcal{H}_\gamma \ (\gamma\in\Delta)$ are pairwise orthogonal.

Let $\mathbf{V}=(V_1,V_2,\cdots)\in\mathcal{DP}(\mathcal{K})$ be
the minimal regular isometric dilation of $\mathbf{T}$. Put
$$\mathcal{K}_\gamma=\bigvee_{\lambda\in\gamma}\mathfrak{D}_{\Phi_\lambda(\mathbf{V})^*}$$
for $\gamma\in\Delta$.
Applying the argument in step I, we see that each  $\mathcal{K}_\gamma$ is $\mathbf{V}$-joint reducing.

For any closed subspace $M$ of the Hilbert space $\mathcal{K}$, let $P_M$ denote the orthogonal projection from $\mathcal{K}$ onto $M$. We first prove the following claims.
\vskip2mm
\textit{(a) For each $\gamma\in\Delta$, $\mathcal{H}_\gamma\subseteq \overline{P_{\mathcal{H}}\mathcal{K}_\gamma}$.}
\vskip2mm
Assume that $x\in\mathcal{H}$  is orthogonal to $P_{\mathcal{H}}\mathcal{K}_\gamma$. Then $x$ is orthogonal to $\mathcal{K}_\gamma$, which implies  $$x\in\mathrm{Ker}\ D_{\Phi_\lambda(\mathbf{V})^*}=\mathrm{Ker}\ D_{\Phi_\lambda(\mathbf{V})^*}^*,\quad \lambda\in\gamma.$$
It follows from  Lemma \ref{min iso dilation} that $\mathbf{V}$ is a coextension of $\mathbf{T}$. Then by Lemma \ref{dilation of defect op} we have  $$x\in\mathrm{Ker}\ D_{\Phi_\lambda(\mathbf{T})^*}=\mathrm{Ker}\ D_{\Phi_\lambda(\mathbf{T})^*}^*,\quad \lambda\in\gamma.$$
Therefore, $x$ is orthogonal to $\mathcal{H}_\gamma$. This gives $\mathcal{H}\ominus\overline{P_{\mathcal{H}}\mathcal{K}_\gamma}\subseteq\mathcal{H}\ominus\mathcal{H}_\gamma$, and then claim (a) is proved.

\vskip2mm
\textit{(b) The subspaces $\mathcal{K}_\gamma \ (\gamma\in\Delta)$ of $\mathcal{K}$ are pairwise orthogonal.}
\vskip2mm
Assume that $x\in\mathfrak{D}_{\Phi_\lambda(\mathbf{V})^*}\neq\{0\}$ for some $\lambda\in\mathbb{D}^\infty$. It suffices to show that
for any  $\mu\in\mathbb{D}^\infty$ not equivalent to $\lambda$, $$x\in\mathrm{Ker}\ D_{\Phi_\mu(\mathbf{V})^*}=\mathrm{Ker}\ D_{\Phi_\mu(\mathbf{V})^*}^*.$$
We first consider the case $\lambda=\mathbf{0}$. Then we have $\mu\notin\mathbb{D}_2^\infty$ in this case.
By Lemma \ref{local chara of K_gamma},
$[x]_{\mathbf{V}}$ is $\mathbf{V}$-joint reducing, and $\mathbf{V}|_{[x]_{\mathbf{V}}}$ is jointly unitarily equivalent to the  tuple $\mathbf{M_\zeta}$ of coordinate multiplication operators on the Hardy space $H^2(\mathbb{D}_2^\infty)$.
Then $[x]_{\mathbf{V}}$ is reducing for $D_{\Phi_\mu(\mathbf{V})^*}$, and $D_{\Phi_\mu(\mathbf{V})^*}|_{[x]_{\mathbf{V}}}$ is jointly unitarily equivalent to
$D_{\Phi_\mu(\mathbf{M_\zeta})^*}$.
It follows from Lemma \ref{wand subsp is 0} that
$D_{\Phi_\mu(\mathbf{M_\zeta})^*}=0$, and hence
$D_{\Phi_\mu(\mathbf{V})^*}$ vanishes on $[x]_{\mathbf{V}}$. The case for general $\lambda\in\mathbb{D}^\infty$ is proved by replacing the sequence $\mathbf{V}$ with $\Phi_\lambda(\mathbf{V})$ in the previous argument.
\vskip2mm
By claims (a) and (b),
it suffices to show that $P_{\mathcal{H}}P_{\mathcal{K}_\gamma}=P_{\mathcal{K}_\gamma}P_{\mathcal{H}}$
for any   $\gamma\in\Delta$.
Put $$\mathcal{H}_n=\bigvee_{\substack{\alpha\in\mathbb{Z}_+^{(\infty)} \\ \alpha_n=0 }}\mathbf{V}^\alpha\mathcal{H}$$
for $n\in\mathbb{N}$.
We also claim that

\vskip2mm
\textit{(c) for each $n\in\mathbb{N}$ and $\gamma\in\Delta$, $P_{\mathcal{H}_n}P_{\mathcal{K}_\gamma}=P_{\mathcal{K}_\gamma}P_{\mathcal{H}_n}$.}
\vskip2mm

%\begin{itemize}
% \item [(c)] $\mathcal{H}=\bigcap_{n=1}^\infty \mathcal{H}_n$;
%     \item [(d)] for each $n\in\mathbb{N}$ and $\gamma\in\Delta$, $P_{\mathcal{H}_n}P_{\mathcal{K}_\gamma}=P_{\mathcal{K}_\gamma}P_{\mathcal{H}_n}$.
%              \end{itemize}
Assuming this claim, the desired conclusion follows immediately from Lemma \ref{intersection of H_n} (3).
Below we will give its proof.
\vskip2mm

 Without loss of generality, we only prove for $n=1$.
Put $\mathbf{V}'=(V_2,V_3,\cdots)$.
By Lemma \ref{intersection of H_n} (1), $\mathcal{H}_1$ is joint invariant for $\mathbf{V}^*$. This gives that $\mathcal{H}_1$ is $\mathbf{V}'$-joint reducing, and then for each $\lambda\in\mathbb{D}^\infty$, $$P_{\mathcal{H}_1}D_{\Phi_\lambda(\mathbf{V}')^*}
=D_{\Phi_\lambda(\mathbf{V}')^*}P_{\mathcal{H}_1}.$$
Applying a similar argument in step I, one obtains
$$\mathcal{K}_\gamma =\bigvee_{\mu\in\gamma'}\mathfrak{D}_{\Phi_\mu(\mathbf{V}')^*},
$$ where $\gamma'=\{(\lambda_2,\lambda_3,\cdots):\lambda=(\lambda_1,\lambda_2,\cdots)\in\gamma\}$.
Thus, claim (c) follows.
\vskip2mm

%\textit{Step III. We prove that each  $\mathcal{H}_\gamma$ is $\mathbf{T}$-joint reducing.}
%
%\vskip2mm

\textit{Step III. We prove that the restriction $\mathbf{T}|_{\mathcal{H}_\gamma}$ on each nonzero $\mathcal{H}_\gamma\ (\gamma\in\Delta)$ is of quasi-Beurling type, and thus complete the proof.}

\vskip2mm

Suppose that $\mathcal{H}_\gamma$ is nonzero for some $\gamma\in\Delta$.
By the claims in step II of the proof, we see  $$\mathcal{H}_\gamma\subseteq \overline{P_{\mathcal{H}}\mathcal{K}_\gamma}=\mathcal{H}\cap\mathcal{K}_\gamma.$$
Then the minimal regular isometric dilation of the sequence $\mathbf{T}|_{\mathcal{H}_\gamma}$ is $\mathbf{V}|_{\mathcal{K}_\gamma}$.
It remains to prove that $\mathbf{V}|_{\mathcal{K}_\gamma}$ is of quasi-Beurling type; that is, $$[\mathfrak{D}_{\Phi_\lambda(\mathbf{V})^*}]_{\mathbf{V}}=[\mathfrak{D}_{\Phi_\lambda(\mathbf{V})^*}]_{\Phi_\lambda(\mathbf{V})}
=\mathcal{K}_\gamma$$ for some $\lambda\in\gamma$.

Fix $\lambda=(\lambda_1,\lambda_2,\cdots)\in\gamma$ satisfying $\mathfrak{D}_{\Phi_\lambda(\mathbf{V})^*}\neq\{0\}$.
Since $\mathcal{K}_\gamma$ is $\mathbf{V}$-joint reducing (see step II), one has $[\mathfrak{D}_{\Phi_\lambda(\mathbf{V})^*}]_{\mathbf{V}}\subseteq \mathcal{K}_\gamma$. For the converse inclusion, take an arbitrary nonzero element $x\in\mathcal{K}_\gamma$. We will show $x\in[\mathfrak{D}_{\Phi_\lambda(\mathbf{V})^*}]_{\mathbf{V}}$.
Since $x\in\mathfrak{D}_{\Phi_\mu(\mathbf{V})^*}$ for some $\mu=(\mu_1,\mu_2,\cdots)\in\gamma$,
 Lemma \ref{local chara of K_gamma} implies that
$[x]_{\mathbf{V}}$ is $\mathbf{V}$-joint reducing, and $\mathbf{V}|_{[x]_{\mathbf{V}}}$ is jointly unitarily equivalent to the  sequence $\Phi_\mu(\mathbf{M_\zeta})$, where $\mathbf{M_\zeta}$ is the tuple of coordinate multiplication operators on the Hardy space $H^2(\mathbb{D}_2^\infty)$.
Then $[x]_{\mathbf{V}}$ is reducing for $D_{\Phi_\lambda(\mathbf{V})^*}$, and $D_{\Phi_\lambda(\mathbf{V})^*}|_{[x]_{\mathbf{V}}}$ is jointly unitarily equivalent to
$$D_{(\Phi_\lambda\circ\Phi_\mu(\mathbf{M_\zeta}))^*}
=\prod_{n=1}^\infty(I-(\varphi_{\lambda_n}\circ\varphi_{\mu_n}(M_{\zeta_n}))
(\varphi_{\lambda_n}\circ\varphi_{\mu_n}(M_{\zeta_n}))^*).$$
Note that for each $n\in\mathbb{N}$, $\varphi_{\lambda_n}\circ\varphi_{\mu_n}=c_n\varphi_{\eta_n}$
for some unimodular constant $c_n$ and $\eta_n\in\mathbb{D}$, we have
\begin{equation*}\begin{split}
                    D_{(\Phi_\lambda\circ\Phi_\mu(\mathbf{M_\zeta}))^*} & =\prod_{n=1}^\infty(I-(c_n\varphi_{\eta_n}(M_{\zeta_n}))(c_n\varphi_{\eta_n}(M_{\zeta_n}))^*) \\
                      & =\prod_{n=1}^\infty(I-\varphi_{\eta_n}(M_{\zeta_n})\varphi_{\eta_n}(M_{\zeta_n})^*) \\
                      & = D_{\Phi_\eta(\mathbf{M_\zeta})^*}.
                 \end{split}
\end{equation*}
The fact  $\lambda\sim\mu$ yields $\eta=(\eta_1,\eta_2,\cdots)\in\mathbb{D}_2^\infty$,
and then Lemma \ref{wand subsp is 0} gives
$D_{\Phi_\eta(\mathbf{M_\zeta})^*}\neq0$. Thus,
there exists a nonzero element $y\in[x]_{\mathbf{V}}\bigcap\mathfrak{D}_{\Phi_\lambda(\mathbf{V})^*}$.
Again by Lemma \ref{local chara of K_gamma},
$[y]_{\mathbf{V}}$ is a joint reducing subspace of $[x]_{\mathbf{V}}$ for the sequence $\mathbf{V}|_{[x]_{\mathbf{V}}}$.
This together with Lemma \ref{M_zeta is irre} implies $$[x]_{\mathbf{V}}=[y]_{\mathbf{V}}\subseteq[\mathfrak{D}_{\Phi_\lambda(\mathbf{V})^*}]_{\mathbf{V}},$$ and then the proof is complete.
\qed
\vskip2mm

%rewrite $\mathcal{E}=\mathfrak{D}_{\Phi_\lambda(\mathbf{V})^*}$.
%If $x, y\in\mathcal{E}$ are mutually orthogonal, then
%for each $\alpha\in\mathbb{Z}_+^{(\infty)}$,
%$$\langle\mathbf{V}^\alpha x,  y\rangle=\langle x, \mathbf{V}^{\alpha*}y\rangle=\lambda^\alpha\langle x, y\rangle=0.$$
%This gives that $y$ is orthogonal to the $\mathbf{V}$-joint invariant subspace $[x]_{\mathbf{V}}$
From the proof of Theorem \ref{direct sums of quasi-BT}, we actually obtain the following collection of results, which will be used later.

\begin{cor} \label{collection of cor} Let $\mathbf{T}$ be a sequence in class $\mathcal{DC}(\mathcal{H})$, and $\mathbf{V}\in\mathcal{DP}(\mathcal{K})$  the minimal regular isometric dilation of $\mathbf{T}$.
We also let $\sim$ be the equivalence relation on $\mathbb{D}^\infty$ so that for  $\lambda,\mu\in\mathbb{D}^\infty$,  $\lambda\sim\mu$ means that $$\sum_{n=1}^{\infty}
\left|\frac{\lambda_n-\mu_n}{1-\overline{\lambda_n}\mu_n}\right|^2<\infty.$$
\begin{itemize}
  \item [(1)] If $\mathbf{T}$ is of Beurling type, then $\mathcal{K}$ has an orthogonal decomposition $$\mathcal{K}
      =\bigoplus_{\alpha\in\mathbb{Z}_+^{(\infty)}}
\mathbf{V}^\alpha\mathfrak{D}_{\mathbf{V}^*},$$
      %$\mathbf{V}$ is jointly unitarily equivalent to the tuple $\mathbf{M}_\zeta$ of coordinate multiplication operators on the vector-valued Hardy space $H_{\mathfrak{D}_{\mathbf{V}^*}}^2(\mathbb{D}_2^\infty)$,
      and for each $x\in \mathcal{H}$, $$\|x\|^2=\sum_{\alpha\in\mathbb{Z}_+^{(\infty)}}
\|D_{\mathbf{T}^*}\mathbf{T}^{*\alpha}x\|^2.$$
  \item [(2)]
  If $\lambda, \mu\in\mathbb{D}^\infty$ are not $\sim$-equivalent,
then the defect spaces $\mathfrak{D}_{\Phi_\lambda(\mathbf{T})^*}$ and $\mathfrak{D}_{\Phi_\mu(\mathbf{T})^*}$ are mutually orthogonal.
\item [(3)] For any $\sim$-equivalence class $\gamma$,
put $\mathcal{H}_\gamma=\bigvee_{\lambda\in\gamma}\mathfrak{D}_{\Phi_\lambda(\mathbf{T})^*}$,
    $\mathcal{K}_\gamma=\bigvee_{\lambda\in\gamma}\mathfrak{D}_{\Phi_\lambda(\mathbf{V})^*}$.
  Then  $\mathcal{K}_\gamma=[\mathfrak{D}_{\Phi_\mu(\mathbf{V})^*}]_{\mathbf{V}}$ for each $\mu\in\gamma$,
    $\mathbf{V}|_{\mathcal{K}_\gamma}$ is the minimal regular isometric dilation of $\mathbf{T}|_{\mathcal{H}_\gamma}$,
    and  $\mathcal{H}_\gamma,\mathcal{K}_\gamma$
    are joint reducing for $\mathbf{T}$ and $\mathbf{V}$, respectively.
\end{itemize}
\end{cor}

We are ready to prove Corollary \ref{BT}.
\vskip2mm
\noindent\textbf{Proof of Corollary \ref{BT}.} (1)$\Rightarrow$(2) and (1)$\Rightarrow$(3). See Corollary \ref{collection of cor} (1).

\vskip2mm

(2)$\Rightarrow$(1). Obvious.

\vskip2mm

(3)$\Rightarrow$(4). This implication follows from the proof of Theorem \ref{direct sums of quasi-BT} and the fact that $\mathbb{D}_2^\infty$ coincides with the  Cartesian product
$$\{(\lambda,\mu):\lambda\in \mathbb{D}^n,\mu\in\mathbb{D}_2^\infty\}.$$

\vskip2mm

(4)$\Rightarrow$(1). Let $\mathbf{V}\in\mathcal{DP}(\mathcal{K})$ be the minimal regular isometric dilation of $\mathbf{T}$. Note that $\mathbb{D}_2^\infty$ is a $\sim$-equivalence class that contains $\mathbf{0}$ and
$\mathcal{H}=\mathcal{H}_{\mathbb{D}_2^\infty}$.
It follows from Corollary \ref{collection of cor} (3) that $\mathcal{K}=\mathcal{K}_{\mathbb{D}_2^\infty}=[\mathfrak{D}_{\mathbf{V}^*}]_{\mathbf{V}}$.
The proof is complete.
\qed
\vskip2mm

%Corollary \ref{BT'} immediately  implies the following.
%
%\begin{cor} \label{BT'} Let  $\mathbf{T}$ be a sequence in class $\mathcal{DC}(\mathcal{H})$. The followings are equivalent:
%\begin{itemize}
%  \item [(1)] $\mathbf{T}$ is of quasi-Beurling type;
%  \item [(2)] the minimal regular isometric dilation of $\mathbf{T}$ is jointly unitarily equivalent to the the sequence $\Phi_\lambda(\mathbf{M}_\zeta)$ on a vector-valued Hardy space $H^2_{\mathcal{E}}(\mathbb{D}_2^\infty)$
%for some $\lambda\in\mathbb{D}^\infty$;
%  \item [(3)] for each $x\in \mathcal{H}$, $\|x\|^2=\sum_{\alpha\in\mathbb{Z}_+^{(\infty)}}
%\|D_{\Phi_{\lambda}(\mathbf{T})^*}
%\Phi_{\lambda}^\alpha(\mathbf{T})^{*}x\|^2$;
%  \item [(4)] $[\mathfrak{D}_{\Phi_{\lambda}(\mathbf{T})^*}]_\mathbf{T}=\mathcal{H}$.
%\end{itemize}
%\end{cor}

It follows immediately  from Corollary \ref{BT} that
a sequence $\mathbf{T}\in\mathcal{DC}$ is of quasi-Beurling type if and only if
  the minimal regular isometric dilation of $\mathbf{T}$ is jointly unitarily equivalent to the sequence $\Phi_\lambda(\mathbf{M}_\zeta)$ on a vector-valued Hardy space $H^2_{\mathcal{E}}(\mathbb{D}_2^\infty)$
for some $\lambda\in\mathbb{D}^\infty$.

Below we give an example to illustrate that  condition (3) in Theorem \ref{direct sums of quasi-BT} is nontrivial for sequences in class $\mathcal{DC}$ (Theorem \ref{example}). By comparing this with Lemma \ref{intersection of Ker}, we see that the infinite-tuple case diverges considerably from the finite-tuple case.

\vskip2mm

\begin{thm}
 There exists a sequence $\mathbf{V}\in\mathcal{DP}$ such that $\mathfrak{D}_{\Phi_\lambda(\mathbf{V})^*}=\{0\}$ for each
$\lambda\in\mathbb{D}^\infty$.
\end{thm}
\begin{proof}
Let $\mathbf{T}^2$ denote  $(T_1^2,T_2^2,\cdots)$ for
a sequence $\mathbf{T}=(T_1,T_2,\cdots)$ of operators, and $\mathbf{M}_\zeta=(M_{\zeta_1},M_{\zeta_2},\cdots)$ be
the tuple of coordinate multiplication operators on $H^2(\mathbb{D}^\infty_2)$. Put $$\mathcal{E}_n=\mathrm{Ker}\ M_{\zeta_n}^*
=H^2(\mathbb{D}^\infty_2)\ominus\zeta_n H^2(\mathbb{D}^\infty_2),\quad n\in\mathbb{N}.$$
It is clear that $H^2(\mathbb{D}^\infty_2)=\bigoplus_{k=0}^\infty \zeta_n^k\mathcal{E}_n$ for each $n\in\mathbb{N}$.
Define a sequence $\mathbf{V}$ of isometries on $H^2(\mathbb{D}^\infty_2)$ by setting
\begin{equation*}
  V_n(\zeta_n^k F)= \begin{cases}
                      \ \zeta_n^{k+3}F, & \mbox{if } k\ \mathrm{is}\ \mathrm{even}; \\
                      \ \zeta_n^{k-1}F, & \mbox{if } k\ \mathrm{is}\ \mathrm{odd}.
                    \end{cases}
\end{equation*}
 for $n\in\mathbb{N}$ and $F\in\mathcal{E}_n$.
It is routine to check that $\mathbf{V}\in\mathcal{DP}$ and $\mathbf{V}^2=\mathbf{M}_\zeta^2$.

Now we will show $\mathfrak{D}_{\Phi_\lambda(\mathbf{V})^*}=\{0\}$ for each
$\lambda\in\mathbb{D}^\infty$. Assume conversely that there exists a point $\lambda\in\mathbb{D}^\infty$ such that $\mathfrak{D}_{\Phi_\lambda(\mathbf{V})^*}$ contains a  function $F\neq0$. By the comments above Lemma \ref{local chara of K_gamma}, $F$ is exactly an eigenvector of the sequence $\mathbf{V}^*$ corresponding to the joint eigenvalue $\overline{\lambda}=(\overline{\lambda_1},\overline{\lambda_2},\cdots)$, and therefore $$F\in\mathfrak{D}_{\Phi_{\lambda^2}(\mathbf{V}^2)^*}=\mathfrak{D}_{\Phi_{\lambda^2}(\mathbf{M}_\zeta^2)^*},$$
where $\lambda^2=(\lambda_1^2,\lambda_2^2,\cdots)$. Let $\mathbf{K}_{\mu}$ denote the reproducing kernel of $H^2(\mathbb{D}^\infty_2)$ at the point $\mu\in\mathbb{D}^\infty_2$. Then
$$\mathbf{K}_{\mu}\in\mathfrak{D}_{\Phi_{\mu^2}(\mathbf{M}_\zeta^2)^*}
\subseteq\bigvee_{\xi\in\mathbb{D}_2^\infty}
\mathfrak{D}_{\Phi_\xi(\mathbf{M}_\zeta^2)^*},$$
which gives $$\bigvee_{\xi\in\mathbb{D}_2^\infty}
\mathfrak{D}_{\Phi_\xi(\mathbf{M}_\zeta^2)^*}
=H^2(\mathbb{D}^\infty_2)$$
since the set $\{\mathbf{K}_{\mu}:\mu\in\mathbb{D}^\infty_2\}$ is complete in $H^2(\mathbb{D}^\infty_2)$.
Then by Corollary \ref{collection of cor} (2), $\mathfrak{D}_{\Phi_\xi(\mathbf{V}^2)^*}=\{0\}$ for $\xi\notin\mathbb{D}_2^\infty$, forcing $\lambda^2\in\mathbb{D}_2^\infty$. In particular, $\lambda_n\rightarrow0\ (n\rightarrow\infty)$.

Write $\mathcal{F}=\mathfrak{D}_{\mathbf{V}^{*2}}
=\mathfrak{D}_{\mathbf{M}_\zeta^{*2}}$.
Then $$\mathcal{F}=\bigcap_{n=1}^\infty\mathrm{Ker}\ M_{\zeta_n}^{*2}
=\overline{\mathrm{span}\{\zeta^\alpha:\alpha=(\alpha_1,\alpha_2,\cdots)\in\mathbb{Z}_+^{(\infty)}\ \mathrm{with}\ \mathrm{each} \ \alpha_n\leq1\}},$$
and $$P_{\mathcal{F}}=\prod_{n=1}^{\infty}(I-V_n^2V_n^{*2})=\prod_{n=1}^{\infty}(I-M_{\zeta_n}^2M_{\zeta_n}^{*2}),$$
Where $P_{\mathcal{F}}$ is the orthogonal projection from $H^2(\mathbb{D}^\infty_2)$ onto $\mathcal{F}$.
This implies that
$$\prod_{i=1}^{n}(I-M_{\zeta_i}^2M_{\zeta_i}^{*2})F
=\prod_{i=1}^{n}(I-M_{\zeta_i}^2V_i^{*2})F
=F\cdot\prod_{i=1}^{n}(1-\overline{\lambda_n^2}\zeta_n^2)$$
converges to $G=P_{\mathcal{F}}F$ ($n\rightarrow\infty$) in  $H^2(\mathbb{D}^\infty_2)$-norm. Note that the reproducing kernel $\mathbf{K}_{\lambda^2}$ vanishes nowhere on $\mathbb{D}^\infty_2$, and $\prod_{i=1}^{n}(1-\overline{\lambda_n^2}\zeta_n^2)$ converges pointwisely to the function $\frac{1}{\mathbf{K}_{\lambda^2}}$ on $\mathbb{D}^\infty_2$ as $n\rightarrow\infty$, $G$ must coincide with the function $\frac{F}{\mathbf{K}_{\lambda^2}}$ on $\mathbb{D}^\infty_2$, forcing $G\neq0$.
Since $G\in\mathcal{F}$, there exists some $\alpha=(\alpha_1,\alpha_2,\cdots)\in\mathbb{Z}_+^{(\infty)}$ with each $\alpha_n\leq1$
such that $\langle G,\zeta^\alpha\rangle\neq0$.

Below we prove $\langle G,\zeta^\alpha\rangle=0$ to reach a contradiction. Since for each $n\in\mathbb{N}$, $$G\in\mathrm{Ker}\ M_{\zeta_n}^{*2}=\mathrm{Ker}\ V_n^{*2},$$ there corresponds a decomposition
$G=G_n+H_n$ of $G$ such that $G_n\in\mathrm{Ker}\ V_n^*$
and $H_n\in\mathrm{Ker}\ V_n^{*2}\ominus\mathrm{Ker}\ V_n^*$.
Since $\mathbf{V}$ is doubly commuting, for any $n\in\mathbb{N}$, $$G_n=(I-V_nV_n^*)G=(I-V_nV_n^*)\cdot
\prod_{m\neq n}^{\infty}(I-V_m^2V_m^{*2})F,$$
and \begin{equation*}
     \begin{split}
       H_n & = (V_nV_n^*-V_n^2V_n^{*2})G \\
       & =(V_nV_n^*-V_n^2V_n^{*2})\cdot
\prod_{m\neq n}^{\infty}(I-V_m^2V_m^{*2})F\\
          & =V_n(I-V_nV_n^*)V_n^*\cdot
\prod_{m\neq n}^{\infty}(I-V_m^2V_m^{*2})F \\
          & =V_n(I-V_nV_n^*)\cdot
\prod_{m\neq n}^{\infty}(I-V_m^2V_m^{*2})V_n^*F \\
& =\overline{\lambda_n}V_n(I-V_nV_n^*)\cdot
\prod_{m\neq n}^{\infty}(I-V_m^2V_m^{*2})F,
     \end{split}
    \end{equation*}
which gives $$\|H_n\|=|\lambda_n|\|V_nG_n\|=|\lambda_n|\|G_n\|.$$
By  the fact that $\lambda_n\rightarrow0\ (n\rightarrow\infty)$ and $\|G\|^2=\|G_n\|^2+\|H_n\|^2\ (n\in\mathbb{N})$, we see $\|H_n\|\rightarrow0\ (n\rightarrow\infty)$.
This gives that $G_n\rightarrow G\ (n\rightarrow\infty)$ in  norm, and then  $\langle G_n,\zeta^\alpha\rangle\rightarrow\langle G,\zeta^\alpha\rangle\ (n\rightarrow\infty)$.
Since $G_n\in\zeta_nH^2(\mathbb{D}^\infty_2)\ (n\in\mathbb{N})$, $\langle G_n,\zeta^\alpha\rangle=0$ for $n$ large sufficiently, forcing $\langle G,\zeta^\alpha\rangle=0$. This completes the proof.
\end{proof}

We also have the following application of Theorem \ref{direct sums of quasi-BT} and  Corollary \ref{BT}.
Set $\Gamma$ to be the set of points $\lambda=(\lambda_1,\lambda_2,\cdots)\in\mathbb{D}^\infty$ satisfying  $$\lim_{m\rightarrow\infty}\prod_{n=m}^{\infty}\lambda_m=1.$$
One can check that for the equivalence relation $\sim$ on $\mathbb{D}^\infty$ given in Corollary \ref{collection of cor}, $\Gamma$ is the union of some $\sim$-equivalence classes, namely
for a pair $\lambda,\mu$ of $\sim$-equivalent points in $\mathbb{D}^\infty$, $\lambda\in\Gamma$ if and only if $\mu\in\Gamma$.

\begin{cor} \label{application} Suppose $\mathbf{T}\in\mathcal{DC}(\mathcal{H})$ and put $A_n=T_1\cdots T_n\ (n\in\mathbb{N})$.
If
$\bigvee_{\lambda\in\Gamma}\mathfrak{D}_{\Phi_\lambda(\mathbf{T})^*}=\mathcal{H}$, then $\{A_n\}_{n\in\mathbb{N}}$ converges in the strong operator topology.
\end{cor}
\begin{proof}
Assume $\bigvee_{\lambda\in\Gamma}\mathfrak{D}_{\Phi_\lambda(\mathbf{T})^*}=\mathcal{H}$.
By Theorem \ref{direct sums of quasi-BT}, $\mathbf{T}$ can be decomposed into the direct sum of $\mathcal{DC}$-sequences of quasi-Beurling type. Without loss of generality, we may further assume that $\mathbf{T}$  itself is of quasi-Beurling type. Take $\lambda\in\Gamma$ so that $\Phi_\lambda(\mathbf{T})$ is of Beurling type. Then Corollary \ref{BT} implies that the minimal regular isometric dilation $\mathbf{V}$ of
$\mathbf{T}$ is jointly unitarily equivalent to  the sequence $(M_{\widetilde{\varphi_{\lambda_1}}},M_{\widetilde{\varphi_{\lambda_2}}},\cdots)$ of multiplication operators on  a vector-valued Hardy space $H_{\mathcal{E}}^2(\mathbb{D}_2^\infty)$, where $\widetilde{\varphi_{\lambda_n}}(\zeta)=\varphi_{\lambda_n}(\zeta_n)\ (\zeta\in\mathbb{D}^\infty_2)$.

Put $F_n=\prod_{i=1}^{n}\widetilde{\varphi_{\lambda_i}}\ (n\in\mathbb{N})$.
It suffices to show that
$\{M_{F_n}\}_{n\in\mathbb{N}}$ converges in the strong operator topology.
Since this sequence has uniformly bounded operator norms and $$\{p\cdot x:p\in\mathcal{P}_\infty,x\in\mathcal{E}\}$$ is complete in $H_{\mathcal{E}}^2(\mathbb{D}_2^\infty)$, we only need to prove for any $p\in\mathcal{P}_\infty$ and any $x\in\mathcal{E}$,
$\{F_np\cdot x\}_{n\in\mathbb{N}}$ converges in $H_{\mathcal{E}}^2(\mathbb{D}_2^\infty)$,
where $\mathcal{P}_\infty$ is the polynomial ring in countably infinitely many variables. This is clearly equivalent to that
$\{F_n\}_{n\in\mathbb{N}}$ converges in $H^2(\mathbb{D}_2^\infty)$.

For $n>m$ a simple calculation gives
$$\|F_n-F_m\|^2=\|\prod_{i=m+1}^{n}\widetilde{\varphi_{\lambda_i}}-1\|^2
=|\prod_{i=m+1}^{n}\lambda_i-1|^2+(1-\prod_{i=m+1}^{n}|\lambda_i|^2).$$
Since $\lim_{m\rightarrow\infty}\prod_{n=m}^{\infty}\lambda_m=1$,
$\{F_n\}_{n\in\mathbb{N}}$ is
a Cauchy sequence, which completes the proof.
\end{proof}

To end this section, we record an example of $\mathcal{DC}$-sequence of quasi-Beurling type.

\begin{prop} \label{sequence of multiplication operators} Let $\{f_n\}_{n\in\mathbb{N}}$ be a sequence of bounded analytic functions on the unit disk $\mathbb{D}$. If for each $n\in\mathbb{N}$, $\|f_n\|_\infty\leq1$ and $f_n$ is not a unimodular constant, then the sequence
$(M_{\widetilde{f_1}},M_{\widetilde{f_2}},\cdots)$ of multiplication operators on
the Hardy space $H^2(\mathbb{D}_2^\infty)$ is a $\mathcal{DC}$-sequence of quasi-Beurling type,
where $\widetilde{f_n}(\zeta)=f_n(\zeta_n)\ (n\in\mathbb{N},\zeta\in\mathbb{D}^\infty_2)$.
\end{prop}
\begin{proof} It is routine to check that
$(M_{\widetilde{f_1}},M_{\widetilde{f_2}},\cdots)$ is a sequence in class $\mathcal{DC}$. Put $g_n=\varphi_{f_n(0)}\circ f_n\ (n\in\mathbb{N})$. Then for each $n\in\mathbb{N}$,
$g_n(0)=0$ and
$M_{\widetilde{g_n}}=\varphi_{f_n(0)}(M_{\widetilde{f_n}})$, where $\widetilde{g_n}(\zeta)=g_n(\zeta_n)$.
We will show that the sequence $\mathbf{M}=(M_{\widetilde{g_1}},M_{\widetilde{g_2}},\cdots)$  is of Beurling type, which implies the desired conclusion.

Note that each $n\in\mathbb{N}$ and $\lambda=(\lambda_1,\lambda_2,\cdots)\in\mathbb{D}^\infty_2$, $$\varphi_{g_n(\lambda_n)}(M_{\widetilde{g_n}})^*\mathbf{K}_\lambda
=M_{\varphi_{g_n(\lambda_n)}\circ\widetilde{g_n}}^*\mathbf{K}_\lambda
=\varphi_{g_n(\lambda_n)}(\widetilde{g_n}(\lambda))\mathbf{K}_\lambda=0,$$
which gives $$(I-\varphi_{g_n(\lambda_n)}(M_{\widetilde{g_n}})
\varphi_{g_n(\lambda_n)}(M_{\widetilde{g_n}})^*))\mathbf{K}_\lambda=\mathbf{K}_\lambda.$$
It follows that $D_{\Phi_{\mu(\lambda)}(\mathbf{M})^*}\mathbf{K}_\lambda=\mathbf{K}_\lambda$ for
any $\lambda\in\mathbb{D}^\infty_2$, where $$\mu(\lambda)=(g_1(\lambda_1),g_2(\lambda_2),\cdots).$$
Also by Schwarz' Lemma, we have  that for
each $\lambda\in\mathbb{D}^\infty_2$, $$\sum_{n=1}^\infty|g_n(\lambda_n)|^2\leq\sum_{n=1}^\infty|\lambda_n|^2<\infty,$$
and then
$$\mathbf{K}_{\lambda}\in\mathfrak{D}_{\Phi_{\mu(\lambda)}(\mathbf{M})^*}
\subseteq\bigvee_{\mu\in\mathbb{D}_2^\infty}
\mathfrak{D}_{\Phi_\mu(\mathbf{M})^*}.$$
This gives  $\bigvee_{\mu\in\mathbb{D}_2^\infty}
\mathfrak{D}_{\Phi_\mu(\mathbf{M})^*}=H^2(\mathbb{D}_2^\infty)$. It follows from
Corollary \ref{BT} that the sequence $\mathbf{M}=(M_{\widetilde{g_1}},M_{\widetilde{g_2}},\cdots)$  is of Beurling type.
\end{proof}
%\begin{prop} Let $\mathcal{E}$ be a Hilbert space and $\Theta$ be a sequence in $H^\infty_{\mathcal{B}(\mathcal{E})}(\mathbb{D})$.
%If
%$\Theta$ is diagonalizable
%and purely contractive, then $\mathbf{M}_{\Theta}$ is a sequence in class $\mathcal{DC}$ and has a decomposition
%of quasi-Beurling type.
%\end{prop}
\begin{rem} \label{remark for Section 4}
 From the proof of Proposition \ref{sequence of multiplication operators}, we see that the $\sim$-equivalence class of the point $(f_1(0),f_2(0),\cdots)\in\mathbb{D}^\infty$ (see Corollary \ref{collection of cor}) is an invariant for the sequence $(M_{\widetilde{f_1}},M_{\widetilde{f_2}},\cdots)$.
 More precisely, let $\{f_n\}_{n\in\mathbb{N}}$ and $\{g_n\}_{n\in\mathbb{N}}$ be two sequence of functions that satisfy the conditions given in the proposition,
 and put $\widetilde{f_n}(\zeta)=f_n(\zeta_n)$, $\widetilde{g_n}(\zeta)=g_n(\zeta_n)\ (n\in\mathbb{N},\zeta\in\mathbb{D}^\infty_2)$.
 If the multiplication operators $(M_{\widetilde{f_1}},M_{\widetilde{f_2}},\cdots)$, $(M_{\widetilde{g_1}},M_{\widetilde{g_2}},\cdots)$ are jointly unitarily equivalent, then
 $(f_1(0),f_2(0),\cdots)$ and $(g_1(0),g_2(0),\cdots)$
 belong to the same $\sim$-equivalence class; that is,
 $$\sum_{n=1}^{\infty}
\left|\frac{f_n(0)-g_n(0)}{1-\overline{f_n(0)}g_n(0)}\right|^2<\infty.$$

Since the number of $\sim$-equivalence classes is countless,
%Since there are numerous $\sim$-equivalence classes,
%the  equivalence class
$\mathbb{D}_2^\infty$
 %containing $\mathbf{0}$
 is a very ``small" part in $\mathbb{D}^\infty$. Hence,
for ``almost all" choices of the sequence $\{f_n\}_{n\in\mathbb{N}}$ of  functions, the defect space of $(M_{\widetilde{f_1}}^*,M_{\widetilde{f_2}}^*,\cdots)$
is $\{0\}$ by Proposition \ref{sequence of multiplication operators} and Corollary \ref{collection of cor}, and then $(M_{\widetilde{f_1}},M_{\widetilde{f_2}},\cdots)$ is not of Beurling type.
\end{rem}

\section{Analytic Model}
In this section, we will prove that every sequence in class $\mathcal{DP}$ is jointly unitarily equivalent to
a sequence of multiplication operators induced by
operator-valued inner functions each of which involves one different variables. We thus establish an operator-valued analytic functional model for general $\mathcal{DC}$-sequences.

\begin{thm} \label{analytic model} Let $\mathbf{T}$ be a sequence in class $\mathcal{DC}(\mathcal{H})$, and $\mathbf{V}\in\mathcal{DP}(\mathcal{K})$  the minimal regular isometric dilation of $\mathbf{T}$.
Then there exist a Hilbert space $\mathcal{E}$, a unitary operator $U:\mathcal{K}\rightarrow H_{\mathcal{E}}^2(\mathbb{D}_2^\infty)$ and a sequence $\Theta=(\theta_1,
\theta_2,\cdots)$ of inner functions in $H^\infty_{\mathcal{B}(\mathcal{E})}(\mathbb{D})$, such that \begin{itemize}
  \item [(1)] for each $n\in\mathbb{N}$,
$UV_nU^*=M_{\widetilde{\theta_n}}$, where
$\widetilde{\theta_n}(\zeta)=\theta_n(\zeta_n)\ (\zeta\in\mathbb{D}^\infty_2)$;
  \item [(2)] $\mathcal{Q}=U\mathcal{H}$ is a quotient module of $H_{\mathcal{E}}^2(\mathbb{D}_2^\infty)$.
\end{itemize}
\end{thm}

The tuple  $(\mathcal{Q},\Theta)$ in Theorem \ref{analytic model} is said to be an \emph{analytic model} for the  $\mathcal{DC}$-sequence $\mathbf{T}$, and the Hilbert space $\mathcal{E}$ is called the \emph{underlying space}
of the analytic model $(\mathcal{Q},\Theta)$.
The sequence $(M_{\widetilde{\theta_1}},M_{\widetilde{\theta_2}},\cdots)$  is denoted by $\mathbf{M}_{\Theta}$ for simplicity. Also, for the trivial case $$\Theta=(z\cdot I_{\mathcal{E}},z\cdot I_{\mathcal{E}},\cdots),$$ we simply write
$\mathcal{Q}$ for  $(\mathcal{Q},\Theta)$.
It is clear that $\mathbf{T}$ is jointly unitarily equivalent to
the sequence $P_{\mathcal{Q}}\mathbf{M}_{\Theta}|_{\mathcal{Q}}$, the compression of the sequence $\mathbf{M}_{\Theta}$ on $\mathcal{Q}$.

To prove Theorem \ref{analytic model}, we need the following.

\begin{prop} \label{power of DC}
 Suppose $\mathbf{T}\in\mathcal{DC}$. Then there exists a sequence $(k_1,k_2,\cdots)$ of positive integers, such that
$(T_1^{k_1},T_2^{k_2},\cdots)$ is of Beurling type.
\end{prop}
We  thank Dr. Y. Wang for discussion on the proof of Proposition
\ref{power of DC}.
\begin{proof}
Assume that $\mathbf{T}\in\mathcal{DC}(\mathcal{H})$ and $\mathbf{V}\in\mathcal{DP}(\mathcal{K})$ is the minimal regular isometric dilation of $\mathbf{T}$.
For a sequence $\mathbf{k}=(k_1,k_2,\cdots)$ of positive integers, put
$$\mathbf{V}_{\mathbf{k}}=(V_1^{k_1},V_2^{k_2},\cdots)$$
and $M_{\mathbf{k}}=[\mathfrak{D}_{\mathbf{V}_{\mathbf{k}}^*}]_{\mathbf{V}_{\mathbf{k}}}$.
It suffices to prove that there exists
a sequence $\mathbf{k}$ of positive integers such that $\mathcal{K}=M_{\mathbf{k}}$.

Note that we have made the
convention that $\mathcal{H}$ is a separable Hilbert space in Subsection 1.2. Then $\mathcal{K}=[\mathcal{H}]_{\mathbf{V}}$ is also separable.
Take a sequence $\{x_n\}_{n\in\mathbb{N}}$ in $\mathcal{K}$ such that all its elements  constitute
a dense subset of $\mathcal{K}$ and each element appears in infinitely many times.
It follows from Lemma \ref{estimate lemma} that
for each $n\in\mathbb{N}$, there exists a sequence
$\mathbf{k}^{(n)}=(k_{1}^{(n)},k_{2}^{(n)},\cdots)$ of positive integers, such that
\begin{equation}\label{estimate for x_n}
 \|D_{\mathbf{V}_{\mathbf{k}^{(n)}}^*}x_n\|\geq(1-\frac{1}{2^n})\|x_n\|.
\end{equation}
Set $k_n=\max\{k_{n}^{(1)},\cdots,k_{n}^{(n)}\}\ (n\in\mathbb{N})$ and put $\mathbf{k}=(k_1,k_2,\cdots)$.
 %Then for each $n\in\mathbb{N}$,
%$$\mathfrak{D}_{\mathbf{V}_{\mathbf{k}^{(n)}}^*}\subseteq
%\mathfrak{D}_{\mathbf{V}_{(0,\cdots,0,k_{n}^{(n)},k_{n+1}^{(n)},\cdots)}^*}\subseteq
%\mathfrak{D}_{\mathbf{V}_{(0,\cdots,0,k_n,k_{n+1},\cdots)}^*}.$$
Note that $$[\mathfrak{D}_{\mathbf{V}_{\mathbf{k}}^*}]_{V_1^{k_1}}
=[(I-V_1^{k_1}V_1^{*k_1})\mathfrak{D}_{\mathbf{V}_{(0,\mathbf{k}')}^*}]_{V_1^{k_1}}
=\mathfrak{D}_{\mathbf{V}_{(0,\mathbf{k}')}^*},$$
where $\mathbf{k}'=(k_2,k_3,\cdots)$. It follows that
$$M_{\mathbf{k}}
=[\mathfrak{D}_{\mathbf{V}_{\mathbf{k}}^*}]_{\mathbf{V}_{\mathbf{k}}}
=[\mathfrak{D}_{\mathbf{V}_{(0,\mathbf{k}')}^*}]_{\mathbf{V}_{(0,\mathbf{k}')}}
=M_{(0,\mathbf{k}')},$$
and thus by induction,
\begin{align*}
    & M_{\mathbf{k}}=M_{(0,\mathbf{k}')}=\cdots=M_{(0,\cdots,0,k_n,k_{n+1},\cdots)}
 \\ \supseteq{} & \mathfrak{D}_{\mathbf{V}_{(0,\cdots,0,k_n,k_{n+1},\cdots)}^*}\supseteq
\mathfrak{D}_{\mathbf{V}_{(0,\cdots,0,k_{n}^{(n)},k_{n+1}^{(n)},\cdots)}^*}
\supseteq\mathfrak{D}_{\mathbf{V}_{\mathbf{k}^{(n)}}^*}.
\end{align*}
for each $n\in\mathbb{N}$.
By (\ref{estimate for x_n}), one obtains
$$\|P_{M_{\mathbf{k}}}x_n\|\geq\|D_{\mathbf{V}_{\mathbf{k}^{(n)}}^*}x_n\|\geq(1-\frac{1}{2^n})\|x_n\|,\quad n\in\mathbb{N},$$
forcing $\|P_{M_{\mathbf{k}}}x\|=\|x\|$ for any $x\in\mathcal{K}$, where $P_{M_{\mathbf{k}}}$ is the orthogonal projection from $\mathcal{K}$ onto $M_{\mathbf{k}}$. This completes the proof.
\end{proof}

\noindent\textbf{Proof of Theorem \ref{analytic model}.}
Assume that $\mathbf{T}\in\mathcal{DC}(\mathcal{H})$ and $\mathbf{V}\in\mathcal{DP}(\mathcal{K})$ is the minimal regular isometric dilation of $\mathbf{T}$. It follows from  Proposition \ref{power of DC} that there exists a sequence $(k_1,k_2,\cdots)$ of positive integers, such that
$(T_1^{k_1},T_2^{k_2},\cdots)$ is of Beurling type. Then  by Corollary \ref{BT}, the sequence $(V_1^{k_1},V_2^{k_2},\cdots)$ is
jointly unitarily equivalent to the tuple $\mathbf{M}_\zeta=(M_{\zeta_1},M_{\zeta_2},\cdots)$ of coordinate multiplication operators on a vector-valued Hardy space $H_{\mathcal{E}}^2(\mathbb{D}_2^\infty)$
 via a unitary operator $U:\mathcal{K}\rightarrow H_{\mathcal{E}}^2(\mathbb{D}_2^\infty)$.
 This implies that for each $n\in\mathbb{N}$, $$M_{\zeta_n}=UV_n^{k_n}U^*=(UV_nU^*)^{k_n},$$
and hence $\widetilde{V_n}=UV_nU^*$ commutes with
$M_{\zeta_n}$ and doubly commutes with $M_{\zeta_m}$ for any $m\neq n$.

It remains to show that $\widetilde{V_n}\ (n\in\mathbb{N})$ is a multiplication operator induced by an operator-valued inner function $\widetilde{\theta_n}\in H^\infty_{\mathcal{B}(\mathcal{E})}(\mathbb{D}_2^\infty)$, which depends only on the $n$-th variable $\zeta_n$.
We first prove for $n=1$.
Put $\mathbf{M}_\zeta'=(M_{\zeta_2},M_{\zeta_3},\cdots)$ and set
$$\mathcal{L}=\overline{\mathrm{span}\{\zeta^\alpha:\alpha=(\alpha_1,\alpha_2,\cdots)\in\mathbb{Z}_+^{(\infty)}\ \mathrm{with} \ \alpha_1=0\}}.$$
Then we have $H_{\mathcal{E}}^2(\mathbb{D}_2^\infty)=H_{\mathcal{E}}^2(\mathbb{D})\otimes\mathcal{L}$ (see Subsection 2.2),  $M_{\zeta_1}=M_z\otimes I_{\mathcal{L}}$
and $\mathbf{M}_\zeta'$ has the form $(I_{H_{\mathcal{E}}^2(\mathbb{D})}\otimes T_1,I_{H_{\mathcal{E}}^2(\mathbb{D})}\otimes T_2,\cdots)$ for a sequence $(T_1,T_2,\cdots)$  of operators on $\mathcal{L}$ jointly unitarily equivalent to the tuple of coordinate multiplication operators on $H^2(\mathbb{D}^\infty_2)$. Since
$\widetilde{V_1}$ doubly commutes with $\mathbf{M}_\zeta'$,
it follows from
Lemma \ref{M_zeta is irre} and Lemma \ref{redu in tensor} that $\widetilde{V_1}=S\otimes I_{\mathcal{L}}$ for some isometry
$S$ on $H_{\mathcal{E}}^2(\mathbb{D})$.
Therefore,
$$M_z\otimes I_{\mathcal{L}}=M_{\zeta_1}=\widetilde{V_1}^{k_1}=S^{k_1}\otimes I_{\mathcal{L}},$$ forcing $M_z=S^{k_1}$.
In particular, $S$ commutes with $M_z$. Then there is  a $\mathcal{B}(\mathcal{E})$-valued inner function $\theta_1$ in single variable $z\in\mathbb{D}$, such that $S=M_{\theta_1}$ (see \cite[pp. 200-201]{SNFBK} for instance), which gives
$$\widetilde{V_1}=M_{\theta_1}\otimes I_{\mathcal{L}}=M_{\widetilde{\theta_1}},$$ where
$\widetilde{\theta_1}(\zeta)=\theta_1(\zeta_1)\ (\zeta\in\mathbb{D}^\infty_2)$.
Similarly, for each $n\geq2$,
$\widetilde{V_n}=M_{\widetilde{\theta_n}}$, where
$\widetilde{\theta_n}(\zeta)=\theta_n(\zeta_n)\ (\zeta\in\mathbb{D}^\infty_2)$ for some inner function $\theta_n\in H^\infty_{\mathcal{B}(\mathcal{E})}(\mathbb{D})$. The proof is complete.
\qed
\vskip2mm

Combining Proposition
\ref{power of DC} with Corollary \ref{application}, we have the following corollary (Corollary \ref{application 2}).
\begin{cor}
  Suppose $\mathbf{T}\in\mathcal{DC}$. Then there exists a sequence
  $\{B_n\}_{n\in\mathbb{N}}$ of finite Blaschke products, such that
  $\{\prod_{i=1}^{n}B_i(T_i)\}_{n\in\mathbb{N}}$ converges in the strong operator topology.
\end{cor}

\vskip2mm

Now we will establish an analytic model for sequences in class
$\mathcal{DC}$ with a decomposition of quasi-Beurling type.

Suppose that $\mathbf{T}\in\mathcal{DC}(\mathcal{H})$ is of Beurling type, and $\mathbf{V}\in\mathcal{DP}(\mathcal{K})$ is the minimal regular isometric dilation of $\mathbf{T}$.  By Corollary \ref{collection of cor} (1), $\mathbf{V}$ is jointly unitarily equivalent to the tuple $\mathbf{M}_\zeta$ of coordinate multiplication operators on a vector-valued Hardy space $H_{\mathfrak{D}_{\mathbf{V}^*}}^2(\mathbb{D}_2^\infty)$.
We  claim that the map
\begin{equation}\label{map can be extended}
 D_{\mathbf{V}^*}x\mapsto D_{\mathbf{T}^*}x,\quad x\in\mathcal{H}
\end{equation}
can be extended to a unitary operator  from $\mathfrak{D}_{\mathbf{V}^*}$ onto
$\mathfrak{D}_{\mathbf{T}^*}$. By Lemma \ref{dilation of defect op},
it remains to prove that $D_{\mathbf{V}^*}\mathcal{H}$ is dense in $\mathfrak{D}_{\mathbf{V}^*}$.
Assume that $x\in\mathfrak{D}_{\mathbf{V}^*}$ is orthogonal to $D_{\mathbf{V}^*}\mathcal{H}$. Then $x$ is orthogonal to $\mathcal{H}$ and
$\mathrm{Ran}\ V_n$ for all $n\in\mathbb{N}$. In particular, $x$ is orthogonal to $[\mathcal{H}]_\mathbf{V}=\mathcal{K}$, forcing $x=0$. This proves the claim.
Therefore, the above unitary operator  from $\mathfrak{D}_{\mathbf{V}^*}$ onto
$\mathfrak{D}_{\mathbf{T}^*}$ naturally induced a unitary operator $U_\mathbf{T}$ from $\mathcal{K}$ onto $H_{\mathfrak{D}_{\mathbf{T}^*}}^2(\mathbb{D}_2^\infty)$.
 It is easy to see that $\mathcal{Q}_\mathbf{T}=U_\mathbf{T}\mathcal{H}$ is an analytic model for the sequence $\mathbf{T}$. We call $\mathcal{Q}_\mathbf{T}$ the
\textit{canonical analytic model}  for  $\mathbf{T}$.

As a consequence, for a sequence $\mathbf{T}\in\mathcal{DC}(\mathcal{H})$ which is of quasi-Beurling type,  one can find some $\lambda=(\lambda_1,\lambda_2,\cdots)\in\mathbb{D}^\infty$ so that
$\Phi_\lambda(\mathbf{T})$ is of Beurling type, and then
$(\mathcal{Q}_{\Phi_\lambda(\mathbf{T})},(\varphi_{\lambda_1}\cdot I,\varphi_{\lambda_2}\cdot I,\cdots))$ is an analytic model for the sequence $\mathbf{T}$.

Suppose that $\{\mathcal{E}_\gamma\}_\gamma$ is a family of Hilbert spaces, and for each index $\gamma$,
$\Theta_\gamma$ is a sequence in $H^\infty_{\mathcal{B}(\mathcal{E}_\gamma)}(\mathbb{D})$. In a natural way, we can define the direct sum $\bigoplus_\gamma\Theta_\gamma$ of  $\{\Theta_\gamma\}_\gamma$, which is then a sequence in $H^\infty_{\mathcal{B}(\bigoplus_\gamma\mathcal{E}_\gamma)}(\mathbb{D})$.
Now let $\mathbf{T}=\bigoplus_\gamma\mathbf{T}_\gamma$ be direct sum of $\mathcal{DC}$-sequences  of quasi-Beurling type. For each index $\gamma$, we have previously established an analytic model $(\mathcal{Q}_\gamma,\Theta_\gamma)$ for $\mathbf{T}_\gamma$. Put $\mathcal{Q}=\bigoplus_\gamma\mathcal{Q}_\gamma$ and $\Theta=\bigoplus_\gamma\Theta_\gamma$. Then
$(\mathcal{Q},\Theta)$ is an analytic model
for $\mathbf{T}$.

In fact, we have
\begin{thm}
  Suppose $\mathbf{T}\in\mathcal{DC}$. Then $\mathbf{T}$ has a decomposition
of quasi-Beurling type if and only if $\mathbf{T}$ has an analytic model $(\mathcal{Q},\Theta)$ so that all $\theta_n(z)\ (n\in\mathbb{N},z\in\mathbb{D})$ are simultaneously diagonalizable
with respect to some orthonormal basis of the underlying space $\mathcal{E}$ of $(\mathcal{Q},\Theta)$.
\end{thm}
\begin{proof}
The necessity follows from the construction in previous paragraphs.
  Now assume that $\mathbf{T}\in\mathcal{DC}(\mathcal{H})$ has an analytic model $(\mathcal{Q},\Theta)$ so that all $\theta_n(z)\ (n\in\mathbb{N},z\in\mathbb{D})$ are simultaneously diagonalizable
with respect to some orthonormal basis $\{e_i\}_{i\in\Lambda}$ of the underlying space $\mathcal{E}$ of $(\mathcal{Q},\Theta)$. We will prove that $\mathbf{T}$ has a decomposition
of quasi-Beurling type.
By the assumption, we have
$$\theta_n=\sum_{i\in\Lambda} \eta_{n i}\cdot e_ i\widehat{\otimes} e_ i,\quad n\in\mathbb{N},$$
where each $\eta_{n i}$ is an $H^\infty(\mathbb{D})$-inner function, $e_ i\widehat{\otimes} e_ i$ denotes the $1$-rank projection $$e_ i\widehat{\otimes} e_ i(x)=\langle x,e_ i\rangle e_ i,\quad x\in\mathcal{E}.$$
Then for each $i\in\Lambda$,
$H_{\mathbb{C}e_ i}^2(\mathbb{D}^\infty_2)$, as a subspace of $H_\mathcal{E}^2(\mathbb{D}_2^\infty)$, is joint reducing for $\mathbf{M}_\Theta$, and the restriction of $\mathbf{M}_\Theta$ on $H_{\mathbb{C}e_ i}^2(\mathbb{D}^\infty_2)$ is jointly unitarily equivalent to
the sequence
$(M_{\widetilde{\eta_{1 i}}},M_{\widetilde{\eta_{2 i}}},\cdots)$ of multiplication operators on
the Hardy space $H^2(\mathbb{D}_2^\infty)$, where $\widetilde{\eta_{n i}}(\zeta)=\eta_{n i}(\zeta_n)\ (n\in\mathbb{N},\zeta\in\mathbb{D}^\infty_2)$.
It follows from Proposition \ref{sequence of multiplication operators} that
the sequence $\mathbf{M}_\Theta$  has a decomposition
of quasi-Beurling type, and therefore so does  the minimal isometric dilation $\mathbf{V}$ of $\mathbf{T}$ by the definition of the analytic model. Then Theorem \ref{direct sums of quasi-BT} together with  Corollary \ref{collection of cor} (3) implies that $\mathbf{T}$ also has a decomposition
of quasi-Beurling type.
\end{proof}

\vskip2mm
To conclude this section, we give a characterization of the
canonical analytic model $\mathcal{Q}_{\mathbf{T}}$ for a sequence $\mathbf{T}\in\mathcal{DC}$ which is of Beurling type. One approach, inspired by the single $C_{.0}$-contraction case, is  utilizing
the characteristic functions of contractions.

Recall that the characteristic function $\theta_T$ of a contraction $T\in\mathcal{B}(\mathcal{H})$ is defined by
$$\theta_T(z)=[-T+zD_{T^*}(1-zT^*)^{-1}D_T]|_{\mathfrak{D}_T},\quad z\in\mathbb{D},$$
which is a $\mathcal{B}(\mathfrak{D}_T,\mathfrak{D}_{T^*})$-valued inner function (see \cite{SNFBK}). Then the multiplication operator $M_{\theta_T}$ is an isometry from $H_{\mathfrak{D}_T}^2(\mathbb{D})$ to $H_{\mathfrak{D}_{T^*}}^2(\mathbb{D})$, and thus
$$M_{\theta_T}^*H_{\mathfrak{D}_{T^*}}^2(\mathbb{D})
=H_{\mathfrak{D}_T}^2(\mathbb{D}).$$
Suppose that in addition $T\in C_{.0}$ and $V\in\mathcal{B}(\mathcal{K})$ is the minimal isometric dilation of $T$. Since $V$ is pure, $\mathcal{K}$ has a decomposition as $\mathcal{K}=\bigoplus_{k=0}^\infty V^k\mathfrak{D}_{V^*}$, and therefore can be identified with the vector-valued Hardy space $H_{\mathfrak{D}_{V^*}}^2(\mathbb{D})$.   Similar to (\ref{map can be extended}), the map
$$D_{V^*}x\mapsto D_{T^*}x,\quad x\in\mathcal{H}$$
can be extended to a  unitary operator from $\mathfrak{D}_{V^*}$ onto
$\mathfrak{D}_{T^*}$, and then induces a unitary operator $U_T$ from $\mathcal{K}$ onto $H_{\mathfrak{D}_{T^*}}^2(\mathbb{D})$. It is easy to see that
$U_T$ is actually an extension of the isometry
\begin{eqnarray*}V:\quad \mathcal{H}& \rightarrow &
  H_{\mathfrak{D}_{T^*}}^2(\mathbb{D}),\\
  x & \mapsto & \sum_{k=1}^\infty z^k\cdot D_{T^*}T^{*k}x.\end{eqnarray*}
 %given by Sz.-Nagy and Foias' functional model for $C_{.0}$-contractions.
%Below we will show that $\mathcal{J}=W\mathcal{H}=U_T\mathcal{H}$ is completely determined by
%the characteristic function $\theta_T$ of $T$.
%For $x\in\mathcal{H}$, $y\in\mathfrak{D}_{T^*}$ and $a\in\mathbb{D}$, $$U_Tx=U_T(\sum_{k=1}^{\infty}V^kD_{V^*}V^{*k}x)=\sum_{k=1}^{\infty}z^kD_{T^*}T^{*k}x,$$
%%=D_{T^*}(1-zT^*)^{-1}x
%and then
%\begin{equation*}
%  \begin{split}
%\langle U_Tx,K_a\cdot y \rangle & =\left\langle \sum_{k=1}^{\infty}z^kD_{T^*}T^{*k}x,K_a\cdot y \right\rangle \\
%     &= \sum_{k=1}^{\infty}\langle z^kD_{T^*}T^{*k}x,K_a\cdot y \rangle \\
%       & =\sum_{k=1}^{\infty}a^k\langle D_{T^*}T^{*k}x, y \rangle \\
%       &= \left\langle x,\sum_{k=1}^{\infty}(\bar{a}T)^kD_{T^*}y \right\rangle\\
%       & =\langle x,(1-\bar{a}T)^{-1}D_{T^*}y \rangle,
%  \end{split}
% \end{equation*} where $K_a(z)=\frac{1}{1-\bar{a}z}$.
% This forces $P_{\mathcal{H}}U_T^*(K_a\cdot y)=(1-\bar{a}T)^{-1}D_{T^*}y$,
%here $P_M$ denotes the orthogonal projection from $\mathcal{K}$ onto a closed subspace $M$.
It follows from \cite[pp. 244-245]{SNFBK} that for $a,b\in\mathbb{D}$,
\begin{equation}\label{theta_T(b)theta_T(a)*}
  I_{\mathfrak{D}_{T^*}}-\theta_T(b)\theta_T(a)^*
=(1-\bar{a}b)[D_{T^*}(1-bT^*)^{-1}(1-\bar{a}T)^{-1}D_{T^*}]|_{\mathfrak{D}_{T^*}}.
\end{equation}
%Then
%for $x,y\in\mathfrak{D}_{T^*}$ and $a,b\in\mathbb{D}$,
%\begin{equation*}
%  \begin{split}
%     \langle P_{\mathcal{H}}U_T^*K_a\cdot x,P_{\mathcal{H}}U_T^*K_b\cdot y \rangle & =\langle (1-\bar{a}T)^{-1}D_{T^*}x,(1-\bar{b}T)^{-1}D_{T^*}y \rangle \\
%       & =\langle D_{T^*}(1-bT^*)^{-1}(1-\bar{a}T)^{-1}D_{T^*}x,y \rangle \\
%       & =\frac{1}{1-\bar{a}b}\langle (I_{\mathfrak{D}_{T^*}}-\theta_T(b)\theta_T(a)^*x,y \rangle \\
%      & =\langle K_a,K_b\rangle(\langle x,y\rangle- \langle \theta_T(a)^*x,\theta_T(b)^*y \rangle) \\
%      & =\langle K_a\cdot x,K_b\cdot y\rangle- \langle M_{\theta_T}^*K_a\cdot x,M_{\theta_T}^*K_b\cdot y \rangle \\
%      & =\langle (I-M_{\theta_T}M_{\theta_T}^*)K_a\cdot x,K_b\cdot y \rangle.
%  \end{split}
% \end{equation*}
%Since the set $\{K_a\cdot x:x\in \mathfrak{D}_{T^*},a\in\mathbb{D}\}$ is complete in $H_{\mathfrak{D}_{T^*}}^2(\mathbb{D})$,
This gives  \begin{equation}\label{projection identity}
 U_TP_{\mathcal{H}}U_T^*=VV^*=I-M_{\theta_T}M_{\theta_T}^*
\end{equation}
(see \cite{BES} or \cite{BNS});
that is, $$H_{\mathfrak{D}_{T^*}}^2(\mathbb{D})\ominus U_T\mathcal{H}
=M_{\theta_T}M_{\theta_T}^*H_{\mathfrak{D}_{T^*}}^2(\mathbb{D}).$$

%\begin{lem} \label{density lemma}
%If $\mathbf{T}$ is a sequence in class $\mathcal{DC}(\mathcal{H})$, and $\mathbf{V}$ is the minimal regular isometric dilation of $\mathbf{T}$, then $D_{\mathbf{V}^*}\mathcal{H}$ is dense in $\mathfrak{D}_{\mathbf{V}^*}$.
%\end{lem}

Let us return to the characterization of the
canonical analytic model $\mathcal{Q}_{\mathbf{T}}$ for the sequence $\mathbf{T}=(T_1,T_2,\cdots)$. Now consider the following spaces
$$M_{\theta_{T_n}}M_{\theta_{T_n}}^*
H_{\mathfrak{D}_{\mathbf{T}^*}}^2(\mathbb{D}),\quad n\in\mathbb{N}.$$
By (\ref{theta_T(b)theta_T(a)*}), for $a,b\in\mathbb{D}$,
\begin{equation*}
 \begin{split}
    \theta_{T_1}(b)\theta_{T_1}(a)^*D_{\mathbf{T}^*}x & = D_{\mathbf{T}'^*}\theta_{T_1}(b)\theta_{T_1}(a)^*D_{T_1^*}x\\
      & = D_{\mathbf{T}^*}x-(1-\bar{a}b)[D_{\mathbf{T}^*}(1-bT^*)^{-1}(1-\bar{a}T)^{-1}D_{T_1^*}^2]x\\
      & =
D_{\mathbf{T}^*}[1-(1-\bar{a}b)(1-bT_1^*)^{-1}(1-\bar{a}T_1)^{-1}D_{T_1^*}^2]x,
 \end{split}
\end{equation*} where $\mathbf{T}'=(T_2,T_3,\cdots)$.
Similarly,
for any $n\in\mathbb{N}$ and $a,b\in\mathbb{D}$,
\begin{equation} \label{reducing for theta_T(b)theta_T(a)*}
    \theta_{T_n}(b)\theta_{T_n}(a)^*D_{\mathbf{T}^*}x  =
D_{\mathbf{T}^*}[1-(1-\bar{a}b)(1-bT_n^*)^{-1}(1-\bar{a}T_n)^{-1}D_{T_n^*}^2]x,
\end{equation}
 and then by Lemma \ref{operator-valued in one variable} (1), $H_{\mathfrak{D}_{\mathbf{T}^*}}^2(\mathbb{D})$ is reducing for
$M_{\theta_{T_n}}M_{\theta_{T_n}}^*$, and
\begin{equation} \label{range of M_theta_T*}
 M_{\theta_{T_n}}^*H_{\mathfrak{D}_{\mathbf{T}^*}}^2(\mathbb{D})
=H_{\mathcal{F}_n}^2(\mathbb{D}),
\end{equation}
 where $\mathcal{F}_n$ is the defect space of the sequence $(T_1^*,\cdots,T_{n-1}^*,T_n,T_{n+1}^*,\cdots)$.
 In particular, $$M_{\theta_{T_n}}M_{\theta_{T_n}}^*
H_{\mathfrak{D}_{\mathbf{T}^*}}^2(\mathbb{D})=\theta_{T_n}H_{\mathcal{F}_n}^2(\mathbb{D}),\quad  n\in\mathbb{N}$$ is a $M_z$-invariant subspace of
$H_{\mathfrak{D}_{\mathbf{T}^*}}^2(\mathbb{D})$.

 Put  $\widetilde{\theta_{T_n}}(\zeta)=\theta_{T_n}(\zeta_n)\ (n\in\mathbb{N},\zeta\in\mathbb{D}^\infty_2)$.
 Then $M_{\widetilde{\theta_{T_n}}}M_{\widetilde{\theta_{T_n}}}^*
H_{\mathfrak{D}_{\mathbf{T}^*}}^2(\mathbb{D}^\infty_2)$  $(n\in\mathbb{N})$ is of form
\begin{equation} \label{H_D_T*^2}
  \underbrace{H^2(\mathbb{D})\otimes\cdots\otimes H^2(\mathbb{D})}_{n-1\ \mathrm{times}}\otimes M_{\theta_{T_n}}M_{\theta_{T_n}}^*H_{\mathfrak{D}_{\mathbf{T}^*}}^2(\mathbb{D}) \otimes H^2(\mathbb{D})\otimes H^2(\mathbb{D})\otimes\cdots,
\end{equation}
forcing it to be a joint invariant subspace of $H_{\mathfrak{D}_{\mathbf{T}^*}}^2(\mathbb{D}^\infty_2)$
for the tuple of coordinate multiplication operators.

Our result  presented below looks somehow similar to the above single $C_{.0}$-contraction case, or the finite-tuple case considered in \cite{BNS}. However, instead of following the proof in \cite{BNS}, we will give an  original proof.

\begin{thm} \label{analytic model for BT} Let $\mathbf{T}=(T_1,T_2,\cdots)\in\mathcal{DC}$ be  of Beurling type, and $\mathcal{Q}_{\mathbf{T}}$  the canonical analytic model for $\mathbf{T}$. Then
$$H_{\mathfrak{D}_{\mathbf{T}^*}}^2(\mathbb{D}^\infty_2)\ominus\mathcal{Q}_{\mathbf{T}} = \bigvee_{n=1}^\infty M_{\widetilde{\theta_{T_n}}}M_{\widetilde{\theta_{T_n}}}^*
H_{\mathfrak{D}_{\mathbf{T}^*}}^2(\mathbb{D}^\infty_2),$$
where $\widetilde{\theta_{T_n}}(\zeta)=\theta_{T_n}(\zeta_n)\ (n\in\mathbb{N},\zeta\in\mathbb{D}^\infty_2)$.
\end{thm}

%Note that $M_{\widetilde{\theta_{T_n}}}M_{\widetilde{\theta_{T_n}}}^*
%H_{\mathfrak{D}_{\mathbf{T}^*}}^2(\mathbb{D}^\infty_2)$ is of form
%$$\underbrace{H^2(\mathbb{D})\otimes\cdots\otimes H^2(\mathbb{D})}_{n-1}\otimes M_{\theta_{T_n}}M_{\theta_{T_n}}^*H_{\mathfrak{D}_{\mathbf{T}^*}}^2(\mathbb{D}) \otimes H^2(\mathbb{D})\otimes\cdots,$$
%here $M_{\theta_{T_n}}M_{\theta_{T_n}}^*H_{\mathfrak{D}_{\mathbf{T}^*}}^2(\mathbb{D})$  is a $M_z$-invariant subspace of $H_{\mathfrak{D}_{\mathbf{T}^*}}^2(\mathbb{D})$
%because $H_{\mathfrak{D}_{\mathbf{T}^*}}^2(\mathbb{D})$ is reducing for
%$M_{\theta_{T_n}}M_{\theta_{T_n}}^*$.

\begin{proof} Let $\mathbf{V}\in\mathcal{DP}(\mathcal{K})$ be the minimal regular isometric dilation of $\mathbf{T}\in\mathcal{DC}(\mathcal{H})$, and put
$$\mathcal{H}_n=\bigvee_{\substack{\alpha\in\mathbb{Z}_+^{(\infty)} \\ \alpha_n=0 }}\mathbf{V}^\alpha\mathcal{H},\quad n\in\mathbb{N}.$$
By  Lemma \ref{intersection of H_n} (3), one has $\mathcal{Q}_{\mathbf{T}}=U_{\mathbf{T}}\mathcal{H}=\bigcap_{n=1}^\infty U_{\mathbf{T}}\mathcal{H}_n$. It thus suffices to prove that for each $n\in\mathbb{N}$,
$$H_{\mathfrak{D}_{\mathbf{T}^*}}^2(\mathbb{D}^\infty_2)\ominus U_{\mathbf{T}}\mathcal{H}_n=M_{\widetilde{\theta_{T_n}}}M_{\widetilde{\theta_{T_n}}}^*
H_{\mathfrak{D}_{\mathbf{T}^*}}^2(\mathbb{D}^\infty_2);$$
equivalently,
\begin{equation} \label{identity want to show}
  I-U_{\mathbf{T}}P_{\mathcal{H}_n}U_{\mathbf{T}}^*
=M_{\widetilde{\theta_{T_n}}}M_{\widetilde{\theta_{T_n}}}^*|_{
H_{\mathfrak{D}_{\mathbf{T}^*}}^2(\mathbb{D}^\infty_2)}.
\end{equation}
Assume $n=1$ without loss of generality, and
put $\mathbf{S}=(S_1,S_2,\cdots)=P_{\mathcal{H}_1}\mathbf{V}|_{\mathcal{H}_1}$, the compression of the sequence $\mathbf{V}$ on the subspace $\mathcal{H}_1$.  Rewrite
$T=T_1$ and $S=S_1$ for simplicity.
One
can define the unitary operators $U_{\mathbf{S}}:\mathcal{K}\rightarrow H_{\mathfrak{D}_{\mathbf{S}^*}}^2(\mathbb{D}^\infty_2)$ and $U_S:\mathcal{K}\rightarrow H_{\mathfrak{D}_{S^*}}^2(\mathbb{D})$ as done previously in this section.
%Also, it follows from (\ref{projection identity}) that
%$$I-U_{\mathbf{T}}P_{\mathcal{H}_1}U_{\mathbf{T}}^*
%=I-U_{\mathbf{T}}(I-U_S^*M_{\theta_S}M_{\theta_S}^*U_S)U_{\mathbf{T}}^*
%=U_{\mathbf{T}}U_S^*M_{\theta_S}M_{\theta_S}^*U_SU_{\mathbf{T}}^*.$$
Since (\ref{projection identity}) remain valid for $S$, (\ref{reducing for theta_T(b)theta_T(a)*}) and (\ref{range of M_theta_T*}) remain valid for  $\mathbf{S}$, it follows that
\begin{equation}\label{P_H_1}
 P_{\mathcal{H}_1}
 =I-U_S^*M_{\theta_S}M_{\theta_S}^*U_S,
\end{equation}
$H_{\mathfrak{D}_{\mathbf{S}^*}}^2(\mathbb{D})$ is reducing for
$M_{\theta_{S}}M_{\theta_{S}}^*$ and
\begin{equation} \label{range of M_theta_S*}
 M_{\theta_{S}}^*H_{\mathfrak{D}_{\mathbf{S}^*}}^2(\mathbb{D})
=H_{\mathcal{F}}^2(\mathbb{D}),
\end{equation}
where $\mathcal{F}$ is the defect space of the sequence
$(S,S_2^*,S_3^*,\cdots)$.

We now claim that
\begin{equation}\label{two orthogonal projections}
U_S^*M_{\theta_S}M_{\theta_S}^*U_S=U_\mathbf{S}^*(M_{\theta_S}M_{\theta_S}^*
|_{H_{\mathfrak{D}_{\mathbf{S}^*}}^2(\mathbb{D})}\otimes I_{\mathcal{L}})U_{\mathbf{S}},
\end{equation}
where
$$\mathcal{L}=\overline{\mathrm{span}\{\zeta^\alpha:\alpha=(\alpha_1,\alpha_2,\cdots)\in\mathbb{Z}_+^{(\infty)}\ \mathrm{with} \ \alpha_1=0\}}.$$
%and hence \begin{equation*}
%            \begin{split}
%               I-U_{\mathbf{T}}P_{\mathcal{H}_1}U_{\mathbf{T}}^*
% & =U_{\mathbf{T}}U_\mathbf{S}^*(M_{\theta_S}M_{\theta_S}^*
%|_{H_{\mathfrak{D}_{\mathbf{S}^*}}^2(\mathbb{D})}\otimes I_{\mathcal{L}})U_\mathbf{S}U_{\mathbf{T}}^* \\
%                 & =(I_{H^2(\mathbb{D})}\otimes\Pi\otimes I_{\mathcal{L}})(M_{\theta_S}M_{\theta_S}^*
%|_{H_{\mathfrak{D}_{\mathbf{S}^*}}^2(\mathbb{D})}\otimes I_{\mathcal{L}})
%(I_{H^2(\mathbb{D})}\otimes\Pi^*\otimes I_{\mathcal{L}}) \\
%                 & =(M_{\Pi\cdot\theta_S}M_{\Pi\cdot\theta_S}^*
%|_{H_{\mathfrak{D}_{\mathbf{T}^*}}^2(\mathbb{D})})\otimes I_{\mathcal{L}},
%            \end{split}
%          \end{equation*}
Since operators on both sides of (\ref{two orthogonal projections}) are orthogonal projections, the claim is equivalent to
\begin{equation} \label{equivalent form of claim}
U_S^*M_{\theta_S}M_{\theta_S}^*H_{\mathfrak{D}_{S^*}}^2(\mathbb{D})
=U_{\mathbf{S}}^*(M_{\theta_S}M_{\theta_S}^*
|_{H_{\mathfrak{D}_{\mathbf{S}^*}}^2(\mathbb{D})}\otimes I_{\mathcal{L}})(H_{\mathfrak{D}_{\mathbf{S}^*}}^2(\mathbb{D})\otimes\mathcal{L}).
\end{equation}
The left side of (\ref{equivalent form of claim}) is
$$U_S^*M_{\theta_S}H_{\mathfrak{D}_S}^2(\mathbb{D}),$$
and by (\ref{range of M_theta_S*}) the right side of (\ref{equivalent form of claim}) is
$$U_{\mathbf{S}}^*(M_{\theta_S}H_{\mathcal{F}}^2
(\mathbb{D})\otimes\mathcal{L}).$$
A calculation gives that for any $\alpha=(\alpha_1,\alpha_2,\cdots)\in\mathbb{Z}_+^{(\infty)}\ \mathrm{with} \ \alpha_1=0$ and $f\in H_{\mathfrak{D}_{\mathbf{S}^*}}^2
(\mathbb{D})$, $$U_SU_{\mathbf{S}}^*(f\otimes\zeta^\alpha)=(I_{H^2(\mathbb{D})}\otimes \mathbf{S}^\alpha)f.$$
Since $S$ doubly commutes with $S_n\ (n\geq2)$, for such $\alpha$ and $g\in H_{\mathcal{F}}^2
(\mathbb{D})$, the above identity gives
\begin{equation} \label{U_SU_S*}
 U_SU_{\mathbf{S}}^*(\theta_Sg\otimes\zeta^\alpha)=(I_{H^2(\mathbb{D})}\otimes \mathbf{S}^\alpha)\theta_Sg=M_{\theta_S}(I_{H^2(\mathbb{D})}\otimes \mathbf{S}^\alpha)g.
\end{equation}
Also by Lemma \ref{intersection of H_n} (1),
$\mathcal{H}_1$ is reducing for $V_n\ (n\geq2)$.
Then $\mathbf{S}'=(S_2,S_3,\cdots)$ is a $\mathcal{DP}$-sequence of Beurling type, and so is
$\mathbf{S}'|_{\mathfrak{D}_S}$, the restriction of $\mathbf{S}'$ on the joint reducing subspace $\mathfrak{D}_S$. This together with Corollary \ref{collection of cor} (1) gives
\begin{equation} \label{decomposition of D_s}
\mathfrak{D}_S=\bigoplus_{\alpha\in\mathbb{Z}_+^{(\infty)}}
\mathbf{S}'^\alpha\mathfrak{D}_{(\mathbf{S}'|_{\mathfrak{D}_S})^*}
=\bigoplus_{\alpha\in\mathbb{Z}_+^{(\infty)}}
\mathbf{S}'^\alpha \overline{D_S\mathfrak{D}_{\mathbf{S}'^*}}
=\bigoplus_{\alpha\in\mathbb{Z}_+^{(\infty)}}
\mathbf{S}'^\alpha\mathcal{F}.
\end{equation}
Therefore, we have
\begin{equation*}
  \begin{split}
   U_SU_{\mathbf{S}}^*(M_{\theta_S}H_{\mathcal{F}}^2
(\mathbb{D})\otimes\mathcal{L})
   & = \bigoplus_{\alpha_1=0}U_SU_{\mathbf{S}}^* (M_{\theta_S}H_{\mathcal{F}}^2
(\mathbb{D})\otimes\mathbb{C}\zeta^\alpha) \\
%& =\bigoplus_{\alpha_1=0}(I_{H^2(\mathbb{D})}\otimes \mathbf{S}^\alpha)M_{\theta_S}H_{\mathcal{F}}^2(\mathbb{D}) \\
       & =\bigoplus_{\alpha\in\mathbb{Z}_+^{(\infty)}}M_{\theta_S}(I_{H^2(\mathbb{D})}\otimes \mathbf{S}'^\alpha)H_{\mathcal{F}}^2(\mathbb{D})\\
       & =\bigoplus_{\alpha\in\mathbb{Z}_+^{(\infty)}}M_{\theta_S}H_{\mathbf{S}'^\alpha\mathcal{F}}^2(\mathbb{D}) \\ &=M_{\theta_S}(\bigoplus_{\alpha\in\mathbb{Z}_+^{(\infty)}}H_{\mathbf{S}'^\alpha\mathcal{F}}^2(\mathbb{D})) \\ &=M_{\theta_S}H_{\mathfrak{D}_S}^2(\mathbb{D}),
  \end{split}
\end{equation*}
where the second identity follows from (\ref{U_SU_S*}) and the last identity follows from (\ref{decomposition of D_s}).
This proves the claim.

The  operator $U_{\mathbf{T}}U_\mathbf{S}^*$ has the form $I_{H^2(\mathbb{D})}\otimes\Pi\otimes I_{\mathcal{L}}$
with respect to the representation
$$H_{\mathfrak{D}_{\mathbf{S}^*}}^2(\mathbb{D}^\infty_2)
=H^2(\mathbb{D})\otimes\mathfrak{D}_{\mathbf{S}^*}\otimes\mathcal{L},$$ where $\Pi:\mathfrak{D}_{\mathbf{S}^*}\rightarrow\mathfrak{D}_{\mathbf{T}^*}$ is a unitary operator satisfying \begin{equation} \label{PiD_S*=D_T*}
  \Pi(D_{\mathbf{S}^*}x)=D_{\mathbf{T}^*}x,\quad x\in\mathcal{H}.
\end{equation}
Let $\Pi\cdot\theta_S$ denotes the operator-valued function defined by $z\mapsto\Pi\theta_S(z)$ $(z\in\mathbb{D})$. Then
 $$M_{\Pi\cdot\theta_S}^*
|_{H_{\mathfrak{D}_{\mathbf{T}^*}}^2(\mathbb{D})}
=M_{\theta_S}^*
|_{H_{\mathfrak{D}_{\mathbf{S}^*}}^2(\mathbb{D})}
(I_{H^2(\mathbb{D})}\otimes\Pi^*)$$
since two operators coincide on the set $\{K_a\cdot x:a\in\mathbb{D},x\in\mathfrak{D}_{\mathbf{T}^*}\}$,
and thus
\begin{equation} \label{Pi dot theta}
\begin{split}
M_{\Pi\cdot\theta_S}M_{\Pi\cdot\theta_S}^*
|_{H_{\mathfrak{D}_{\mathbf{T}^*}}^2(\mathbb{D})}
& = (M_{\Pi\cdot\theta_S}^*
|_{H_{\mathfrak{D}_{\mathbf{T}^*}}^2(\mathbb{D})})^*
M_{\Pi\cdot\theta_S}^*
|_{H_{\mathfrak{D}_{\mathbf{T}^*}}^2(\mathbb{D})} \\
%& = (M_{\theta_S}^*
%|_{H_{\mathfrak{D}_{\mathbf{S}^*}}^2(\mathbb{D})}
%(I_{H^2(\mathbb{D})}\otimes\Pi^*))^*M_{\theta_S}^*
%|_{H_{\mathfrak{D}_{\mathbf{S}^*}}^2(\mathbb{D})}
%(I_{H^2(\mathbb{D})}\otimes\Pi^*) \\
&=  (I_{H^2(\mathbb{D})}\otimes\Pi)
(M_{\theta_S}^*
|_{H_{\mathfrak{D}_{\mathbf{S}^*}}^2(\mathbb{D})})^*
M_{\theta_S}^*
|_{H_{\mathfrak{D}_{\mathbf{S}^*}}^2(\mathbb{D})}
(I_{H^2(\mathbb{D})}\otimes\Pi^*) \\
& =(I_{H^2(\mathbb{D})}\otimes\Pi)M_{\theta_S}M_{\theta_S}^*
|_{H_{\mathfrak{D}_{\mathbf{S}^*}}^2(\mathbb{D})}
(I_{H^2(\mathbb{D})}\otimes\Pi^*).
  \end{split}
\end{equation}
Therefore, we have
\begin{equation*}
            \begin{split}
               I-U_{\mathbf{T}}P_{\mathcal{H}_1}U_{\mathbf{T}}^*
 & =I-U_{\mathbf{T}}(I-U_S^*M_{\theta_S}M_{\theta_S}^*U_S)U_{\mathbf{T}}^* \\ & =U_{\mathbf{T}}U_S^*M_{\theta_S}M_{\theta_S}^*U_SU_{\mathbf{T}}^* \\
 & =U_{\mathbf{T}}U_\mathbf{S}^*(M_{\theta_S}M_{\theta_S}^*
|_{H_{\mathfrak{D}_{\mathbf{S}^*}}^2(\mathbb{D})}\otimes I_{\mathcal{L}})U_\mathbf{S}U_{\mathbf{T}}^* \\
                 & =(I_{H^2(\mathbb{D})}\otimes\Pi\otimes I_{\mathcal{L}})(M_{\theta_S}M_{\theta_S}^*
|_{H_{\mathfrak{D}_{\mathbf{S}^*}}^2(\mathbb{D})}\otimes I_{\mathcal{L}})
(I_{H^2(\mathbb{D})}\otimes\Pi^*\otimes I_{\mathcal{L}}) \\
& =((I_{H^2(\mathbb{D})}\otimes\Pi)M_{\theta_S}M_{\theta_S}^*
|_{H_{\mathfrak{D}_{\mathbf{S}^*}}^2(\mathbb{D})}
(I_{H^2(\mathbb{D})}\otimes\Pi^*))\otimes I_{\mathcal{L}}\\
                 & =(M_{\Pi\cdot\theta_S}M_{\Pi\cdot\theta_S}^*
|_{H_{\mathfrak{D}_{\mathbf{T}^*}}^2(\mathbb{D})})\otimes I_{\mathcal{L}},
            \end{split}
          \end{equation*}
where the first identity follows from (\ref{P_H_1}),
the third identity follows from (\ref{two orthogonal projections}), and
the fifth identity follows from (\ref{Pi dot theta}).
On the other hand, (\ref{H_D_T*^2}) gives $$M_{\widetilde{\theta_T}}M_{\widetilde{\theta_T}}^*|_{
H_{\mathfrak{D}_{\mathbf{T}^*}}^2(\mathbb{D}^\infty_2)}=M_{\theta_T}M_{\theta_T}^*
|_{H_{\mathfrak{D}_{\mathbf{T}^*}}^2(\mathbb{D})}\otimes I_{\mathcal{L}},$$
which reduces (\ref{identity want to show})  to
$$M_{\Pi\cdot\theta_S}M_{\Pi\cdot\theta_S}^*
|_{H_{\mathfrak{D}_{\mathbf{T}^*}}^2(\mathbb{D})}=M_{\theta_T}M_{\theta_T}^*
|_{H_{\mathfrak{D}_{\mathbf{T}^*}}^2(\mathbb{D})}.$$
By Lemma \ref{operator-valued in one variable} (2), it remains to prove that for any fixed $a,b\in\mathbb{D}$ and any fixed $x\in\mathcal{H}\subseteq\mathcal{H}_1$,
$$(\Pi\cdot\theta_S)(b)(\Pi\cdot\theta_S)(a)^*
D_{\mathbf{T}^*}x=\Pi\theta_S(b)\theta_S(a)^*\Pi^*D_{\mathbf{T}^*}x
=\theta_T(b)\theta_T(a)^*D_{\mathbf{T}^*}x.$$
Note that  Lemma \ref{intersection of H_n} (2) implies that $\mathcal{H}$ is reducing for $S$ and $T=S|_\mathcal{H}$. By (\ref{reducing for theta_T(b)theta_T(a)*}), we have $$\theta_S(b)\theta_S(a)^*D_{\mathbf{S}^*}x=D_{\mathbf{S}^*}y$$ and $$\theta_T(b)\theta_T(a)^*D_{\mathbf{T}^*}x=D_{\mathbf{T}^*}y,$$ where
%$$y  = [1-(1-\bar{a}b)(1-bS^*)^{-1}(1-\bar{a}S)^{-1}D_{S^*}^2]x
%      =[1-(1-\bar{a}b)(1-bT^*)^{-1}(1-\bar{a}T)^{-1}D_{T^*}^2]x\in\mathcal{H}.$$
\begin{equation*}
  \begin{split}
    y & = [1-(1-\bar{a}b)(1-bS^*)^{-1}(1-\bar{a}S)^{-1}D_{S^*}^2]x\\
     & =[1-(1-\bar{a}b)(1-bT^*)^{-1}(1-\bar{a}T)^{-1}D_{T^*}^2]x\in\mathcal{H}.
  \end{split}
\end{equation*}
It follows from (\ref{PiD_S*=D_T*}) that
$$\Pi\theta_S(b)\theta_S(a)^*\Pi^*D_{\mathbf{T}^*}x
=\Pi\theta_S(b)\theta_S(a)^*D_{\mathbf{S}^*}x
=\Pi D_{\mathbf{S}^*}y= D_{\mathbf{T}^*}y.$$
This completes the proof.
%$$\theta_T(b)\theta_T(a)^*D_{\mathbf{T}^*}x=
%D_{\mathbf{T}^*}[1-(1-\bar{a}b)(1-bT^*)^{-1}(1-\bar{a}T)^{-1}D_{T^*}^2]x$$
%and
%\begin{equation*}
%  \begin{split}
%     \Pi\theta_S(b)\theta_S(a)^*\Pi^*D_{\mathbf{T}^*}x & = \Pi\theta_S(b)\theta_S(a)^*D_{\mathbf{S}^*}x \\
%     & =
%\Pi D_{\mathbf{S}^*}y \\
%& = D_{\mathbf{T}^*}y \\
%       & = D_{\mathbf{T}^*}[1-(1-\bar{a}b)(1-bT^*)^{-1}(1-\bar{a}T)^{-1}D_{T^*}^2]x\\
%       & = \theta_T(b)\theta_T(a)^*D_{\mathbf{T}^*}x.
%  \end{split}
%\end{equation*}
%The proof is complete.
\end{proof}

Now we are ready to refine the representation (\ref{representation of direct sum of quasi-BT})  in Subsection 1.2.
%Following \cite{SNFBK}, say that two operator-valued functions
%$\theta\in H_{\mathcal{B}(\mathcal{F},\mathcal{E})}^\infty(\mathbb{D})$ and $\vartheta\in H_{\mathcal{B}(\mathcal{H},\mathcal{G})}^\infty(\mathbb{D})$
%coincide if there exist unitary operators $U:\mathcal{F}\rightarrow\mathcal{H}$ and $V:\mathcal{E}\rightarrow\mathcal{G}$ such that
%$$\vartheta(z)=V\theta(z)U^*,\quad z\in\mathbb{D}.$$
%Suppose that $\mathbf{T}\in\mathcal{DC}$ with a decomposition $\mathbf{T}=\bigoplus_\gamma\mathbf{T}_\gamma$ of quasi-Beurling type. Then
%there corresponds a point $\lambda_\gamma\in\mathbb{D}^\infty$ to each index $\gamma$, such that
Note that for each index $\gamma$, $\mathbf{S}_\gamma=(S_{\gamma1},S_{\gamma2},\cdots)=\Phi_{\lambda_\gamma}(\mathbf{T}_\gamma)$ is of Beurling type. Without loss of generality, we may assume that $\mathcal{E}_\gamma=\mathfrak{D}_{\mathbf{S}_\gamma^*}$ and $\mathcal{Q}_\gamma$ is
the canonical analytic model for $\mathbf{S}_\gamma$
in (\ref{representation of direct sum of quasi-BT}).
%Let $\widetilde{\mathcal{Q}_\gamma}$  be
 %Therefore, $\mathbf{T}_\gamma$ is jointly unitarily equivalent to $P_{\widetilde{\mathcal{Q}_\gamma}}\Phi_{\lambda_\gamma}(\mathbf{M}_\zeta)|_{\widetilde{\mathcal{Q}_\gamma}}$, and thus
% $$\mathbf{T}\cong\bigoplus_\gamma P_{\widetilde{\mathcal{Q}_\gamma}}\Phi_{\lambda_\gamma}(\mathbf{M}_\zeta)|_{\widetilde{\mathcal{Q}_\gamma}}. $$
%where each $\widetilde{\mathcal{Q}_\gamma}\subseteq H^2_{\mathfrak{D}_{\mathbf{S}_\gamma^*}}(\mathbb{D}_2^\infty)$  satisfies that any nonzero element $x\in\mathfrak{D}_{\mathbf{S}_\gamma^*}$ is not orthogonal to $\widetilde{\mathcal{Q}_\gamma}$.
%Define a sequence $(\theta_1,\theta_2,\cdots)$ of operator-valued inner functions on $\mathbb{D}$ as $\theta_n(z)=\oplus_\gamma\theta_{S_{\gamma n}}(z)$, where $\theta_{S_{\gamma n}}$ is the characteristic function of $S_{\gamma n}$ ($n\in\mathbb{N}$). It is routine to check that
%\begin{itemize}
%  \item [(1)] $\theta_n$ coincide with the characteristic function $\theta_{T_n}$ of $T_n$ for each $n\in\mathbb{N}$;
%  \item [(2)] $H^2_{\mathcal{E}_\gamma}(\mathbb{D})$ is joint reducing for the sequence $(M_{\theta_1}M_{\theta_1}^*,M_{\theta_2}M_{\theta_2}^*,
%      \cdots)$ for each index $\gamma$.
%\end{itemize}
Put $\lambda_\gamma=(\lambda_{\gamma 1},\lambda_{\gamma 2},\cdots)$ and
let $\theta_{\gamma n}$, $\vartheta_{\gamma n}$ ($n\in\mathbb{N}$) denote the characteristic functions of $T_{\gamma n}$ and $S_{\gamma n}$, respectively. Then for each $\gamma$, $\vartheta_{\gamma n}$ coincides with $\theta_{\gamma n}\circ
\varphi_{\lambda_{\gamma n}}$ (see \cite[pp. 246-247]{SNFBK}), and
it follows from Theorem \ref{analytic model for BT} that $$H^2_{\mathcal{E}_\gamma}(\mathbb{D}_2^\infty)\ominus\mathcal{Q}_\gamma
      =\bigvee_{n=1}^\infty M_{\widetilde{\vartheta_{\gamma n}}}M_{\widetilde{\vartheta_{\gamma n}}}^*H^2_{\mathcal{E}_\gamma}(\mathbb{D}_2^\infty),$$
     where $\widetilde{\vartheta_{\gamma n}}(\zeta)=\vartheta_{\gamma n}(\zeta_n)\ (n\in\mathbb{N},\zeta\in\mathbb{D}^\infty_2)$.

\section{Doubly commuting submodules and quotient modules of $H_{\mathcal{E}}^2(\mathbb{D}^\infty_2)$}

For submodules and quotient modules of the  Hardy module, we are interested in the module actions on them; that is, the restrictions of  the tuple  of coordinate multiplication operators on submodules and the compressions of  the tuple  of coordinate multiplication operators on quotient modules. In this section, we mainly consider such restrictions and  compressions  that are doubly commuting.

Recall that
   a submodule $\mathcal{S}$ of $H_{\mathcal{E}}^2(\mathbb{D}_2^\infty)$ is said to be doubly commuting if the restriction $$(M_{\zeta_1}|_\mathcal{S},M_{\zeta_2}|_\mathcal{S},\cdots)$$ of $\mathbf{M}_\zeta$ on $\mathcal{S}$ is doubly commuting.

\begin{thm} \label{restriction on doubly commuting submodule}
Let $\mathbf{M}_\zeta$ be the tuple of coordinate multiplication operators on the vector-valued Hardy space $H_{\mathcal{E}}^2(\mathbb{D}_2^\infty)$.
 Then the  restriction of $\mathbf{M}_\zeta$ on a doubly commuting submodule of
 $H_{\mathcal{E}}^2(\mathbb{D}_2^\infty)$ is  of Beurling type.
\end{thm}

Before giving the proof, we introduce the notion of homogeneous components of functions in $H_{\mathcal{E}}^2(\mathbb{D}_2^\infty)$.
Suppose $F\in H_{\mathcal{E}}^2(\mathbb{D}_2^\infty)$, and let $F=\sum_{\alpha\in\mathbb{Z}_+^{(\infty)}}\zeta^\alpha\cdot x_\alpha$ be the power series expansion of $F$. The sum $\sum_{|\alpha|=k}\zeta^\alpha\cdot x_\alpha$  ($k=0,1,2,\cdots$) is called the \textit{$k$-th homogeneous component} of $F$,
 where $|\alpha|=\alpha_1+\alpha_2+\cdots$. It is clear that $\|F\|^2$ is equal to the
 quadratic sum of norms of all homogeneous components of $F$.

\begin{proof}
  Let $\mathcal{S}$ be a doubly commuting  quotient module of
 $H_{\mathcal{E}}^2(\mathbb{D}_2^\infty)$, and set $\mathbf{R}=\mathbf{M}_\zeta|_{\mathcal{S}}$, the restriction of $\mathbf{M}_\zeta$ on $\mathcal{S}$. Assume conversely that $\mathbf{R}$ is not of Beurling type to reach a contradiction. Then  by Corollary \ref{collection of cor} (3), there exists an $\mathbf{R}$-joint reducing subspace $\widetilde{\mathcal{S}}$ of $\mathcal{S}$ that is
 orthogonal to the defect space $\mathfrak{D}_{\mathbf{R}^*}$ of $\mathbf{R}^*$;
 equivalently, $\mathfrak{D}_{\widetilde{\mathbf{R}}^*}=\{0\}$, where $\widetilde{\mathbf{R}}=\mathbf{R}|_{\widetilde{\mathcal{S}}}$.
 So without loss of generality, we may assume $\mathfrak{D}_{\mathbf{R}^*}=\{0\}$.

 Set $k_0$ to be the minimal non-negative integer among those $k's$ that let nonzero $k$-th homogeneous components appear in some functions belonging to $\mathcal{S}$.
Now choose a function $F\in\mathcal{S}$ so that the norm of  the $k_0$-th homogeneous component of $F$ is $1$.
Since $\mathfrak{D}_{\mathbf{R}^*}=\{0\}$, one has
$$\mathcal{S}=\bigvee_{n=1}^\infty\mathrm{Ran}\ R_n=\bigvee_{n=1}^\infty\zeta_n\mathcal{S},$$
and then there exist $n\in\mathbb{N}$ and $n$ functions $F_1,\cdots,F_n\in\mathcal{S}$, such that
$$\|F-\sum_{i=1}^{n}\zeta_i F_i\|<1.$$
This implies that the norm of  the $k_0$-th homogeneous component of the function $F-\sum_{i=1}^{n}\zeta_i F_i$ is less than $1$. However, it is clear that
the $k_0$-th homogeneous component of $\sum_{i=1}^{n}\zeta_i F_i$ is $0$,
which contradicts with the choice of $F$. The proof is complete.
\end{proof}

We also recall that a function $\Psi\in H_{\mathcal{B}(\mathcal{F},\mathcal{E})}^\infty(\mathbb{D}_2^\infty)$
is inner if the multiplication operator $M_\Psi$ induced by $\Psi$ is an isometry.
  It is clear that for any  inner function $\Psi\in H_{\mathcal{B}(\mathcal{F},\mathcal{E})}^\infty(\mathbb{D}_2^\infty)$,
 $\Psi H_{\mathcal{F}}^2(\mathbb{D}_2^\infty)$ is a doubly commuting submodule of
 $H_{\mathcal{E}}^2(\mathbb{D}_2^\infty)$.
The following (Theorem \ref{doubly commuting submodule}) is a Beurling-Lax type Theorem for the vector-valued Hardy space in infinitely many variables.

\begin{cor} \label{doubly commuting Hardy submodule} Let   $\mathcal{S}$ be a  submodule of the vector-valued Hardy module
 $H_{\mathcal{E}}^2(\mathbb{D}_2^\infty)$. Then $\mathcal{S}$ is doubly commuting if and only if there exist a Hilbert space $\mathcal{F}$ and an inner function $\Psi\in H_{\mathcal{B}(\mathcal{F},\mathcal{E})}^\infty(\mathbb{D}_2^\infty)$, such that
 \begin{equation}\label{rep of doubly commuting Hardy submodule}
  \mathcal{S}=\Psi H_{\mathcal{F}}^2(\mathbb{D}_2^\infty).
 \end{equation}
\end{cor}
\begin{proof} Let  $\mathbf{M}_\zeta$ be the tuple of coordinate multiplication operators on $H_{\mathcal{E}}^2(\mathbb{D}_2^\infty)$, and $\mathbf{R}$ be the  restriction of $\mathbf{M}_\zeta$ on $\mathcal{S}$. It follows from Theorem \ref{restriction on doubly commuting submodule} and Corollary \ref{BT} that
 $\mathbf{R}$  is jointly unitarily equivalent to the tuple $\mathbf{M}_\xi$ of coordinate multiplication operators on a vector-valued Hardy space $H_{\mathcal{F}}^2(\mathbb{D}_2^\infty)$. This naturally induces an isometry $V:H_{\mathcal{F}}^2(\mathbb{D}_2^\infty)\rightarrow H_{\mathcal{E}}^2(\mathbb{D}_2^\infty)$ satisfying $\mathrm{Ran}\ V=\mathcal{S}$ and
 $$VM_{\xi_n}=R_nV=M_{\zeta_n}V,\quad n\in\mathbb{N}.$$
 Then by Proposition \ref{operatoe interwine CMO}, $V=M_{\Psi}$ for some operator-valued inner function $\Psi\in H_{\mathcal{B}(\mathcal{F},\mathcal{E})}^\infty(\mathbb{D}_2^\infty)$,
 where $M_{\Psi}$ is the multiplication operator induced by $\Psi$. This completes the proof.
 \end{proof}

In particular, we reprove the known result that every  doubly commuting submodule of
 $H^2(\mathbb{D}_2^\infty)$ is principle (see \cite{O}).

 \begin{cor} Let  $\mathcal{S}$ be a  nonzero submodule of
 $H^2(\mathbb{D}_2^\infty)$. Then the followings are equivalent.
 \begin{itemize}
   \item [(1)] $\mathcal{S}$ is doubly commuting;
   \item [(2)] $\mathcal{S}$ is generated by a single inner function in $H^\infty(\mathbb{D}_2^\infty)$;
   \item [(3)]  $\mathcal{S}$ as a $\mathcal{P}_\infty$-module is unitarily equivalent to $H^2(\mathbb{D}_2^\infty)$.
 \end{itemize}
 \end{cor}

We say that  two $\mathcal{P}_\infty$-modules $(\mathcal{H},\mathbf{T})$, $(\mathcal{K},\mathbf{S})$ are  \textit{unitarily equivalent} if  $\mathbf{T}$, $\mathbf{S}$ are jointly unitarily equivalent, and the unitary operator $U:\mathcal{H}\rightarrow\mathcal{K}$ interwinding $\mathbf{T}$ and $\mathbf{S}$ is called  a \textit{unitary module map}.

The classification of doubly commuting Hardy submodules
up to unitary equivalence of $\mathcal{P}_\infty$-modules
is trivial, since by Corollary \ref{doubly commuting Hardy submodule}, it is  completely determined by the dimension of
 $\mathcal{F}$ in
 the representation (\ref{rep of doubly commuting Hardy submodule}).

\vskip2mm

Now we turn to the situation of quotient modules.
Recall that a quotient module $\mathcal{Q}$ of $H_{\mathcal{E}}^2(\mathbb{D}_2^\infty)$ is said to be doubly commuting if the compression $$(P_\mathcal{Q}M_{\zeta_1}|_\mathcal{Q},P_\mathcal{Q}M_{\zeta_2}|_\mathcal{Q},\cdots)$$ of $\mathbf{M}_\zeta$ on $\mathcal{Q}$ is doubly commuting.
  %If $(\mathcal{Q},\Theta)$ is an analytic model for a sequence $\mathbf{T}\in\mathcal{DC}$ with the underlying space $\mathcal{E}$, then $\mathcal{Q}$ is a doubly commuting quotient module of $H_{\mathcal{E}}^2(\mathbb{D}_2^\infty)$.

\begin{thm} \label{compression  on  doubly commuting  quotient module}
Let   $\mathbf{M}_\zeta$ be the tuple of coordinate multiplication operators on the vector-valued Hardy module $H_{\mathcal{E}}^2(\mathbb{D}_2^\infty)$.
 Then  the compression of $\mathbf{M}_\zeta$ on a doubly commuting  quotient module of
 $H_{\mathcal{E}}^2(\mathbb{D}_2^\infty)$ is of Beurling type.
\end{thm}
\begin{proof} Let $\mathcal{Q}$ be a doubly commuting  quotient module of
 $H_{\mathcal{E}}^2(\mathbb{D}_2^\infty)$, and set $\mathbf{C}=P_{\mathcal{Q}}\mathbf{M}_\zeta|_{\mathcal{Q}}$, the compression of $\mathbf{M}_\zeta$ on $\mathcal{Q}$. Since both two sequences $\mathbf{M}_\zeta$ and $\mathbf{C}$ are doubly commuting,
 for $\alpha,\beta\in\mathbb{Z}_+^{(\infty)}$ satisfying
$\alpha\wedge\beta=(0,0,\cdots)$, we have
$$P_{\mathcal{Q}}\mathbf{M}_\zeta^{*\alpha}\mathbf{M}_\zeta^\beta|_{\mathcal{Q}}
=P_{\mathcal{Q}}\mathbf{M}_\zeta^\beta\mathbf{M}_\zeta^{*\alpha}|_{\mathcal{Q}}
=P_{\mathcal{Q}}\mathbf{M}_\zeta^\beta P_{\mathcal{Q}}
\mathbf{M}_\zeta^{*\alpha}|_{\mathcal{Q}}
=\mathbf{C}^\beta\mathbf{C}^{*\alpha}=\mathbf{C}^{*\alpha}\mathbf{C}^\beta,$$
where $\alpha\wedge\beta=(\min\{\alpha_1,\beta_1\},\min\{\alpha_2,\beta_2\},\cdots)$.
That is to say, $\mathbf{M}_\zeta$ is a regular isometric dilation of $\mathbf{C}$.
So the minimal regular isometric dilation of $\mathbf{C}$ is the restriction of $\mathbf{M}_\zeta$ on the subspace $[\mathcal{Q}]_{\mathbf{M}_\zeta}$.
It follows from Lemma \ref{intersection of H_n} (1) that $[\mathcal{Q}]_{\mathbf{M}_\zeta}$ is joint reducing for $\mathbf{M}_\zeta$,
and then by Lemma \ref{redu in tensor}, one has $$[\mathcal{Q}]_{\mathbf{M}_\zeta}=H_{\mathcal{E}_0}^2(\mathbb{D}_2^\infty)$$ for some closed subspace $\mathcal{E}_0$ of $\mathcal{E}$. Therefore, Corollary \ref{BT} gives that $\mathbf{C}$ is of Beurling type.
\end{proof}

Follow the notations in the proof of  Theorem \ref{compression  on  doubly commuting  quotient module}.
Now we can use the characterization of the canonical analytic model $\mathcal{Q}_{\mathbf{C}}$ for $\mathbf{C}$ (see Theorem \ref{analytic model for BT}) to study the structure of the doubly commuting quotient module $\mathcal{Q}$ since
as $\mathcal{P}_\infty$-modules, $\mathcal{Q}$ and $\mathcal{Q}_{\mathbf{C}}$  are unitarily equivalent. The unitary module map is of form $I_{H^2(\mathbb{D}_2^\infty)}\otimes U$, where
$U:\mathcal{E}_0\rightarrow\mathfrak{D}_{\mathbf{C}^*}$ is given as in (\ref{map can be extended}). Then there exists a sequence of quotient modules $\{\mathcal{J}_n\}_{n\in\mathbb{N}}$ of $H_{\mathcal{E}_0}^2(\mathbb{D})$, such that
\begin{equation} \label{doubly commuting quotient module}
 \mathcal{Q}=\bigcap_{n=1}^\infty\underbrace{H^2(\mathbb{D})\otimes\cdots\otimes H^2(\mathbb{D})}_{n-1\ \mathrm{times}} \otimes \mathcal{J}_n \otimes H^2(\mathbb{D})\otimes H^2(\mathbb{D})\otimes\cdots.
\end{equation}

Note  that the joint unitary equivalence is a equivalence relation on class $\mathcal{DC}$.
 We can establish a one-to-one corresponding between the equivalence classes of doubly commuting quotient Hardy modules and  the equivalence classes of $\mathcal{DC}$-sequences  of Beurling type as illustrated below.
 $$\left\{ \begin{gathered}
\mathrm{equivalence}\ \mathrm{classes}\ \mathrm{of}\\
\mathrm{doubly}\ \mathrm{commuting}\\
\mathrm{quotient}\ \mathrm{Hardy} \ \mathrm{modules}
\end{gathered}\right\} \autorightleftharpoons{\small module actions}{\footnotesize canonical  analytic models}
\left\{ \begin{gathered}
\mathrm{equivalence}\ \mathrm{classes}\ \mathrm{of}\\
\mathcal{DC}\mathrm{-sequences}\\
\mathrm{of}\ \mathrm{Beurling}\ \mathrm{type}
\end{gathered}\right\}  $$
In another word,
the classification of doubly commuting Hardy quotient modules  is equivalent to the classification of $\mathcal{DC}$-sequences of Beurling type.

\vskip2mm

Finally, we will consider the particular case $\mathcal{E}=\mathbb{C}$. By (\ref{doubly commuting quotient module}), we have
\begin{equation}\label{Q can be viewed as infinite tensor product}
 \mathcal{Q}=\bigcap_{n=1}^\infty\mathcal{J}_1 \otimes\cdots\otimes \mathcal{J}_n \otimes H^2(\mathbb{D})\otimes H^2(\mathbb{D})\otimes\cdots.
\end{equation}
It follows that $\mathcal{Q}=\mathcal{J}_1\otimes\mathcal{Q}'$
for some closed subspace $\mathcal{Q}'$ of
$$\mathcal{L}=\overline{\mathrm{span}\{\zeta^\alpha:\alpha=(\alpha_1,\alpha_2,\cdots)\in\mathbb{Z}_+^{(\infty)}\ \mathrm{with} \ \alpha_1=0\}},$$
and thus $P_\mathcal{Q}M_{\zeta_1}|_\mathcal{Q}=P_{\mathcal{J}_1}M_z|_{\mathcal{J}_1}\otimes I_{\mathcal{Q}'}$. Similarly, $P_\mathcal{Q}M_{\zeta_n}|_\mathcal{Q}$ is the tensor product of $P_{\mathcal{J}_n}M_z|_{\mathcal{J}_n}$ and an identity operator for each $n\in\mathbb{N}$.
Recall that a model space is a  of $H^2(\mathbb{D})$, and
the compression of the Hardy shift $M_z$ on a  model space is called a Jordan block.
Then for any $n\in\mathbb{N}$ so that $\mathcal{J}_n\neq H^2(\mathbb{D})$, the compression $P_\mathcal{Q}M_{\zeta_n}|_\mathcal{Q}$ is the tensor product of a Jordan block and an identity operator.
So the compression of the tuple $\mathbf{M}_\zeta$ on $\mathcal{Q}$ can be considered as a Jordan block in the infinite-variable
setting.

Recall that $C_0$ is the class of thoes completely nonunitary contractions $T$ for which there exists a nonzero function $f\in H^\infty(\mathbb{D})$ such that $f(T)=0$ \cite{SNFBK}.
The following    application to  operator theory is motivated by \cite[Proposition 4.1]{DY2}.
\begin{cor} Suppose $\mathbf{T}\in\mathcal{DC}$.
If there exists a cyclic vector $x$ for  $\mathbf{T}$ such that
$T_n^*x=\overline{\lambda_n}x\ (n\in\mathbb{N})$ for some $\lambda\in\mathbb{D}^\infty$, then for each $n\in\mathbb{N}$, $T_n$ is either a $C_0$-contraction or a pure isometry.
\end{cor}
\begin{proof} Assume $\mathbf{T}\in\mathcal{DC}(\mathcal{H})$
  Since for each $n\in\mathbb{N}$, $T_n$ is a $C_0$-contraction (pure isometry) if and only if $\varphi_{\lambda_n}(T_n)$ is a $C_0$-contraction (pure isometry), we may assume $\lambda=\mathbf{0}$ without loss of generality.

  Put $\widetilde{\mathcal{H}}=\bigvee_{\mu\in\mathbb{D}_2^\infty}
\mathfrak{D}_{\Phi_\mu(\mathbf{T})^*}$. Then by Corollary \ref{collection of cor} (3), $\widetilde{\mathcal{H}}$ is joint reducing for $\mathbf{T}$. Therefore,
$$\widetilde{\mathcal{H}}=[\widetilde{\mathcal{H}}]_{\mathbf{T}}
\supseteq[\mathfrak{D}_{\mathbf{T}^*}]_{\mathbf{T}}\supseteq[x]_{\mathbf{T}}=\mathcal{H},$$
forcing $\widetilde{\mathcal{H}}=\mathcal{H}$.
This together with Corollary \ref{BT} imples that $\mathbf{T}$ is of Beurling type. Since $\dim\mathfrak{D}_{\mathbf{T}^*}=1$ (otherwise $x$ is not cyclic for $\mathbf{T}$), the canonical analytic model $\mathcal{Q}_{\mathbf{T}}$ for $\mathbf{T}$ can be viewed as a doubly commuting quotient module $\mathcal{Q}$ of $H^2(\mathbb{D}_2^\infty)$, and then $\mathbf{T}$ is jointly unitarily equivalent to
the compression $P_\mathcal{Q}\mathbf{M}_\zeta|_\mathcal{Q}$ of  $\mathbf{M}_\zeta$ on $\mathcal{Q}$, where
$\mathbf{M}_\zeta$ is the tuple
 of coordinate multiplication operators on $H^2(\mathbb{D}_2^\infty)$. This completes the proof.
\end{proof}

Also from (\ref{Q can be viewed as infinite tensor product}), it seems plausible to view every doubly commuting quotient module of $H^2(\mathbb{D}_2^\infty)$ as the  tensor product of infinitely many model spaces or $H^2(\mathbb{D})$'s.  This can be realized after giving an appropriate definition for the infinite tensor product.

Let $\{M_n\}_{n\in\mathbb{N}}$ be a sequence of closed subspaces of $H^2(\mathbb{D})$, \textit{the tensor product of} $\{M_n\}_{n\in\mathbb{N}}$ in $H^2(\mathbb{D}_2^\infty)$, denoted by $\bigotimes_{n=1}^\infty M_n$, is definied to be the closed subspace of $H^2(\mathbb{D}_2^\infty)$ spanned by the functions in $H^2(\mathbb{D}_2^\infty)$ of form $\prod_{n=1}^\infty f_n(\zeta_n)$ (in pointwise convergence for $\zeta\in\mathbb{D}_2^\infty$)
with $f_n\in M_n$ for each $n\in\mathbb{N}$. Note that for  infinite tensor products of form $$M_1\otimes\cdots\otimes M_n\otimes H^2(\mathbb{D})\otimes H^2(\mathbb{D})\otimes\cdots,$$
this new definition coincides with the original one that the
$(M_{\zeta_{n+1}},M_{\zeta_{n+2}},\cdots)$-joint invariant subspace
 generated by $M_1\otimes\cdots\otimes M_n$ (see Subsection 2.2).

\begin{cor} \label{doubly commuting quotient of H^2(D^infty)'} Every  doubly commuting quotient module of
 $H^2(\mathbb{D}_2^\infty)$ is the tensor product of some sequence of quotient modules of $H^2(\mathbb{D})$.
\end{cor}
\begin{proof} Let $\mathcal{Q}$ be a doubly commuting quotient module of
 $H^2(\mathbb{D}_2^\infty)$.  Then there exists a sequence
 $\{\mathcal{J}_n\}_{n\in\mathbb{N}}$ of  quotient modules of $H^2(\mathbb{D})$, such that
 $$\mathcal{Q}=\bigcap_{n=1}^\infty\mathcal{J}_1 \otimes\cdots\otimes \mathcal{J}_n \otimes H^2(\mathbb{D})\otimes H^2(\mathbb{D})\otimes\cdots.$$
 Now we will prove $\mathcal{Q}=\bigotimes_{n=1}^\infty \mathcal{J}_n$.

 For simplicity, rewrite $$\mathcal{Q}_n =\mathcal{J}_1 \otimes\cdots\otimes \mathcal{J}_n \otimes H^2(\mathbb{D})\otimes H^2(\mathbb{D})\otimes\cdots.$$
 The inclusion $\bigotimes_{n=1}^\infty \mathcal{J}_n\subseteq\bigcap_{n=1}^\infty \mathcal{Q}_n =\mathcal{Q}$ is trivial to see. For the reverse inclusion, note that
 the set $\{P_{\mathcal{Q}}\zeta^\alpha:\alpha\in\mathbb{Z}_+^{(\infty)}\}$ is complete in $\mathcal{Q}$, where $P_{\mathcal{Q}}$ is the orthogonal projection from  $H^2(\mathbb{D}_2^\infty)$ onto  $\mathcal{Q}$. It suffices to show that
 for any fixed $\alpha=(\alpha_1,\alpha_2,\cdots)\in\mathbb{Z}_+^{(\infty)}$, $P_{\mathcal{Q}}\zeta^\alpha$ belongs to the infinite tensor product $\bigotimes_{n=1}^\infty \mathcal{J}_n$. Let $P_{\mathcal{Q}_n }\ (n\in\mathbb{N})$ denote the orthogonal projection from  $H^2(\mathbb{D}_2^\infty)$ onto  $\mathcal{Q}_n $, and
 $P_{\mathcal{J}_n}$ denote the orthogonal projection from $H^2(\mathbb{D})$ onto $\mathcal{J}_n$.
 Then for each $n\in\mathbb{N}$,
 $$P_{\mathcal{Q}_n }=P_{\mathcal{J}_1}\otimes\cdots\otimes P_{\mathcal{J}_n}\otimes I_{H^2(\mathbb{D})}\otimes I_{H^2(\mathbb{D})}\otimes\cdots.$$
Taking $m\in\mathbb{N}$ so that $\alpha_{m+1}=\alpha_{m+2}=\cdots=0$ and setting $f_n=P_{\mathcal{J}_n}1\ (n\in\mathbb{N})$, we further have
  $$(P_{\mathcal{Q}_n }\zeta^\alpha)(\zeta)=(P_{\mathcal{J}_1}z^{\alpha_1})(\zeta_1)\cdots (P_{\mathcal{J}_m}z^{\alpha_m})(\zeta_m) f_{m+1}(\zeta_{m+1})\cdots f_n(\zeta_n)$$
  for any $n\geq m+1$. On the other hand,
  since $\{P_{\mathcal{Q}_n }\}_{n\in\mathbb{N}}$ converges to $P_{\mathcal{Q}}$ $(n\rightarrow\infty)$ in the strong operator topology, $P_{\mathcal{Q}_n }\zeta^\alpha$ converges to $P_{\mathcal{Q}}\zeta^\alpha$ $(n\rightarrow\infty)$ in $H^2(\mathbb{D}^\infty_2)$-norm.
  In particular, $P_{\mathcal{Q}_n}\zeta^\alpha$ converges pointwisely to $P_{\mathcal{Q}}\zeta^\alpha$ as $n\rightarrow\infty$, and then
$P_{\mathcal{Q}}\zeta^\alpha$ is of form $$(P_{\mathcal{J}_1}z^{\alpha_1})(\zeta_1)\cdots (P_{\mathcal{J}_m}z^{\alpha_m})(\zeta_m)\prod_{n=m+1}^{\infty} f_n(\zeta_n),\quad \zeta\in\mathbb{D}_2^\infty.$$
This gives $P_{\mathcal{Q}}\zeta^\alpha\in\bigotimes_{n=1}^\infty \mathcal{J}_n$,
and the proof is complete.

 %So it remains to prove $$\bigcap_{n=1}^\infty\mathcal{J}_1 \otimes\cdots\otimes \mathcal{J}_n \otimes H^2(\mathbb{D})\otimes\cdots=\bigotimes_{n=1}^\infty \mathcal{J}_n.$$
% For simplicity, rewrite $\mathcal{Q}_n =\mathcal{J}_1 \otimes\cdots\otimes \mathcal{J}_n \otimes H^2(\mathbb{D})\otimes\cdots$.
% The inclusion $\bigotimes_{n=1}^\infty \mathcal{J}_n\subseteq\bigcap_{n=1}^\infty \mathcal{Q}_n $ is trivial to see. For the reverse inclusion, take an arbitrary function $F\in\bigcap_{n=1}^\infty \mathcal{Q}_n $. We will show $F\in\bigotimes_{n=1}^\infty \mathcal{J}_n$.
%
% Let $P_{\mathcal{Q}}$, $P_{\mathcal{Q}_n }\ (n\in\mathbb{N})$ denote the orthogonal projection from  $H^2(\mathbb{D}_2^\infty)$ onto $\mathcal{Q}$ and $\mathcal{Q}_n $ respectively, and
% $P_{\mathcal{J}_n}$ denote the orthogonal projection from $H^2(\mathbb{D})$ onto $\mathcal{J}_n$. Then for each $n\in\mathbb{N}$,
% $$P_{\mathcal{Q}_n }=P_{\mathcal{J}_1}\otimes\cdots\otimes P_{\mathcal{J}_n}\otimes I_{H^2(\mathbb{D})}\otimes\cdots.$$
% Setting $f_n=P_{\mathcal{J}_n}1$, we furthermore have
%  $$P_{\mathcal{Q}_n }1(\zeta)=f_1(\zeta_1)\cdots f_n(\zeta_n),\quad n\in\mathbb{N}.$$
%  Since $\mathcal{Q}$ is nonzero, $P_{\mathcal{Q}}1$ is also nonzero, and thus
%  $$\|P_{\mathcal{Q}}1\|=\lim_{n\rightarrow\infty}\|P_{\mathcal{Q}_n }1\|
%  =\lim_{n\rightarrow\infty}\|f_1\|_{H^2(\mathbb{D})}\cdots\|f_n\|_{H^2(\mathbb{D})}.$$
\end{proof}
%\bigskip \noindent{\bf Acknowledgement}.
%The authors with to express their sincere gratitude to the anonymous referee,
%whose extremely careful reading of the manuscript led to many clarifications and
%improvements in the text. We are also grateful to Prof. Xiantao Wang for his help.

%------------------------------------------------------------------------------

\bibliographystyle{amsplain}

\end{document}